\documentclass[draft]{amsart}

\usepackage{amsmath,amssymb,amsthm,enumerate} 
\usepackage{amsfonts} 
\usepackage[dvipdfmx]{graphicx}
\usepackage{color}

\theoremstyle{plain}
\newtheorem{pretheo}{Theorem}[section]
\newtheorem{preassu}[pretheo]{Assumption}
\newtheorem{precoro}[pretheo]{Corollary}
\newtheorem{predefi}[pretheo]{Definition}
\newtheorem{preexam}[pretheo]{Example}
\newtheorem{prelemm}[pretheo]{Lemma}
\newtheorem{preprop}[pretheo]{Proposition}
\newtheorem{prerema}[pretheo]{Remark}

\newenvironment{theo}{\begin{pretheo}}{\end{pretheo}}

\newenvironment{coro}{\begin{precoro}}{\end{precoro}}
\newenvironment{defi}{\begin{predefi}}{\end{predefi}}

\newenvironment{lemm}{\begin{prelemm}}{\end{prelemm}}

\newenvironment{rema}{\begin{prerema}\rm}{\end{prerema}}

\DeclareMathOperator{\di}{div}

\DeclareMathOperator{\Di}{Div}

\newcommand{\intd}{\,d}

\newcommand{\pa}{\partial}

\newcommand{\wh}[1]{\widehat{#1}}
\newcommand{\wt}[1]{\widetilde{#1}}

\newcommand{\Hol}{{\rm Hol}\,}

\newcommand{\BBM}{\mathbb{M}}

\newcommand{\Ba}{\mathbf{a}}
\newcommand{\Bb}{\mathbf{b}}

\newcommand{\Bf}{\mathbf{f}}
\newcommand{\Bg}{\mathbf{g}}

\newcommand{\Bn}{\mathbf{n}}

\newcommand{\Bu}{\mathbf{u}}
\newcommand{\Bv}{\mathbf{v}}

\newcommand{\BC}{\mathbf{C}}
\newcommand{\BD}{\mathbf{D}}

\newcommand{\BF}{\mathbf{F}}
\newcommand{\BG}{\mathbf{G}}

\newcommand{\BI}{\mathbf{I}}

\newcommand{\BK}{\mathbf{K}}
\newcommand{\BL}{\mathbf{L}}
\newcommand{\BM}{\mathbf{M}}
\newcommand{\BN}{\mathbf{N}}

\newcommand{\BR}{\mathbf{R}}
\newcommand{\BS}{\mathbf{S}}
\newcommand{\BT}{\mathbf{T}}
\newcommand{\BU}{\mathbf{U}}

\newcommand{\CA}{\mathcal{A}}
\newcommand{\CB}{\mathcal{B}}

\newcommand{\CF}{\mathcal{F}}
\newcommand{\CG}{\mathcal{G}}
\newcommand{\CH}{\mathcal{H}}

\newcommand{\CL}{\mathcal{L}}
\newcommand{\CM}{\mathcal{M}}
\newcommand{\CN}{\mathcal{N}}

\newcommand{\CR}{\mathcal{R}}
\newcommand{\CS}{\mathcal{S}}
\newcommand{\CT}{\mathcal{T}}

\newcommand{\Fa}{\mathfrak{a}}
\newcommand{\Fb}{\mathfrak{b}}

\newcommand{\Fl}{\mathfrak{l}}
\newcommand{\Fm}{\mathfrak{m}}

\newcommand{\Fp}{\mathfrak{p}}
\newcommand{\Fq}{\mathfrak{q}}
\newcommand{\Fr}{\mathfrak{r}}

\newcommand{\FA}{\mathfrak{A}}
\newcommand{\FB}{\mathfrak{B}}

\newcommand{\FX}{\mathfrak{X}}

\newcommand{\SSS}{\mathsf{S}}
\newcommand{\SST}{\mathsf{T}}

\newcommand{\bdry}{{\BR_0^N}}

\newcommand{\uhs}{{\BR_+^N}}

\newcommand{\ws}{{\BR^N}}

\newcommand{\al}{\alpha}

\newcommand{\ga}{\gamma}
\newcommand{\de}{\delta}

\newcommand{\te}{\theta}
\newcommand{\ka}{\kappa}
\newcommand{\la}{\lambda}
\newcommand{\si}{\sigma}
\newcommand{\ph}{\varphi}
\newcommand{\om}{\omega}

\newcommand{\Ga}{\Gamma}
\newcommand{\De}{\Delta}

\newcommand{\La}{\Lambda}
\newcommand{\Si}{\Sigma}
\newcommand{\Om}{\Omega}

\numberwithin{equation}{section} 
\allowdisplaybreaks[4]

\begin{document}
\title[Compressible fluid model of Korteweg type]
{Compressible fluid model of Korteweg type with free boundary condition: model problem}

\author[Hirokazu Saito]{Hirokazu Saito}
\address{Department of Mathematics,
Faculty of Science and Engineering,
Waseda University, Okubo 3-4-1, Shinjuku-ku,
Tokyo 169-8555, Japan}
\email{hsaito@aoni.waseda.jp}



\thanks{}


\dedicatory{}

\begin{abstract}
The aim of this paper is to prove the existence of $\CR$-bounded solution operator families
for a resolvent problem on the upper half-space
arising from a compressible fluid model of Korteweg type with free boundary condition.
Such a compressible fluid model was derived by Dunn and Serrin (1985)
and studied by Kotschote (2008) as a boundary value problem with non-slip boundary condition.
\end{abstract}

\maketitle

\renewcommand{\thefootnote}{\arabic{footnote})}


\section{Introduction}
\subsection{Problem}
Let $\BR_+^N$ and $\BR_0^N$ be the upper half-space and its boundary for $N\geq 2$, respectively, 
i.e.
\begin{align*}
\BR_+^N&=\{x=(x',x_N) \mid x'=(x_1,\dots,x_{N-1})\in\BR^{N-1}, x_N>0\}, \\
\BR_0^N&=\{x=(x',x_N) \mid x'=(x_1,\dots,x_{N-1})\in\BR^{N-1}, x_N=0\}.
\end{align*}

This paper is concerned with the existence of $\CR$-bounded
solution operator families for the following resolvent problem
arising from a compressible fluid model of Korteweg type with free boundary condition:
\begin{equation}\label{resolvent}
\left\{\begin{aligned}
\la\rho + \rho_*\di\Bu&=d && \text{in $\BR_+^N$,} \\
\rho_*\la\Bu-\Di\{\mu_*\BD(\Bu)+(\nu_*-\mu_*)\di\Bu\BI-(\ga_*-\rho_*\ka_*\De)\rho\BI\}&=\Bf && \text{in $\uhs$,} \\
\{\mu_*\BD(\Bu)+(\nu_*-\mu_*)\di\Bu\BI-(\ga_*-\rho_*\ka_*\De)\rho\BI\}\Bn&=\Bg && \text{on $\BR_0^N$,} \\
\Bn\cdot\nabla\rho&= h && \text{on $\BR_0^N$.}
\end{aligned}\right.
\end{equation}

Here, $\la$ is the resolvent parameter varying in
\begin{equation*}
\Si_{\si,\de}=\{\la\in\BC \mid |\arg\la|<\pi-\si,|\la|>\de\}
\end{equation*}
for $\si\in(0,\pi/2)$ and $\de\geq0$;
$\rho=\rho(x)$ and $\Bu=\Bu(x)=(u_1(x),\dots,u_N(x))^\SST$ 
are the fluid density and the fluid velocity, respectively, that are unknown functions;
$\rho_*$, $\mu_*$, $\nu_*$, $\ga_*$, and $\ka_*$ are given constants such that
\begin{equation*}
\rho_*>0, \quad \mu_*>0, \quad \nu_*>0,
 \quad \ga_*\in(-\infty,\infty), \quad \ka_*>0,
\end{equation*}
while $d=d(x)$, $\Bf=\Bf(x)=(f_1(x),\dots, f_N(x))^\SST$,
$\Bg=\Bg(x)=(g_1,\dots,g_N(x))^\SST$, and $h=h(x)$ are given functions;
$\BD(\Bu)$ is the doubled strain tensor whose $(i,j)$-component is given by
$\pa_i u_j+\pa_j u_i$ with $\pa_j=\pa/\pa x_j$;
$\BI$ is the $N\times N$ identity matrix and $\Bn=(0,\dots,0,-1)^\SST$ is the unit outer normal to $\BR_0^N$;
$\Ba\cdot\Bb=\sum_{j=1}^N a_j b_j$
for $N$-vectors $\Ba=(a_1,\dots,a_N)^\SST$, $\Bb=(b_1,\dots,b_N)^\SST$.
Here and subsequently, we use the following notation for differentiations:
Let $u=u(x)$, $\Bv=(v_1(x),\dots,v_N(x))^\SST$, and $\BM=(M_{ij}(x))$ be
a scalar-, a vector-, and an $N\times N$ matrix-valued function 
defined on a domain of $\BR^N$, and then
\begin{align*}
&\nabla u = (\pa_1 u,\dots,\pa_N u)^\SST, \quad
\De u = \sum_{j=1}^N\pa_j^2 u, \quad
\De \Bv = (\De v_1,\dots , \De v_N)^\SST, \\
&\di\Bv = \sum_{j=1}^N\pa_j v_j, \quad
\nabla\Bv=\{\pa_j v_k \mid j,k=1,\dots,N\}, \\
&\nabla^2 \Bv = \{\pa_j\pa_k v_l \mid j,k,l=1,\dots,N\}, \quad
\Di\BM = \bigg(\sum_{j=1}^N \pa_j M_{1j},\dots,\sum_{j=1}^N \pa_j M_{Nj}\bigg)^\SST.
\end{align*}
It especially holds that $\Bn\cdot\nabla\rho=-\pa_N\rho$ and
\begin{align}\label{DivT}
&\Di \{\mu_*\BD(\Bu)+(\nu_*-\mu_*)\di\Bu\BI
-(\ga_*-\ka_*\De)\rho\BI\} \\
&=\mu_*\De\Bu+\nu_*\nabla\di\Bu-(\ga_*-\ka_*\De)\nabla\rho.
\notag
\end{align}

The motion of barotropic compressible fluids is governed by
\begin{equation}\label{170306_1}
\left\{\begin{aligned}
\pa_t\rho+\di(\rho\Bu)&=0 && \text{in $\Om$, $t>0$,} \\
\rho(\pa_t\Bu+\Bu\cdot\nabla\Bu) &= \Di(\BT-P(\rho)\BI) && \text{in $\Om$, $t>0$,} \\
(\rho,\Bu)|_{t=0} &=(\rho_0,\Bu_0) && \text{in $\Om$,}
\end{aligned}\right.
\end{equation}
where $\Om$ is a domain of $\BR^N$;
$P:[0,\infty)\to\BR$ is a given smooth function describing the pressure;
$\rho_0$ and $\Bu_0$ are given initial data.
Dunn and Serrin \cite{DS85} introduced the thermomechanics of interstitial working,
which tells us that, for a special material of Korteweg type,
the stress tensor $\BT$ is given by 
\begin{align}\label{SK}
&\BT=\BS(\Bu)+\BK(\rho), \quad 
\BS(\Bu) = \mu\BD(\Bu)+(\nu-\mu)\di\Bu\BI, \\
&\BK(\rho) = \frac{\ka}{2}\left(\De\rho^2-|\nabla\rho|^2\right)\BI -\ka\nabla\rho\otimes\nabla\rho.
\notag
\end{align}
Here $\mu$ and $\nu$ are viscosity coefficients, while $\ka$ is a capillary coefficient.

Let us recall brief history of mathematical analysis for
the compressible fluid model of Korteweg type
which means the system \eqref{170306_1} together with \eqref{SK}
in the present paper.

We start with the whole space case $\Om=\BR^N$.
Danchin and Desjardins \cite{DD01} proved in some critical space  
the existence and uniqueness of strong solutions global in time for 
initial data close enough to stable equilibria
and local in time for initial densities bounded away from zero.
For strong solutions, we also refer to a previous work, Hattori and Li \cite{HL96}, 
that requires higher regularity of initial data than Danchin and Desjardins \cite{DD01}.
Haspot \cite{Haspot11}, on the other hand, proved 
the existence of global weak solutions.

Concerning boundary value problems,
Kotschote \cite{Kotschote08} treated mainly bounded domains with the boundary condition:
\begin{equation}\label{DNbc}
\Bu=0 \quad \text{on $\Ga$,} \quad \Bn\cdot\nabla\rho=0 \quad \text{on $\Ga$,}
\end{equation}
where $\Ga$ is the boundary of $\Om$ and $\Bn$ is the unit outer normal to $\Ga$.
In \cite{Kotschote08}, he proved for isothermal cases
the existence and uniqueness of strong solutions local in time 
by means of the contraction mapping principle together with the maximal $L_p\text{-}L_q$ regularity $(p=q)$
of the linearized system,
where the initial data should satisfy some compatibility conditions and $\rho_0>0$ on $\overline{\Om}$.
For non-isothermal cases,
we also refer to his papers \cite{Kotschote10, Kotschote12, Kotschote14}
based on approaches similar to \cite{Kotschote08}.

We are interested in a free boundary value problem  
of the compressible fluid model of Korteweg type,
i.e. the domain $\Om$ is replaced by an unknown domain $\Om(t)$, depending on time $t$, with boundary $\Ga(t)$.
Then we have a further equation of time-evolution of $\Ga(t)$
such as $V_t=\Bu\cdot\Bn_t$, 
where $\Bn_t$ is the unit outer normal to $\Ga(t)$
and $V_t$ is the velocity of evolution of $\Ga(t)$ with respect to $\Bn_t$.
In addition, the boundary condition \eqref{DNbc} should be replaced by
the free boundary condition:
\begin{equation*}
(\BT-P(\rho)\BI)\Bn_t = -P_0\Bn_t \quad \text{on $\Ga(t)$,} \quad
\Bn_t\cdot\nabla\rho =0 \quad \text{on $\Ga(t)$,}
\end{equation*}
where $P_0$ is a given constant that is specified below. 
Thus the free boundary value problem is formulated as follows:
\begin{equation*}
\left\{\begin{aligned}
&\pa_t\rho + \di(\rho\Bu) =0 && \text{in $\Om(t)$, $t>0$}, \\
&\rho(\pa_t\Bu+\Bu\cdot\nabla\Bu) =\Di(\BS(\Bu)+\BK(\rho)-P(\rho)\BI) && \text{in $\Om(t)$, $t>0$,}  \\ 
&(\BS(\Bu)+\BK(\rho)-P(\rho)\BI)\Bn_t = -P_0\Bn_t && \text{on $\Ga(t)$, $t>0$,} \\
&\Bn_t\cdot\nabla\rho = 0 && \text{on $\Ga(t)$, $t>0$,} \\
&(\rho,\Bu)|_{t=0} = (\rho_0,\Bu_0)&& \text{in $\Om_0$},
\end{aligned}\right.
\end{equation*}
subject to $V_t=\Bu\cdot\Bn_t$ on $\Ga(t)$.
Here $\Om_0$ is the given initial domain of $\BR^N$.

Let $\rho_0(x)=\wt\rho_0(x)+\rho_\infty$ for some positive constant $\rho_\infty$.
Replacing $\rho$ by $\rho+\rho_\infty$ in the above system yields 
\begin{equation*}
\left\{\begin{aligned}
&\pa_t\rho + \di((\rho+\rho_\infty)\Bu) =0 && \text{in $\Om(t)$, $t>0$,} \\
&(\rho+\rho_\infty)(\pa_t\Bu+\Bu\cdot\nabla\Bu) \\
&\quad=\Di(\BS(\Bu)+\BK(\rho+\rho_\infty)-P(\rho+\rho_\infty)\BI) && \text{in $\Om(t)$, $t>0$,} \\ 
&(\BS(\Bu)+\BK(\rho+\rho_\infty)-P(\rho+\rho_\infty)\BI)\Bn_t 
= -P(\rho_\infty)\Bn_t && \text{on $\Ga(t)$, $t>0$,} \\
&\Bn_t\cdot\nabla\rho =0 && \text{on $\Ga(t)$, $t>0$,} \\
&(\rho,\Bu)|_{t=0} = (\wt\rho_0,\Bu_0) && \text{in $\Om_0$},
\end{aligned}\right.
\end{equation*}
where we have chosen $P_0=P(\rho_\infty)$. 
We transform this system to a problem on the initial domain $\Om_0$
by some transformation (e.g. Lagrangian transformation),
and thus we achieve the following linearized system on $\Om_0$:
\begin{equation}\label{eq:Om_0}
\left\{\begin{aligned}
&\pa_t\rho + \rho_\infty\di\Bu = d && \text{in $\Om_0$, $t>0$,} \\
&\rho_\infty\pa_t\Bu-\Di\{\mu\BD(\Bu)+(\nu-\mu)\di\Bu\BI  \\
&\quad -(P'(\rho_\infty)-\rho_\infty\ka\De)\rho\BI\} = \Bf && \text{in $\Om_0$, $t>0$,} \\
&\{\mu\BD(\Bu)+(\nu-\mu)\di\Bu\BI 
-(P'(\rho_\infty)-\rho_\infty\ka\De)\rho\BI\}\Bn_0 = \Bg && \text{on $\Ga_0$, $t>0$,} \\
&\Bn_0\cdot\nabla\rho = h && \text{on $\Ga_0$, $t>0$,} \\
&(\rho,\Bu)|_{t=0} =(\wt\rho_0,\Bu_0) && \text{in $\Om_0$,}
\end{aligned}\right.
\end{equation}
where $\Ga_0$ is the boundary of $\Om_0$ and $\Bn_0$ is the unit outer normal to $\Ga_0$.
Then the resolvent problem \eqref{resolvent} can be obtained by 
the Laplace transform applied to \eqref{eq:Om_0} for $\Om_0=\BR_+^N$ and $\Ga_0=\BR_0^N$.


Throughout this paper, the letter $C$ denotes generic constants and
$C_{a,b,c,\dots}$ means that the constant depends on the quantities $a,b,c,\dots$.
The values of constants $C$ and $C_{a,b,c,\dots}$ may change from line to line.

\subsection{Main results}
To state our main results precisely, 
we first introduce notation.

Let $\BN$ be the set of all natural numbers and $\BN_0=\BN\cup\{0\}$.
Let $1\leq q \leq \infty$ and $G$ be a domain of $\BR^N$.
Then $L_q(G)$ and $W_q^m(G)$, $m\in\BN$, denote the usual $\BK$-valued Lebesgue spaces on $G$
and the usual $\BK$-valued Sobolev spaces on $G$, respectively, where $\BK=\BR$ or $\BK=\BC$.
We set $W_q^0(G)=L_q(G)$ and denote the norm of $W_q^n(G)$, $n\in\BN_0$, by $\|\cdot\|_{W_q^n(G)}$.
In addition, $C_0^\infty(G)$ is the set of all functions of $C^\infty(G)$ whose supports are compact subsets of $G$,
while $(\Bu,\Bv)_G=\sum_{j=1}^N\int_G u_j(x)v_j(x) \intd x$
for $N$-vector functions $\Bu=(u_1(x),\dots,u_N(x))^\SST$, $\Bv=(v_1(x),\dots,v_N(x))^\SST$.

Let $X$, $Y$ be Banach spaces.
Then $\CL(X,Y)$ is the Banach space of all bounded linear operators from $X$ to $Y$,
and $\CL(X)$ is the abbreviation of $\CL(X,X)$.
For a subset $U$ of $\BC$,
$\Hol(U,\CL(X,Y))$ stands for the set of all $\CL(X,Y)$-valued analytic functions defined on $U$.


We next introduce the definition of the $\CR$-boundedness of operator families.
Let ${\rm sign}(a)$ be the sign function of $a$.

\begin{defi}[$\CR$-boundedness]\label{defi:R}
Let $X$ and $Y$ be Banach spaces,
and let $r_j(u)$ be the Rademacher functions on $[0,1]$, i.e.
\begin{equation*}
r_j(u) = {\rm sign}\sin(2^j\pi u) \quad (j\in\BN, 0\leq u\leq 1).
\end{equation*}

A family of operators $\CT\subset\CL(X,Y)$ is called $\CR$-bounded on $\CL(X,Y)$,
if for some $p\in [1,\infty)$
there exists a positive constant $C$ such that
the following assertion holds true:
For each $m\in\BN$, $\{T_j\}_{j=1}^m\subset \CT$, and $\{f_j\}_{j=1}^m\subset X$,
we have
\begin{equation*}
\left(\int_0^1\Big\|\sum_{j=1}^m r_j(u)T_j f_j\Big\|_Y^p\intd u\right)^{1/p}
\leq C\left(\int_0^1\Big\|\sum_{j=1}^m r_j(u)f_j\Big\|_X^p\intd u\right)^{1/p}.
\end{equation*}
The smallest such $C$ is called $\CR$-bound of $\CT$ on $\CL(X,Y)$
and denoted by $\CR_{\CL(X,Y)}(\CT)$.
\end{defi}

\begin{rema}\label{rema:Kahane}
It is known that $\CT$ is $\CR$-bounded for any
$p\in[1,\infty)$, provided that $\CT$ is $\CR$-bounded for some $p\in[1,\infty)$.
This fact follows from Kahane's inequality (cf. e.g. \cite[Theorem 2.4]{KW04}).
\end{rema}

We treat in this paper a rescaled version of \eqref{resolvent} as follows:
\begin{equation}\label{eq:rescale}
\left\{\begin{aligned}
\la\rho+\di\Bu&=d && \text{in $\BR_+^N$,} \\
\la\Bu-\mu_*\De\Bu-\nu_*\nabla\di\Bu+(\ga_*-\ka_*\De)\nabla\rho&=\Bf && \text{in $\BR_+^N$,} \\
\{\mu_*\BD(\Bu)+(\nu_*-\mu_*)\di\Bu\BI-(\ga_*-\ka_*\De)\rho\BI\}\Bn&=\Bg && \text{on $\BR_0^N$,} \\
\Bn\cdot\nabla\rho &= h && \text{on $\BR_0^N$,}
\end{aligned}\right.
\end{equation}
where we have used \eqref{DivT} and 
replaced in \eqref{resolvent} $\rho$, $\mu_*$, $\nu_*$, and $\ka_*$
by $\rho_*\rho$, $\rho_*\mu_*$, $\rho_*\nu_*$, and $\ka_*/\rho_*$, respectively.
For the right member $(d,\Bf,\Bg,h)$ of \eqref{eq:rescale}, we set
\begin{equation*}
X_q(G) 
= W_q^1(G)\times L_q(G)^N\times W_q^1(G)^N\times W_q^2(G).
\end{equation*}
Let $\BF=(d,\Bf,\Bg,h)\in X_q(G)$, and then
the symbols $\FX_q(G)$, $\CF_\la$ are defined as
\begin{align*}
&\FX_q(G)
=
W_q^1(G)\times L_q(G)^{\CN}, \quad \CN = N+N^2+N+N^2+N+1, \\
&\CF_\la\BF
=(d,\Bf,\nabla\Bg,\la^{1/2}\Bg,\nabla^2 h, \nabla\la^{1/2} h,\la h)
\in\FX_q(G).
\notag
\end{align*}
We also set, for solutions $(\rho,\Bu)$ of \eqref{eq:rescale},
\begin{alignat}{2}\label{ABset}
\FA_q(G) &= L_q(G)^{N^3+N^2}\times W_q^1(G), \quad
&\CS_\la \rho &= (\nabla^3 \rho, \la^{1/2}\nabla^2\rho,\la\rho), \\
\FB_q(G) &= L_q(G)^{N^3+N^2+N}, \quad
&\CT_\la\Bu &= (\nabla^2\Bu,\la^{1/2}\nabla\Bu,\la\Bu).
\notag
\end{alignat}

Let $\eta_*^w$ be a constant given by 
\begin{equation*}
\eta_*^w=\left(\frac{\mu_*+\nu_*}{2\ka_*}\right)^2 -\frac{1}{\ka_*}, \quad
\end{equation*}
and let $\si_*^w\in[0,\pi/2)$ be an angle defined as
\begin{align*}
\si_*^w &= 
\left\{\begin{aligned}
&0 && (\eta_*^w\geq 0), \\
&\arg\left(\frac{\mu_*+\nu_*}{2\ka_*}+i\sqrt{|\eta_*^w|}\right) && (\eta_*^w<0).
\end{aligned}\right.
\end{align*}

Now we state main results of this paper. 

\begin{theo}[Case $\ga_*=0$]\label{theo1}
Let $q\in(1,\infty)$ and $\de>0$.
Assume that $\mu_*$, $\nu_*$, and $\ka_*$ are positive constants
and that
\begin{equation}\label{para_cond}
\eta_*^w\neq 0, \quad \ka_*\neq\mu_*\nu_*.
\end{equation}
Then there exists a constant $\si_*\in(\si_*^w\,\pi/2)$, independent of $q$ and $\de$,
such that  for any $\si\in(\si_*,\pi/2)$
the following assertions hold true:
\begin{enumerate}[$(1)$]
\item
For any $\la\in\Si_{\si,\de}$ there are operators $\CA_0(\la)$ and $\CB_0(\la)$, with
\begin{align*}
\CA_0(\la)
&\in\Hol(\Si_{\si,\de},\CL(\FX_q(\BR_+^N),W_q^3(\BR_+^N))), \\
\CB_0(\la)
&\in\Hol(\Si_{\si,\de},\CL(\FX_q(\BR_+^N),W_q^2(\BR_+^N)^N)),
\end{align*}
such that, for every $\BF=(d,\Bf,\Bg,h)\in X_q(\BR_+^N)$,
$(\rho,\Bu)=(\CA_0(\la)\CF_\la\BF,\CB_0(\la)\CF_\la\BF)$ is a unique solution of \eqref{eq:rescale} for $\ga_*=0$.
\item
There is a positive constant $c(\de,\si)$, depending on at most 
$N$, $q$, $\de$, $\si$, $\mu_*$, $\nu_*$, and $\ka_*$,
such that for $n=0,1$
\begin{align*}
\CR_{\CL(\FX_q(\BR_+^N),\FA_q(\BR_+^N))}
\left(\left\{\left(\la \frac{d}{d\la}\right)^n\left(\CS_\la\CA_0(\la)\right)\mid \la\in\Si_{\si,\de}\right\}\right)
&\leq c(\de,\si), \\
\CR_{\CL(\FX_q(\BR_+^N),\FB_q(\BR_+^N))}
\left(\left\{\left(\la \frac{d}{d\la}\right)^n\left(\CT_\la\CB_0(\la)\right)\mid \la\in\Si_{\si,\de}\right\}\right)
&\leq c(\de,\si).
\end{align*}
\end{enumerate}
\end{theo}

\begin{rema}
We discuss the condition \eqref{para_cond} in more detail in Remark \ref{rema:sol} below.
\end{rema}

\begin{theo}[Case $\ga_*\in(-\infty,\infty)$]\label{theo2}
Let $q\in(1,\infty)$. Assume that
$\mu_*$, $\nu_*$, $\ga_*$, and $\ka_*$ are constants satisfying
\begin{equation*}
\mu_*>0, \quad \nu_*>0, \quad \ga_*\*\in(-\infty,\infty),
\quad \ka_*>0,
\end{equation*}
and also assume that the condition \eqref{para_cond} holds. 
Let $\de=1/2$ and $\si\in(\si_*,\pi/2)$ for $\si_*$ given in Theorem $\ref{theo1}$,
and let $c_0$ be a positive constant defined as $c_0=c(\de,\si)$
for the positive constant $c(\de,\si)$ of Theorem $\ref{theo1}$.
Then there is a constant $\la_0\geq 1$, depending only on $c_0$, $q$, and $\ga_*$,
such that the following assertions hold true:
\begin{enumerate}[$(1)$]
\item
For any $\la\in\Si_{\si,\la_0}$ there are operators $\CA(\la)$ and $\CB(\la)$, with
\begin{align*}
\CA(\la)
&\in\Hol(\Si_{\si,\la_0},\CL(\FX_q(\BR_+^N),W_q^3(\BR_+^N))), \\
\CB(\la)
&\in\Hol(\Si_{\si,\la_0},\CL(\FX_q(\BR_+^N),W_q^2(\BR_+^N)^N)),
\end{align*}
such that, for every $\BF=(d,\Bf,\Bg,h)\in X_q(\BR_+^N)$,
$(\rho,\Bu)=(\CA(\la)\CF_\la\BF,\CB(\la)\CF_\la\BF)$ is a unique solution of \eqref{eq:rescale}.
\item
For $n=0,1$,
\begin{align*}
\CR_{\CL(\FX_q(\BR_+^N),\FA_q(\BR_+^N))}
\left(\left\{\left(\la\frac{d}{d\la}\right)^n\left(\CS_\la\CA(\la)\right)\mid \la\in\Si_{\si,\la_0}\right\}\right)
&\leq 4c_0, \\
\CR_{\CL(\FX_q(\BR_+^N),\FB_q(\BR_+^N))}
\left(\left\{\left(\la\frac{d}{d\la}\right)^n\left(\CT_\la\CB(\la)\right)\mid \la\in\Si_{\si,\la_0}\right\}\right)
&\leq 4c_0.
\end{align*}
\end{enumerate}
\end{theo}

\begin{rema}\label{rema:theo}
\begin{enumerate}[(1)]
\item
Let $\si\in(\si_*,\pi/2)$, $\la\in\Si_{\si,\la_0}$, and $\BF=(d,\Bf,\Bg,h)\in X_q(\BR_+^N)$,
and let $(\rho,\Bu)\in W_q^3(\BR_+^N)\times W_q^2(\BR_+^N)^N$ be the solution of \eqref{eq:rescale}.
We then have by Definition \ref{defi:R} with $m=1$
\begin{equation}\label{170418_1}
\|(\CS_\la\rho,\CT_\la\Bu)\|_{\FA_q(\BR_+^N)\times\FB_q(\BR_+^N)}
\leq 8c_0\|\CF_\la\BF\|_{\FX_q(\BR_+^N)}.
\end{equation}
In addition, we have by \eqref{170418_1} and the first equation of \eqref{eq:rescale}
\begin{equation*}
|\la|^{3/2}\|\rho\|_{L_q(\BR_+^N)}
\leq |\la|^{1/2}\|d\|_{L_q(\BR_+^N)} + 8c_0\|\CF_\la\BF\|_{\FX_q(\BR_+^N)}.
\end{equation*}
\item\label{rema:theo_1}
Let $\Bg=0$ and $h=0$ in \eqref{eq:rescale}.
Then, by \eqref{170418_1},
\begin{align*}
&|\la|\|(\rho,\Bu)\|_{W_q^1(\BR_+^N)\times L_q(\BR_+^N)^N}+\|(\rho,\Bu)\|_{W_q^3(\BR_+^N)\times W_q^2(\BR_+^N)^N} \\
&\leq C_{\la_0}(8c_0)\|(d,\Bf)\|_{W_q^1(\BR_+^N)\times L_q(\BR_+^N)^N} \quad (\la\in\Si_{\si,\la_0})
\end{align*}
for some positive constant $C_{\la_0}$,
which implies that we can construct an analytic $C^0$-semigroup on $W_q^1(\BR_+^N)\times L_q(\BR_+^N)^N$
associated with \eqref{eq:Om_0} for $\Om_0=\BR_+^N$ and $\Ga_0=\BR_0^N$
under constant coefficients $\mu$, $\nu$, and $\ka$.
\item\label{rema:theo_2}
Combining Theorem \ref{theo2} with the operator-valued Fourier multiplier theorem due to Weis \cite[Theorem 3.4]{Weis01},
we can prove the maximal $L_p\text{-}L_q$ regularity $(1<p,q<\infty)$ for \eqref{eq:Om_0} (cf. \cite[Subsection 2.5]{MS17}).
%
%
\end{enumerate}
\end{rema}

This paper is organized as follows:
The next section proves the existence of $\CR$-bounded solution operator families
for the whole space problem with $\ga_*=0$.
In Section \ref{sec:half}, we first reduce \eqref{eq:rescale} with $\ga_*=0$
to the case where $d=0$ and $\Bf=0$
by means of solutions of the whole space problem.
Next we derive representation formulas of solutions of the reduced problem.
Section \ref{sec:tec} introduces some technical lemmas
in order to show the existence of $\CR$-bounded solution operator families
associated with the reduced problem.
Section \ref{sec:proof1} proves Theorem \ref{theo1}
by combining the result of the whole space problem with
$\CR$-bounded solution operator families of the reduced problem
that are constructed by the technical lemmas of Section \ref{sec:tec}
together with the representation formulas obtained in Section \ref{sec:half}.
Section \ref{sec:proof2} shows Theorem \ref{theo2} by Theorem \ref{theo1}
and a perturbation method.

\section{Whole space problem}\label{sec:whole}
In this section, we consider the following whole space problem:
\begin{equation}\label{eq:whole}
\left\{\begin{aligned}
\la\rho + \di\Bu &= d && \text{in $\ws$,} \\
\la\Bu - \mu_*\De\Bu -\nu_*\nabla\di\Bu
-\ka_*\De\nabla\rho &= \Bf && \text{in $\ws$.}
\end{aligned}\right.
\end{equation}
More precisely, we prove

\begin{theo}\label{theo:whole}
Let $q\in(1,\infty)$ and $\de>0$, and set $Y_q(\ws)=W_q^1(\ws)\times L_q(\ws)^N$.
Assume that $\mu_*$, $\nu_*$, and $\ka_*$ are constants satisfying
\begin{equation}\label{170425_1}
\mu_*>0, \quad \mu_*+\nu_*>0, \quad \ka_*>0.
\end{equation}
Then for any $\si\in(\si_*^w,\pi/2)$ the following assertions hold true:
\begin{enumerate}[$(1)$]
\item\label{theo:whole1}
For any $\la\in\Si_{\si,\de}$
there are operators $\CA_1(\la)$, $\CB_1(\la)$, with
\begin{align*}
\CA_1(\la) &\in \Hol(\Si_{\si,\de},\CL(Y_q(\ws),W_q^3(\ws))), \\
\CB_1(\la) &\in \Hol(\Si_{\si,\de},\CL(Y_q(\ws),W_q^2(\ws)^N)),
\end{align*}
such that, for every $\BF=(d,\Bf)\in Y_q(\ws)$,
$(\rho,\Bu)=(\CA_1(\la)\BF,\CB_1(\la)\BF)$ is a unique solution to the system \eqref{eq:whole}.
\item\label{theo:whole2}
There exists a positive constant $c_1$,
depending on at most
$N$, $q$, $\de$, $\si$, $\mu_*$, $\nu_*$, and $\ka_*$, such that for $n=0,1$
\begin{align*}
\CR_{\CL(Y_q(\ws),\FA_q(\ws))}
\left(\left\{\left(\la\frac{d}{d\la}\right)^n\left(\CS_\la \CA_1(\la)\right) \mid \la\in\Si_{\si,\de}\right\}\right)
\leq c_1,  \\
\CR_{\CL(Y_q(\ws), \FB_q(\ws))}
\left(\left\{\left(\la\frac{d}{d\la}\right)^n\left(\CT_\la \CB_1(\la)\right) \mid \la\in\Si_{\si,\de}\right\}\right)
\leq c_1, 
\end{align*}
\end{enumerate}
where $\FA(\ws)$, $\FB(\ws)$, $\CS_\la$, and $\CT_\la$ are given by \eqref{ABset}.
\end{theo}

We devote the remaining part of this section to the proof of Theorem \ref{theo:whole}.

First, we derive representation formulas of solutions of \eqref{eq:whole}.
Let us define the Fourier transform and its inverse transform as
\begin{equation*}
\wh u= \wh u (\xi) = \int_\ws e^{-ix\cdot\xi}\,u(x)\intd x, \quad
\CF_\xi^{-1}[v](x) = \frac{1}{(2\pi)^N}\int_\ws e^{ix\cdot\xi}\,v(\xi)\intd\xi,
\end{equation*} 
respectively.
We then apply the Fourier transform to \eqref{eq:whole}
in order to obtain
\begin{align}
\la\wh{\rho}+i\xi\cdot\wh\Bu&=\wh d,
\label{160909_1} \\
\la\wh \Bu +\mu_*|\xi|^2\wh\Bu-\nu_*i\xi(i\xi\cdot\wh \Bu)
+\ka_*i\xi|\xi|^2\wh\rho &=\wh \Bf.
\label{160909_2}
\end{align}
Inserting \eqref{160909_1} into \eqref{160909_2} furnishes 
\begin{equation}\label{160909_5}
\la^2\wh \Bu +\mu_*\la |\xi|^2\wh\Bu
-i\xi(\nu_*\la +\ka_*|\xi|^2)(i\xi\cdot\wh \Bu) 
= \la \wh \Bf-\ka_*i\xi|\xi|^2\wh d, 
\end{equation}
and the scalar product of this identity and $i\xi$ yields
\begin{equation*}
i\xi\cdot\wh\Bu = P(\xi,\la)^{-1}(\la i\xi\cdot\wh\Bf+\ka_*|\xi|^4\wh d), \quad
P(\xi,\la) = \la^2+(\mu_*+\nu_*)\la|\xi|^2+\ka_*|\xi|^4. \notag
\end{equation*}
We insert these formulas into \eqref{160909_1} and \eqref{160909_5}
in order to obtain
\begin{align*}
\wh\rho
&= \left(\frac{\la+(\mu_*+\nu_*)|\xi|^2}{P(\xi,\la)}\right)\wh d
-\sum_{j=1}^N\frac{i\xi_j}{P(\xi,\la)}\wh f_j, \\
\wh\Bu
&=-\frac{\ka_*i\xi|\xi|^2}{P(\xi,\la)}\wh d
+\frac{1}{\la+\mu_*|\xi|^2}
\left(\wh \Bf-\sum_{j=1}^N \frac{\xi_j\xi(\nu_*\la+\ka_*|\xi|^2)}{P(\xi,\la)}\wh f_j\right).
\end{align*}
Thus we have
\begin{align}
\rho &= \CF_\xi^{-1}\left[\left(\frac{\la+(\mu_*+\nu_*)|\xi|^2}{P(\xi,\la)}\right)\wh d(\xi)\right](x)
-\sum_{j=1}^N\CF_\xi^{-1}
\left[\frac{i\xi_j}{P(\xi,\la)}\wh f_j(\xi)\right](x) \label{160909_3} \\
&=: \CA_1(\la)\BF, \notag \\
\Bu &=- \CF_\xi^{-1}\left[\frac{\ka_*i\xi|\xi|^2}{P(\xi,\la)}\wh{d}(\xi)\right](x)
+\CF_\xi^{-1}\left[\frac{\wh\Bf(\xi)}{\la+\mu_*|\xi|^2}\right](x)
\label{160909_4}  \\
&-\sum_{j=1}^N\CF_\xi^{-1}
\left[\frac{\xi_j\xi(\nu_*\la+\ka_*|\xi|^2)}{(\la+\mu_*|\xi|^2)P(\xi,\la)}\wh f_j(\xi)\right](x)
\notag\\
&=:\CB_1(\la)\BF. \notag
\end{align}

Next we estimate $P(\xi,\la)$.

\begin{lemm}\label{lemm:P}
Assume that \eqref{170425_1} holds.
Then, for any $\si\in(\si_*^w,\pi/2)$, $\la\in\Si_{\si,0}$, and $\xi\in\BR^N$, we have
\begin{equation*}
|P(\xi,\la)|\geq C_{\si,\mu_*,\nu_*,\ka_*}(|\la|^{1/2}+|\xi|)^4
\end{equation*}
with some positive constant $C_{\si,\mu_*,\nu_*,\ka_*}$ independent of $\xi$ and $\la$.
%
%
\end{lemm}

\begin{proof}
We rewrite $P(\xi,\la)$ as 
\begin{equation*}
P(\xi,\la)
=\left(\la+\frac{\mu_*+\nu_*}{2}|\xi|^2\right)^2
-\eta_*^w\ka_*^2|\xi|^4,
\end{equation*}
and note that, for any $\si\in(0,\pi/2)$, $\la\in\Si_{\si,0}$, and $a\geq 0$,
\begin{equation}\label{160910_1}
|\la+a|\geq \left(\sin\frac{\si}{2}\right)(|\la|+a).
\end{equation}

First we consider the case $\eta_*^w\geq 0$. Since 
\begin{equation*}
P(\xi,\la) = (\la-\la_+)(\la-\la_-), \quad
\la_\pm = -\ka_*\left(\frac{\mu_*+\nu_*}{2\ka_*}\pm \sqrt{\eta_*^w}\right)|\xi|^2
\end{equation*}
and since $(\mu_*+\nu_*)/(2\ka_*)\pm \sqrt{\eta_*^w}>0$,
it follows from \eqref{160910_1} that for any $\si\in(0,\pi/2)$
\begin{equation}\label{160911_14}
|P(\xi,\la)|\geq C_{\si,\mu_*,\nu_*,\ka_*}(|\la|+|\xi|^2)^2 \quad (\la\in\Si_{\si,0},\,\xi\in\BR^N).
\end{equation}

Next we consider the case $\eta_*^w<0$.
One has
\begin{equation}\label{160911_12}
P(\xi,\la)=(\la-\la_+)(\la-\la_-), \quad \la_\pm = 
-\ka_*\left(\frac{\mu_*+\nu_*}{2\ka_*}\pm i\sqrt{|\eta_*^w|}\right)|\xi|^2.
\end{equation}
%
%
%
%
It then holds that
$\la_\pm = -\ka_*|\xi|^2e^{\pm i\si_*^w}\sqrt{\{(\mu_*+\nu_*)/(2\ka_*)\}^2+|\eta_*^w|}$.
Let $\la=|\la|e^{i\te}$ for $0\leq |\te|\leq \pi-\si$, $\si\in(\si_*^w,\pi/2)$. 
Noting $\overline{\la_\pm}=\la_\mp$, we observe that
\begin{align*}
&|\la-\la_\pm|^2
= (\la-\la_\pm)(\overline{\la}-\overline{\la_\pm}) \\
&=|\la|^2+|\la_\pm|^2 +2\ka_*|\la||\xi|^2\cos(\te\mp \si_*^w)\sqrt{\{(\mu_*+\nu_*)/(2\ka_*)\}^2+|\eta_*^w|}.
\end{align*}
Since the angles satisfy
\begin{equation*}
\begin{aligned}
-\si_*^w\leq \te-\si_*^w \leq \te+\si_*^w \leq \pi-(\si-\si_*^w) & &&\text{when $0\leq \te\leq \pi-\si$,} \\
-(\pi-(\si-\si_*^w))\leq \te-\si_*^w\leq \te+\si_*^w\leq \si_*^w & && \text{when $-(\pi-\si)\leq \te\leq 0$,}
\end{aligned}
\end{equation*}
we have
$\cos(\te\mp\si_*^w) \geq  \cos\{\pi-(\si-\si_*^w)\}=-\cos(\si-\si_*^w)$.
Combining this inequality with the last identity furnishes that
\begin{align*}
|\la-\la_\pm|^2 
&\geq |\la|^2+\ka_*^2|\xi|^4\left\{\left(\frac{\mu_*+\nu_*}{2\ka_*}\right)^2+|\eta_*^w|\right\} \\
&-2\ka_*|\la||\xi|^2\cos(\si-\si_*^w)\sqrt{\left(\frac{\mu_*+\nu_*}{2\ka_*}\right)^2+|\eta_*^w|} \\
&\geq (1-\cos(\si-\si_*^w))\left[|\la|^2+\ka_*^2|\xi|^4\left\{\left(\frac{\mu_*+\nu_*}{2\ka_*}\right)^2+|\eta_*^w|\right\}\right].
\end{align*}
Thus, by \eqref{160911_12},
we obtain for any $\si\in(\si_*^w,\pi/2)$
\begin{equation*}
|P(\xi,\la)|\geq C_{\si,\mu_*,\nu_*,\ka_*}(|\la|+|\xi|^2)^2 \quad (\la\in\Si_{\si,0},\,\xi\in\BR^N),
\end{equation*}
which, combined with \eqref{160911_14},
completes the proof of Lemma \ref{lemm:P}.
%
%
\end{proof}

\begin{coro}\label{coro:P}
Assume that \eqref{170425_1} holds.
Then, for any $\si\in(\si_*^w,\pi/2)$ and  any multi-index $\al\in\BN_0^N$,
there is a positive constant $C_{\al,\si,\mu_*,\nu_*,\ka_*}$ such that,
for any $\la\in\Si_{\si,0}$ and $\xi\in\BR^N$,
\begin{equation*}
\left|\pa_\xi^\al \left\{\la^n(d/d\la)^n P(\xi,\la)^{-1}\right\}\right|
\leq C_{\al,\si,\mu_*,\nu_*,\ka_*}(|\la|^{1/2}+|\xi|)^{-4-|\al|} \quad (n=0,1).
\end{equation*}
%
%
\end{coro}

\begin{proof}
To prove the corollary,
we use Bell's formula for derivatives of the composite function of $f(t)$ and $t=g(\xi)$
as follows: For any multi-index $\al\in\BN_0^N$,
\begin{equation}\label{Bell}
\pa_\xi^\al f(g(\xi))=\sum_{k=1}^{|\al|} f^{(k)}(g(\xi))
\sum_{\stackrel{\text{\scriptsize{$\al_1+\dots+\al_k=\al$,}}}{|\al_i|\geq 1}}
\Ga_{\al_1,\dots,\al_k}^\al(\pa_\xi^{\al_1}g(\xi))\dots(\pa_\xi^{\al_k}g(\xi))
\end{equation}
with suitable coefficients $\Ga_{\al_1,\dots,\al_k}^\al$,
where $f^{(k)}(t)$ is the $k$th derivative of $f(t)$.

Let $\al$ be any multi-index of $\BN_0^N$ in this proof.
Since 
\begin{equation}\label{170418_7}
|\pa_\xi^\al|\xi|^2|
\leq 
\left\{\begin{aligned}
&2|\xi| && (|\al|=1), \\
&2 && (|\al|=2), \\
&0 && (|\al|\geq 3),
\end{aligned}\right. 
\end{equation}
 we have by Leibniz's rule
$|\pa_\xi^\al|\xi|^4|\leq C_\al|\xi|^{4-|\al|}$ when $|\al|\leq 4$
and  $\pa_\xi^\al|\xi|^4=0$ when $|\al|\geq 5$.
These inequalities yield 
\begin{equation}\label{170418_5}
\left|\pa_\xi^\al\left\{\la^n(d/d\la)^n
P(\xi,\la)\right\}\right| \leq C_{\al,\mu_*,\nu_*,\ka_*}(|\la|^{1/2}+|\xi|)^{4-|\al|} \quad (n=0,1)
\end{equation}
for any $(\xi,\la)\in\BR^N\times \Si_{\si,0}$.
We thus see that by \eqref{170418_5} with $n=0$ and by \eqref{Bell} with $f(t)=t^s$ $(s\in\BR)$ and $t=g(\xi)=P(\xi,\la)$
\begin{align*}
|\pa_\xi^\al P(\xi,\la)^s|
&\leq C_{s,\al,\mu_*,\nu_*,\ka_*}\sum_{k=1}^{|\al|}|P(\xi,\la)|^{s-k}(|\la|^{1/2}+|\xi|)^{4k-|\al|},
\end{align*}
which, combined with Lemma \ref{lemm:P} when $s-k<0$
and with \eqref{170418_5} for $(\al,n)=(0,0)$ when $s-k\geq 0$, furnishes 
\begin{equation}\label{170418_9}
|\pa_\xi^\al P(\xi,\la)^s| \leq C_{s,\al,\si,\mu_*,\nu_*,\ka_*}(|\la|^{1/2}+|\xi|)^{4s-|\al|}
\end{equation}
for any $(\xi,\la)\in \BR^N\times\Si_{\si,0}$.
Setting $s=-1$ in \eqref{170418_9} especially implies
the required estimate of Corollary \ref{coro:P}  for $n=0$.
Noting that
\begin{equation*}
\la\frac{d}{d\la} P(\xi,\la)^{-1}=-P(\xi,\la)^{-2}\left(\la\frac{d}{d\la} P(\xi,\la)\right),
\end{equation*}
we also have the required estimate of Corollary \ref{coro:P}  for $n=1$
by Leibniz's rule together with \eqref{170418_5} for $n=1$ and \eqref{170418_9} for $s=-2$.
This completes the proof of the corollary.
\end{proof}

For $k,l,m=1,\dots,N$ we have, by \eqref{160909_3} and \eqref{160909_4},
\begin{align}
&\pa_k\pa_l\pa_m\CA_1(\la)\BF
= -\CF_\xi^{-1}\left[\frac{i\xi_k\xi_l\xi_m\{\la+(\mu_*+\nu_*)|\xi|^2\}}{P(\xi,\la)}\wh d(\xi)\right](x)
\label{160930_20} \\
& \quad-\sum_{j=1}^N\CF_\xi^{-1}\left[\frac{\xi_j\xi_k\xi_l\xi_m }{P(\xi,\la)}\wh f_j(\xi)\right](x),
\notag \\
&\la^{1/2}\pa_k\pa_l\CA_1(\la)\BF
= -\CF_\xi^{-1}\left[\frac{\xi_k\xi_l\la^{1/2}\{\la+(\mu_*+\nu_*)|\xi|^2\}}{P(\xi,\la)}\wh d(\xi)\right](x)
\label{160930_21} \\
& \quad+\sum_{j=1}^N\CF_\xi^{-1}\left[\frac{i\xi_j\xi_k\xi_l\la^{1/2}}{P(\xi,\la)}\wh f_j(\xi)\right](x),
\notag  \\
&\la\CA_1(\la)\BF =\CF_\xi^{-1}\left[\frac{\la\{\la+(\mu_*+\nu_*)|\xi|^2\}}{P(\xi,\la)}\wh d(\xi)\right](x)
\label{160930_22} \\
& \quad -\sum_{j=1}^N\CF_\xi^{-1}\left[\frac{i\xi_j\la}{P(\xi,\la)}\wh f_j(\xi)\right](x),
\notag \\
&\big(\pa_k\pa_l,\la^{1/2}\pa_k,\la\big)\CB_1(\la)\BF \label{160930_23} \\
&\quad=-\CF_\xi^{-1}\left[\left(-\xi_k\xi_l,i\xi_k\la^{1/2},\la\right)
\frac{\ka_*i\xi|\xi|^2}{P(\xi,\la)}\wh d(\xi)\right](x) 
\notag \\
&\quad+\CF_\xi^{-1}\left[\left(-\xi_k\xi_l,i\xi_k\la^{1/2},\la\right)
\frac{\wh\Bf(\xi)}{\la+\mu_*|\xi|^2}\right](x)
\notag \\
&\quad-\sum_{j=1}^N\CF_\xi^{-1}
\left[\left(-\xi_k\xi_l,i\xi_k\la^{1/2},\la\right)
\frac{\xi_j\xi(\nu_*\la+\ka_*|\xi|^2)}{(\la+\mu_*|\xi|^2)P(\xi,\la)}\wh f_j(\xi)\right](x).
\notag
\end{align}
%
%
%
%
To estimate these terms, we introduce the following two lemmas.
The first one was proved by \cite[Theorem 3.3]{ES13}.

\begin{lemm}\label{lemm:ES13}
Let $q\in(1,\infty)$ and $\La$ be a subset of $\BC$. 
Assume that $t(\xi,\la)$ is a function defined on $(\ws\setminus\{0\})\times\La$
such that for any multi-index $\al\in\BN_0^N$
there exists a positive constant $M_{\al,\La}$, depending on $\al$ and $\La$,
such that
\begin{equation*}
|\pa_\xi^\al  t(\xi,\la)| \leq M_{\al,\La} |\xi|^{-|\al|}
\end{equation*}
for any $(\xi,\la)\in (\ws\setminus\{0\})\times\La$.
Let $T(\la)$ be an operator defined by $[T(\la)f](x) = \CF_{\xi}^{-1}[t(\xi,\la)\wh f(\xi)](x)$.
Then the set $\{T(\la) \mid \la\in\La\}$ is $\CR$-bounded on $\CL(L_q(\ws))$ and
\begin{equation*}
\CR_{\CL(L_q(\ws))}(\{T(\la) \mid \la\in\La\}) \leq C_{N,q}\max_{|\al|\leq N+1}M_{\al,\La},
\end{equation*}
with some positive constant $C_{N,q}$ that depends solely on $N$ and $q$.
\end{lemm}

\begin{lemm}\label{lemm:whole}
Let $q\in(1,\infty)$, $\de>0$, and $\si\in(0,\pi/2)$. 
Assume that $k(\xi,\la)$, $l(\xi,\la)$, and $m(\xi,\la)$ are functions on $(\BR^N\setminus\{0\})\times \Si_{\si,0}$
such that for any multi-index $\al\in\BN_0^N$
there exists a positive constant $M_{\al,\si}$ such that
\begin{align*}
&|\pa_\xi^\al k(\xi,\la)| \leq M_{\al,\si}|\xi|^{1-|\al|}, \quad
|\pa_\xi^{\al}l(\xi,\la)| \leq M_{\al,\si}|\xi|^{-|\al|}, \\
&|\pa_\xi^\al m(\xi,\la)| \leq M_{\al,\si}(|\la|^{1/2}+|\xi|)^{-1}|\xi|^{-|\al|},
\end{align*}
for any $(\xi,\la)\in(\BR^N\setminus\{0\})\times \Si_{\si,0}$.
Let $K(\la)$, $L(\la)$, $M(\la)$ be operators given by
\begin{alignat*}{2}
[K(\la) f](x) &= \CF_{\xi}^{-1}[k(\xi,\la)\wh f(\xi)](x) && \quad (\la\in\Si_{\si,0}), \\
[L(\la) f](x) &= \CF_{\xi}^{-1}[l(\xi,\la)\wh f(\xi)](x) && \quad (\la\in\Si_{\si,0}), \\
[M(\la) f](x) &=\CF_{\xi}^{-1}[m(\xi,\la)\wh f(\xi)](x) && \quad (\la\in\Si_{\si,\de}).
\end{alignat*}
Then the following assertions hold true:
\begin{enumerate}[$(1)$]
\item\label{lemm:whole1}
The set $\{K(\la) \mid \la\in\Si_{\si,0}\}$ is $\CR$-bounded on
$\CL(W_q^1(\ws),L_q(\ws))$, and there exists a positive constant $C_{N,q}$ such that
\begin{equation*}
\CR_{\CL(W_q^1(\ws),L_q(\ws))}
\left(\left\{K(\la) \mid \la\in\Si_{\si,0}\right\}\right)
\leq C_{N,q}\max_{|\al|\leq N+1} M_{\al,\si}.
\end{equation*}
\item\label{lemm:whole2}
Let $n=0,1$.
Then the set $\{L(\la) \mid \la\in\Si_{\si,0}\}$ is $\CR$-bounded on $\CL(W_q^n(\ws))$,
and there exists a positive constant $C_{N,q}$ such that
\begin{equation*}
\CR_{\CL(W_q^n(\ws))}
\left(\left\{L(\la) \mid \la\in\Si_{\si,0}\right\}\right)
\leq C_{N,q}\max_{|\al|\leq N+1} M_{\al,\si}.
\end{equation*}
\item\label{lemm:whole3}
The set $\{M(\la) \mid \la\in\Si_{\si,\de}\}$ is $\CR$-bounded on
$\CL(L_q(\ws),W_q^1(\ws))$, and there exists a positive constant $C_{N,q,\de}$ such that
\begin{equation*}
\CR_{\CL(L_q(\ws),W_q^1(\ws))}
\left(\left\{M(\la) \mid \la\in\Si_{\si,\de}\right\}\right) \leq C_{N,q,\de}\max_{|\al|\leq N+1} M_{\al,\si}.
\end{equation*}
\end{enumerate}
\end{lemm}

\begin{proof}
Let $\al$ be any multi-index of $\BN_0^N$ in this proof.

\eqref{lemm:whole1}.
By using $1=|\xi|^2/|\xi|^2=-\sum_{j=1}^N(i\xi_j)^2/|\xi|^2$,
we write $K(\la) f$ as
\begin{equation*}
[K(\la) f](x)
=-\sum_{j=1}^N \CF_\xi^{-1}\left[\frac{i\xi_j k(\xi,\la)}{|\xi|^2}\wh{\pa_j f}(\xi)\right](x)
=:[\wt K(\la)\nabla f](x).
\end{equation*}
Since $|\pa_\xi^\al|\xi|^2|\leq 2|\xi|^{2-|\al|}$
for any $\xi\in\BR^N\setminus\{0\}$ by \eqref{170418_7},
we observe by \eqref{Bell} with $f(t)=t^{s/2}$ $(s\in\BR)$ and $t=g(\xi)=|\xi|^2$ that
$|\pa_\xi^\al|\xi|^s|\leq C_{s,\al}|\xi|^{s-|\al|}$ for any $\xi\in\BR^N\setminus\{0\}$.
Setting $s=-2$ in this inequality yields, together with Leibniz's rule and the assumption for $k(\xi,\la)$, that 
\begin{equation*}
\left|\pa_\xi^\al\left(\frac{i\xi_j k(\xi,\la)}{|\xi|^2}\right)\right|
\leq C_\al\left(\max_{|\beta|\leq |\al|}M_{\beta,\si}\right)|\xi|^{-|\al|}
\end{equation*}
for any $(\xi,\la)\in(\BR^N\setminus\{0\})\times\Si_{\si,0}$.
Thus, by Lemma \ref{lemm:ES13} and Definition \ref{defi:R},
we have
for any $m\in\BN$, $\{\la_j\}_{j=1}^m\subset\Si_{\si,0}$, and $\{\Bf_j\}_{j=1}^m\subset L_q(\ws)^N$
\begin{align*}
&\int_0^1\Big\|\sum_{j=1}^m r_j(u) K(\la_j) f_j\Big\|_{L_q(\ws)}^q\intd u
=\int_0^1\Big\|\sum_{j=1}^m r_j(u) \wt K(\la_j) \nabla  f_j\Big\|_{L_q(\ws)}^q\intd u \\
&\leq \left(C_{N,q}\max_{|\al|\leq N+1}M_{\al,\si}\right)^q
\int_0^1\Big\|\sum_{j=1}^m r_j(u)\nabla f_j\Big\|_{L_q(\ws)}^q\intd u \\
&\leq \left(C_{N,q}\max_{|\al|\leq N+1}M_{\al,\si}\right)^q\int_0^1\Big\|\sum_{j=1}^m r_j(u) f_j\Big\|_{W_q^1(\ws)}^q\intd u,
\end{align*}
which completes the proof of Lemma \ref{lemm:whole} \eqref{lemm:whole1}.

\eqref{lemm:whole2}.
The case $n=0$ was already proved in Lemma \ref{lemm:ES13}.
The proof of $n=1$ is similar to \eqref{lemm:whole1},
so that we may omit the detailed proof.

\eqref{lemm:whole3}.
By Leibniz's rule and the assumption for $m(\xi,\la)$,
\begin{align*}
&|\pa_\xi^\al m(\xi,\la)|\leq M_{\al,\si}\de^{-1/2}|\xi|^{-|\al|}, \\
&|\pa_\xi^\al\{i\xi_j m(\xi,\la)\}| \leq C_\al\left(\max_{|\beta|\leq |\al|} M_{\beta,\si}\right)|\xi|^{-|\al|}
\quad (j=1,\dots,N),
\end{align*}
for any $(\xi,\la)\in(\BR^N\setminus\{0\})\times \Si_{\si,\de}$. 
We then observe by Lemma \ref{lemm:ES13} that
$\{M(\la)\mid\la\in\Si_{\si,\de}\}$ and
$\{\pa_j M(\la)\mid\la\in\Si_{\si,\de}\}$ are $\CR$-bounded on $\CL(L_q(\BR^N))$
and their $\CR$-bounds are bounded above by $C_{N,q,\de}\max_{|\al|\leq N+1}M_{\al,\si}$
for some positive constant $C_{N,q,\de}$.
Thus we have the required properties by Definition \ref{defi:R} immediately.
This completes the proof of Lemma \ref{lemm:whole} \eqref{lemm:whole3}.
\end{proof}

In the same manner as we have obtained \eqref{170418_9},
we can prove the following inequality (cf. also \cite[Lemma 3.4]{SS12}):
For any $a,b>0$, $s\in\BR$, $\si\in(0,\pi/2)$, and multi-index $\al\in\BN_0^N$,
there is a positive constant $C_{a,b,s,\al,\si}$ such that
\begin{align}\label{ele:1}
\left|\pa_\xi^\al(a\la+b|\xi|^2)^s\right|
\leq C_{a,b,s,\al,\si}(|\la|^{1/2}+|\xi|)^{2s-|\al|}
\end{align}
for any $\la\in\Si_{\si,0}$ and $\xi\in\BR^N$.
Let $j,k,l,m=1,\dots,N$, $n=0,1$, and $\si\in(\si_*^w,\pi/2)$,
and let $\al$ be any multi-index of $\BN_0^N$ in what follows.

First, we consider the formula \eqref{160930_20}.
By Corollary \ref{coro:P}, \eqref{ele:1}, and Leibniz's rule, we have
for any $(\xi,\la)\in(\BR^N\setminus\{0\})\times\Si_{\si,0}$
\begin{align*}
\left|\pa_\xi^\al\left\{\left(\la\frac{d}{d\la}\right)^n
\left(\frac{i\xi_k\xi_l\xi_m\{\la+(\mu_*+\nu_*)|\xi|^2\}}{P(\xi,\la)}\right)\right\}\right|
&\leq C_{\al,\si,\mu_*,\nu_*,\ka_*}|\xi|^{1-|\al|}, \\
\left|\pa_\xi^\al\left\{\left(\la\frac{d}{d\la}\right)^n\left(\frac{\xi_j\xi_k\xi_l\xi_m }{P(\xi,\la)}\right)\right\}\right|
&\leq C_{\al,\si,\mu_*,\nu_*,\ka_*}|\xi|^{-|\al|},
\end{align*}
which, combined with Lemma \ref{lemm:whole},
furnishes 
\begin{align}\label{whole:A1}
&\CR_{\CL(Y_q(\ws),L_q(\ws)^{N^3})}
\big(\big\{\la^n(d/d\la)^n (\nabla^3\CA_1(\la)) \mid \la\in\Si_{\si,0}\big\}\big) \\
&\leq C_{N,q,\si,\mu_*,\nu_*,\ka_*}. \notag
\end{align}
Analogously, concerning \eqref{160930_21}, we can prove that 
\begin{align}\label{whole:A2}
&\CR_{\CL(Y_q(\ws),L_q(\ws)^{N^2})}
\big(\big\{\la^n(d/d\la)^n (\la^{1/2}\nabla^2\CA_1(\la)) \mid \la\in\Si_{\si,0}\big\}\big) \\
&\leq C_{N,q,\si,\mu_*,\nu_*,\ka_*}. \notag
\end{align}

Secondly, we consider the formula \eqref{160930_22}.
By Corollary \ref{coro:P}, \eqref{ele:1}, and Leibniz's rule, we have
for any $(\xi,\la)\in(\BR^N\setminus\{0\})\times\Si_{\si,0}$
\begin{align*}
&\left|\pa_\xi^\al\left\{\left(\la\frac{d}{d\la}\right)^n\left(\frac{\la\{\la+(\mu_*+\nu_*)|\xi|^2\}}{P(\xi,\la)}\right)\right\}\right|
\leq C_{\al,\si,\mu_*,\nu_*,\ka_*}|\xi|^{-|\al|}, \\
&\left|\pa_\xi^\al\left\{\left(\la\frac{d}{d\la}\right)^n\left(\frac{i\xi_j\la}{P(\xi,\la)}\right)\right\}\right|
\leq C_{\al,\si,\mu_*,\nu_*,\ka_*}(|\la|^{1/2}+|\xi|)^{-1}|\xi|^{-|\al|},
\end{align*}
which, combined with Lemma \ref{lemm:whole}, furnishes 
\begin{align}\label{whole:A3}
&\CR_{\CL(Y_q(\ws),W_q^1(\ws))}
\big(\big\{\la^n(d/d\la)^n (\la\CA_1(\la)) \mid \la\in\Si_{\si,\de}\big\}\big) \\
&\leq C_{N,q,\de,\si,\mu_*,\nu_*,\ka_*}. \notag
\end{align}
Thus \eqref{whole:A1}, \eqref{whole:A2}, and \eqref{whole:A3} imply 
the $\CR$-boundedness of $\CS_\la\CA_1(\la)$ required in Theorem \ref{theo:whole} \eqref{theo:whole2}.

Thirdly, we consider the formula \eqref{160930_23}.
By Corollary \ref{coro:P}, \eqref{ele:1}, and Leibniz's rule, we have
for any $(\xi,\la)\in(\BR^N\setminus\{0\})\times\Si_{\si,0}$
\begin{align*}
\left|\pa_\xi^\al\left\{\left(\la\frac{d}{d\la}\right)^n\left(\frac{\ka_*i\xi_k\xi_l\xi|\xi|^2}{P(\xi,\la)}\right)\right\}\right|
&\leq C_{\al,\si,\mu_*,\nu_*,\ka_*}|\xi|^{1-|\al|},\\
\left|\pa_\xi^\al  \left\{\left(\la\frac{d}{d\la}\right)^n\left(\frac{\xi_k\xi_l}{\la+\mu_*|\xi|^2}\right)\right\}\right|
&\leq C_{\al,\si,\mu_*,\nu_*,\ka_*}|\xi|^{-|\al|}, \\
\left|\pa_\xi^\al\left\{\left(\la\frac{d}{d\la}\right)^n
\left(\frac{\xi_j\xi_k\xi_l\xi(\nu_*\la+\ka_*|\xi|^2)}{(\la+\mu_*|\xi|^2)P(\xi,\la)}\right)\right\}\right|
&\leq C_{\al,\si,\mu_*,\nu_*,\ka_*}|\xi|^{-|\al|},
\end{align*}
which, combined with Lemma \ref{lemm:whole},
furnishes 
\begin{align}\label{whole:B1}
&\CR_{\CL(Y_q(\ws), L_q(\ws)^{N^3})}
\big(\big\{\la^n(d/d\la)^n (\nabla^2\CB_1(\la)) \mid \la\in\Si_{\si,0}\big\}\big) \\
&\leq C_{N,q,\si,\mu_*,\nu_*,\ka_*}. \notag
\end{align}
Analogously, we have
\begin{align}
&\CR_{\CL(Y_q(\ws), L_q(\ws)^{N^2})}
\big(\big\{\la^n(d/d\la)^n (\la^{1/2}\nabla\CB_1(\la)) \mid \la\in\Si_{\si,0}\big\}\big) \label{whole:B2} \\
&\leq C_{N,q,\si,\mu_*,\nu_*,\ka_*}, \notag \\
&\CR_{\CL(Y_q(\ws), L_q(\ws)^N)}
\big(\big\{\la^n(d/d\la)^n (\la\CB_1(\la)) \mid \la\in\Si_{\si,0}\big\}\big) \label{whole:B3} \\
&\leq C_{N,q,\si,\mu_*,\nu_*,\ka_*}. \notag 
\end{align}
Thus \eqref{whole:B1}, \eqref{whole:B2}, and \eqref{whole:B3} imply 
the $\CR$-boundedness of $\CT_\la\CB_1(\la)$ required in Theorem \ref{theo:whole} \eqref{theo:whole2}.

Finally, we show the uniqueness by means of the existence of solutions already proved above.
Let $(\rho,\Bu)\in W_q^3(\BR^N)\times W_q^2(\BR^N)^N$ be a solution 
of \eqref{eq:whole} with $\la\in\Si_{\si,\de}$ ($\si\in(\si_*^w,\pi/2)$, $\de>0$), $d=0$, and $\Bf=0$,
and let $\ph\in C_0^\infty(\BR^N)^N$.
Since $\ph\in L_{q'}(\BR^N)^N$ for $q'=q/(q-1)$,
we have a solution $(\wt\rho,\wt\Bu)\in W_{q'}^3(\BR^N)\times W_{q'}^2(\BR^N)^N$ to 
the following system:
\begin{equation*}
\left\{\begin{aligned}
\la\wt\rho + \di\wt\Bu &= 0 && \text{in $\ws$,} \\
\la\wt\Bu - \mu_*\De\wt\Bu -\nu_*\nabla\di\wt\Bu
-\ka_*\De\nabla\wt\rho &= \ph && \text{in $\ws$.}
\end{aligned}\right.
\end{equation*}
Then, by integration by parts, 
\begin{equation*}
(\Bu,\ph)_{\BR^N} 
= (\la\Bu - \mu_*\De\Bu -\nu_*\nabla\di\Bu,\wt\Bu)_{\BR^N}
+(\di\Bu,\ka_*\De\wt\rho)_{\BR^N}. 
\end{equation*}
We furthermore observe that 
\begin{align*}
(\di\Bu,\ka_*\De\wt\rho)_{\BR^N}
&=-(\la\rho,\ka_*\De\wt\rho)_{\BR^N}=
-(\ka_*\De\rho,\la\wt\rho)_{\BR^N} \\
& = (\ka_*\De\rho,\di\wt\Bu)_{\BR^N} 
=-(\ka_*\De\nabla\rho,\wt\Bu)_{\BR^N}.
\end{align*}
Thus, for any $\ph\in C_0^\infty(\BR^N)^N$,
\begin{equation*}
(\Bu,\ph)_{\BR^N}=(\la\Bu - \mu_*\De\Bu -\nu_*\nabla\di\Bu-\ka_*\De\nabla\rho,\wt\Bu)_{\BR^N}=0.
\end{equation*}
This implies that $\Bu=0$, which, combined with the equation $\la\rho+\di\Bu=0$ in $\BR^N$,
furnishes that $\rho=0$.
The proof of Theorem \ref{theo:whole} has been completed.

\section{Representation formulas of solutions for a half-space problem}\label{sec:half}
We prove Theorem \ref{theo1} from this section to Section \ref{sec:proof1},
i.e. we consider 
\begin{equation}\label{eq:full}
\left\{\begin{aligned}
\la\rho + \di\Bu &= d && \text{in $\uhs$,} \\
\la\Bu - \mu_*\De\Bu-\nu_*\nabla\di\Bu
-\ka_*\De\nabla\rho &= \Bf && \text{in $\uhs$,} \\
\left\{\mu_*\BD(\Bu)+(\nu_*-\mu_*)\di\Bu\BI
+\ka_*\De\rho\BI\right\}\Bn &= \Bg && \text{on $\bdry$,} \\
\Bn\cdot\nabla\rho &=h && \text{on $\bdry$.}
\end{aligned}\right.
\end{equation}

Let $E$ be an extension operator from $W_q^1(\BR_+^N)$ to $W_q^1(\BR^N)$,
and let $E_0$ be the zero extension operator from $L_q(\BR_+^N)$ to $L_q(\BR^N)$.
Setting in \eqref{eq:full}
\begin{equation*}
\rho = \CA_1(\la)(E d, E_0\Bf)+\wt\rho, \quad \Bu = \CB_1(\la)(E d, E_0\Bf)+\wt\Bu
\end{equation*}
for $\CA_1(\la)$, $\CB_1(\la)$ obtained in Theorem \ref{theo:whole},
we have
\begin{equation*}
\left\{\begin{aligned}
\la\wt\rho + \di\wt\Bu &= 0 && \text{in $\uhs$,} \\
\la\wt\Bu- \mu_*\De\wt\Bu -\nu_*\nabla\di\wt\Bu
-\ka_*\De\nabla\wt\rho &= 0 && \text{in $\uhs$,} \\
\left\{\mu_*\BD(\wt\Bu)+(\nu_*-\mu_*)\di\wt\Bu\BI
+\ka_*\De\wt\rho\,\BI\right\}\Bn &= \wt\Bg && \text{on $\bdry$,} \\
\Bn\cdot\nabla\wt\rho &= \wt h && \text{on $\bdry$,}
\end{aligned}\right.
\end{equation*}
where we have set
\begin{align*}
\wt\Bg
&=\Bg 
- \big\{\mu_*\BD\left(\CB_1(\la)(E d, E_0\Bf)\right)
+(\nu_*-\mu_*)\di\CB_1(\la)(E d, E_0\Bf)\BI \\
& +\ka_*\De\CA_1(\la)(E d, E_0\Bf)\BI\big\}\Bn, \\
\wt h &= h - \Bn\cdot\nabla\CA_1(\la)(E d, E_0\Bf).
\end{align*}
From this viewpoint, we consider the following reduced system:
\begin{equation}\label{eq:half1}
\left\{\begin{aligned}
\la\rho + \di\Bu &= 0 && \text{in $\uhs$,} \\
\la\Bu - \mu_*\De\Bu -\nu_*\nabla\di\Bu
-\ka_*\De\nabla\rho&= 0 && \text{in $\uhs$,} \\
\left\{\mu_*\BD(\Bu)+(\nu_*-\mu_*)\di\Bu\BI
+\ka_*\De\rho\BI\right\}\Bn &= \Bg && \text{on $\bdry$,} \\
\Bn\cdot\nabla\rho &=h && \text{on $\bdry$.}
\end{aligned}\right.
\end{equation}

We denote the reduced versions of $X_q(\BR_+^N)$, $\FX_q(\BR_+^N)$, and $\CF_\la$
by $\wt X_q(\BR_+^N)$, $\wt\FX_q(\BR_+^N)$, and $\wt\CF_\la$, respectively,
i.e. we set $\wt X_q(\BR_+^N)=W_q^1(\BR_+^N)^N\times W_q^2(\BR_+)$ and set, for $\BG=(\Bg,h)\in \wt X_q(\BR_+^N)$, 
\begin{align*}
&\wt\FX_q(\BR_+^N)=L_q(\BR_+^N)^{\wt\CN}, \quad \wt\CN=N^2+N+N^2+N+1, \\
&\wt\CF_\la \BG = (\nabla\Bg,\la^{1/2}\Bg,\nabla^2 h ,\nabla\la^{1/2}h,\la h)\in\wt\FX_q(\BR_+^N).
\end{align*}
It then suffices to show the following theorem
in order to prove Theorem \ref{theo1}.

\begin{theo}\label{theo:half}
Let $q\in(1,\infty)$ and $\de>0$.
Assume that $\mu_*$, $\nu_*$, and $\ka_*$ are positive constants
and that the condition \eqref{para_cond} holds.
Then there exists a constant $\si_*\in(\si_*^w,\pi/2)$,
independent of $q$ and $\de$, such that
for any $\si\in(\si_*,\pi/2)$ the following assertions hold true:
\begin{enumerate}[$(1)$]
\item
For any $\la\in\Si_{\si,\de}$ there are operator families $\CA_2(\la)$, $\CB_2(\la)$, with
\begin{align*}
\CA_2(\la) &\in \Hol(\Si_{\si,\de},\CL(\wt\FX_q(\uhs),W_q^3(\uhs))), \\
\CB_2(\la) &\in \Hol(\Si_{\si,\de},\CL(\wt\FX_q(\uhs),W_q^2(\uhs)^N)),
\end{align*}
such that, for every $\BG=(\Bg,h)\in\wt X_q(\uhs)$,
$(\rho,\Bu)=(\CA_2(\la)\wt\CF_\la\BG,\CB_2(\la)\wt\CF_\la\BG)$
is a solution of \eqref{eq:half1}.
\item
There exists a positive constant $c_2$,
depending on at most $N$, $q$, $\de$, $\si$, $\mu_*$, $\nu_*$, and $\ka_*$,
such that for $n=0,1$
\begin{align*}
\CR_{\CL(\wt\FX_q(\uhs), \FA_q(\uhs))}
\left(\left\{\left(\la\frac{d}{d\la}\right)^n(\CS_\la\CA_2(\la)) \mid \la\in\Si_{\si,\de}\right\}\right)&\leq c_2, \\
\CR_{\CL(\wt\FX_q(\uhs),\FB_q(\uhs))}
\left(\left\{\left(\la\frac{d}{d\la}\right)^n(\CT_\la\CB_2(\la)) \mid \la\in\Si_{\si,\de}\right\}\right)&\leq c_2,
\end{align*}
where $\FA(\BR_+^N)$, $\FB(\BR_+^N)$, $\CS_\la$, and $\CT_\la$ are given by \eqref{ABset}.
\end{enumerate}
\end{theo}

\begin{rema}\label{rema:half}
\begin{enumerate}[(1)]
\item\label{rema:half1}
The uniqueness of \eqref{eq:full} follows from the existence of solutions
for a dual problem in a similar way to the whole space problem.
\item\label{rema:half2}
The $\CR$-boundedness has the following properties (cf. \cite[Proposition 3.4]{DHP03}):
\begin{enumerate}[(a)]
\item
Let $X$ and $Y$ be Banach spaces, and let $\CT$ and $\CS$ be $\CR$-bounded families on $\CL(X,Y)$.
Then $\CT+\CS=\{T+S \mid T\in\CT, S\in\CS\}$ is also $\CR$-bounded on $\CL(X,Y)$
and $\CR_{\CL(X,Y)}(\CT+\CS)\leq \CR_{\CL(X,Y)}(\CT)+\CR_{\CL(X,Y)}(\CS)$.
\item
Let $X$, $Y$, and $Z$ be Banach spaces, and let $\CT$ and $\CS$ be $\CR$-bounded families
on $\CL(X,Y)$ and on $\CL(Y,Z)$, respectively.
Then $\CS\CT=\{ST \mid S\in\CS, T\in\CT\}$
is also $\CR$-bounded on $\CL(X,Z)$ and $\CR_{\CL(X,Z)}(\CS\CT)\leq \CR_{\CL(X,Y)}(\CT)\CR_{\CL(Y,Z)}(\CS)$.
\end{enumerate}
By these properties, 
Theorem \ref{theo1} follows from Theorems \ref{theo:whole}, \ref{theo:half} immediately.
\end{enumerate}
\end{rema}

The remaining part of this section is devoted to compute
the representation formulas of solutions of \eqref{eq:half1}.
Let us define the partial Fourier transform with respect to $x'=(x_1,\dots,x_{N-1})$ and its inverse transform as
\begin{align*}
&\wh u = \wh u(x_N) =\wh u(\xi',x_N) =\int_{\BR^{N-1}}e^{-ix'\cdot\xi'} u(x',x_N)\intd x', \\
&\CF_{\xi'}^{-1}[v(\xi',x_N)](x') = \frac{1}{(2\pi)^{N-1}}\int_{\BR^{N-1}}e^{ix'\cdot\xi'} v(\xi',x_N)\intd \xi'.
\end{align*}

The system \eqref{eq:half1} can be written as
\begin{equation*}
\left\{\begin{aligned}
\la \rho +\di\Bu &= 0 && \text{in $\uhs$,} \\
\la\Bu -\mu_*\De\Bu-\nu_*\nabla\di\Bu-\ka_*\De\nabla\rho&=0 && \text{in $\uhs$,} \\
\mu_*(\pa_j u_N+\pa_N u_j) &=-g_j && \text{on $\BR_0^N$,} \\
2\mu_*\pa_N u_N +(\nu_*-\mu_*)\di\Bu +\ka_*\De\rho&=-g_N && \text{on $\BR_0^N$,} \\
\pa_N\rho &=-h && \text{on $\BR_0^N$,}
\end{aligned}\right.
\end{equation*}
for $j=1,\dots,N-1$, where $\Bu=(u_1,\dots,u_N)^\SST$.
Set $\ph=\di\Bu$.
Applying the partial Fourier transform to the above system yields
the ordinary differential equations: 
\begin{align}
\la\wh\rho +\wh\ph &=0,
\text{ $x_N>0$,} \label{eq:1} \\ 
\la \wh u_j-\mu_*(\pa_N^2-|\xi'|^2)\wh u_j-\nu_*i\xi_j\wh\ph
-\ka_*i\xi_j(\pa_N^2-|\xi'|^2)\wh\rho&=0,
\text{ $x_N>0$,} \label{eq:2} \\
\la \wh u_N-\mu_*(\pa_N^2-|\xi'|^2)\wh u_N-\nu_*\pa_N\wh\ph
-\ka_*\pa_N(\pa_N^2-|\xi'|^2)\wh\rho&=0, 
\text{ $x_N>0$,} \label{eq:3}
\end{align}
with the boundary conditions: 
\begin{align}
\mu_*(i\xi_j\wh u_N(0)+\pa_N\wh u_j(0)) &=-\wh g_j(0), \label{eq:4} \\
2\mu_*\pa_N\wh u_N(0)+(\nu_*-\mu_*)\wh \ph(0)+\ka_*(\pa_N^2-|\xi'|^2)\wh\rho(0) &=-\wh g_N(0), \label{eq:5} \\
\pa_N\wh\rho(0) &=-\wh h(0). \label{eq:6}
\end{align}
We insert \eqref{eq:1} into \eqref{eq:2}, \eqref{eq:3}, \eqref{eq:5}, and \eqref{eq:6}
in order to obtain
\begin{align}
\la^2\wh u_j-\la\mu_*(\pa_N^2-|\xi'|^2)\wh u_j
-i\xi_j\{\la\nu_* -\ka_*(\pa_N^2-|\xi'|^2)\}\wh \ph &=0,\text{ $x_N>0$,} \label{eq:41} \\
\la^2\wh u_N-\la\mu_*(\pa_N^2-|\xi'|^2)\wh u_N
-\pa_N\{\la\nu_* -\ka_*(\pa_N^2-|\xi'|^2)\}\wh \ph &=0, \text{ $x_N>0$,} \label{eq:42} \\
2\la\mu_*\pa_N\wh u_N(0)+
\left\{\la(\nu_*-\mu_*)-\ka_*(\pa_N^2-|\xi'|^2)\right\}\wh\ph(0) &=-\la\wh g_N(0), \label{eq:43} \\
\pa_N\wh \ph(0) &=\la\wh h(0). \label{eq:44}
\end{align}
By \eqref{eq:41} and \eqref{eq:42}, we have
\begin{equation*}
\la^2\wh \ph-\la(\mu_*+\nu_*) (\pa_N^2-|\xi'|^2)\wh \ph 
+\ka_*(\pa_N^2-|\xi'|^2)^2\wh \ph=0, \text{ $x_N>0$,}
\end{equation*}
which implies that $P_\la(\pa_N)\wh \ph =0$ with
\begin{equation*}
P_\la(t) 
=\la^2 -\la(\mu_*+\nu_*)(t^2-|\xi'|^2)
+\ka_*(t^2-|\xi'|^2)^2.
\end{equation*}
Here we  set
\begin{equation}\label{omega}
\om_\la = \sqrt{|\xi'|^2+ \frac{\la}{\mu_*}}, \quad
\Re\om_\la>0 \quad \text{for $\la\in\Si_{\si,0}$, $\si\in(0,\pi/2)$.}
\end{equation}
Applying $P_\la(\pa_N)$ to \eqref{eq:41}, \eqref{eq:42} furnishes 
\begin{equation}\label{160827_1}
(\pa_N^2-\om_\la^2)P_\la(\pa_N) \wh u_J =0 \quad
(J=1,\dots,N).
\end{equation}

We consider the roots of $P_\la(t)$ at this point. 
Since
\begin{equation*}
P_\la(t)=\ka_*\la^2\left\{\frac{1}{\ka_*}-\left(\frac{\mu_*+\nu_*}{\ka_*}\right)
\left(\frac{t^2-|\xi'|^2}{\la}\right)+\left(\frac{t^2-|\xi'|^2}{\la}\right)^2\right\},
\end{equation*}
we set $s=(t^2-|\xi'|^2)/\la$ and solve the equation:
\begin{equation}\label{170309_1}
s^2-\frac{\mu_*+\nu_*}{\ka_*}s+\frac{1}{\ka_*}=0.
\end{equation}
By the assumption $\eta_*^w\neq0$, we have the solutions $s_1$, $s_2$
$(s_1\neq s_2)$ of \eqref{170309_1}
such that $s_1=s_+$ and $s_2=s_-$ with
\begin{equation*}
s_\pm = 
\left\{\begin{aligned}
&\frac{\mu_*+\nu_*}{2\ka_*} \pm \sqrt{\eta_*^w} && (\eta_*^w>0), \\
&\frac{\mu_*+\nu_*}{2\ka_*} \pm i\sqrt{|\eta_*^w|} && (\eta_*^w<0).
\end{aligned}\right.
\end{equation*}
Thus setting, for $\la\in\Si_{\si,0}$ with $\si\in(\si_*^w,\pi/2)$,
\begin{alignat}{2}\label{t1t2}
t_1&=\sqrt{|\xi'|^2+s_1\la}, \quad &t_2&=\sqrt{|\xi'|^2+s_2\la}, \\
t_3&=-\sqrt{|\xi'|^2+s_1\la}, \quad &t_4 &=-\sqrt{|\xi'|^2+s_2\la},
\notag
\end{alignat}
we see that $t_k=t_k(\xi',\la)$ $(k=1,2,3,4)$ are the roots of $P_\la(t)$ different from each other
and that $\Re t_1>0$, $\Re t_2>0$, $\Re t_3<0$, and $\Re t_4<0$.

\begin{rema}\label{rema:sol}
We have in general the following situations concerning roots with positive real parts
for the characteristic equation of \eqref{160827_1}: 
\begin{enumerate}[(1)]
\item Case $\eta_*^w<0$.
It holds that $\om_\la\neq t_1$, $\om_\la\neq t_2$, and $t_1\neq t_2$.
\item Case $\eta_*^w=0$.
There are two cases: $\om_\la\neq t_1$ and $t_1=t_2$; $\om_\la=t_1=t_2$.
\item Case $\eta_*^w>0$.
There are three cases: 
$\om_\la\neq t_1$, $\om_\la\neq t_2$, and $t_1\neq t_2$;
$\om_\la=t_1$ and $t_1\neq t_2$;
$\om_\la=t_2$ and $t_1\neq t_2$.
\end{enumerate}
The condition \eqref{para_cond} guarantees that
we have the three roots with positive real parts different from each other.
We, however, think that our technique in the following 
can be applied to the case of equal roots.
\end{rema}

In view of \eqref{160827_1} and Remark \ref{rema:sol},
we look for solutions $\wh u_J$ of the forms:
\begin{equation*}
\wh u_J = \al_J e^{-\om_\la x_N} + \beta_J(e^{-t_1 x_N}-e^{-\om_\la x_N})
+ \ga_J(e^{-t_2 x_N}-e^{-\om_\la x_N}).
\end{equation*}
Here and subsequently, $J$ runs from $1$ to $N$, while $j$ runs from $1$ to $N-1$.
It then holds that
\begin{align}
\pa_N \wh u_J
&=\big(-\om_\la\al_J+\om_\la\beta_J+\om_\la\ga_J\big)e^{-\om_\la x_N} \label{eq:8} \\
&-t_1\beta_J e^{-t_1 x_N}-t_2\ga_J e^{-t_2 x_N}, \notag \\
\wh \ph
&=\big(i\xi'\cdot\al'-i\xi'\cdot\beta'-i\xi'\cdot\ga'-\om_\la\al_N+\om_\la\beta_N+\om_\la\ga_N\big)e^{-\om_\la x_N} \label{eq:9} \\
&+\big(i\xi'\cdot\beta'-t_1\beta_N\big)e^{-t_1 x_N} 
+ \big(i\xi'\cdot\ga' -t_2\ga_N \big)e^{-t_2 x_N}, \notag
\end{align}
where $i\xi'\cdot a'=\sum_{j=1}^{N-1}i\xi_j a_j$ for $a\in \{\al,\beta,\ga\}$.
By \eqref{eq:41} and \eqref{eq:42}, we have
\begin{align*}
\mu_*\la(\pa_N^2-\om_\la^2)\wh u_j
+i\xi_j\{\nu_*\la-\ka_*(\pa_N^2-|\xi'|^2)\}\wh \ph=0, \text{ $x_N>0$,} \\
\mu_*\la(\pa_N^2-\om_\la^2)\wh u_N
+\pa_N\{\nu_*\la-\ka_*(\pa_N^2-|\xi'|^2)\}\wh \ph=0, \text{ $x_N>0$,}
\end{align*}
which, combined with \eqref{eq:9} and the assumption $\ka_*\neq \mu_*\nu_*$, furnishes that
\begin{align}
i\xi'\cdot\al'-i\xi'\cdot\beta'-i\xi'\cdot\ga'
-\om_\la\al_N+\om_\la\beta_N+\om_\la\ga_N &=0, \label{eq:10} \\
\mu_*\la\beta_j(t_1^2-\om_\la^2)+i\xi_j(i\xi'\cdot\beta'-t_1\beta_N)
\{\nu_*\la-\ka_*(t_1^2-|\xi'|^2)\}&=0, \label{eq:11} \\
\mu_*\la\ga_j(t_2^2-\om_\la^2)+i\xi_j(i\xi'\cdot\ga'-t_2\ga_N)
\{\nu_*\la-\ka_*(t_2^2-|\xi'|^2)\}&=0, \label{eq:12} \\
\mu_*\la\beta_N(t_1^2-\om_\la^2)-t_1(i\xi'\cdot\beta'-t_1\beta_N)
\{\nu_*\la-\ka_*(t_1^2-|\xi'|^2)\}&=0, \label{eq:13} \\
\mu_*\la\ga_N(t_2^2-\om_\la^2)-t_2(i\xi'\cdot\ga'-t_2\ga_N)
\{\nu_*\la-\ka_*(t_2^2-|\xi'|^2)\}&=0. \label{eq:14}
\end{align}
By \eqref{eq:11}-\eqref{eq:14}, we have
\begin{equation*}
(t_1^2-\om_\la^2)\left(\beta_j+\frac{i\xi_j}{t_1}\beta_N\right) = 0,\quad
(t_2^2-\om_\la^2)\left(\ga_j+\frac{i\xi_j}{t_2}\ga_N\right)=0.
\end{equation*}
As was seen in Remark \ref{rema:sol}, we know that $\om_\la\neq t_1$ and $\om_\la \neq t_2$
under the condition \eqref{para_cond},
and therefore the last two identities imply that
\begin{equation}\label{eq:51}
\beta_j = -\frac{i\xi_j}{t_1}\beta_N, \quad \ga_j = -\frac{i\xi_j}{t_2}\ga_N.
\end{equation}
These relations, furthermore, yield that
\begin{align}
i\xi'\cdot\beta'-t_1\beta_N
&=-t_1^{-1}(t_1^2-|\xi'|^2)\beta_N, \label{eq:60} \\
i\xi'\cdot\ga'-t_2\ga_N
&=	-t_2^{-1}(t_2^2-|\xi'|^2)\ga_N. \label{eq:61} 
\end{align}
On the other hand, we have by \eqref{eq:9} and \eqref{eq:10}
\begin{equation}\label{eq:50}
\wh \ph = (i\xi'\cdot\beta'-t_1\beta_N)e^{-t_1 x_N}
+(i\xi'\cdot\ga'-t_2\ga_N)e^{-t_2 x_N}. 
\end{equation}

Next we consider the boundary conditions.
By \eqref{eq:4} and \eqref{eq:8}, we have 
\begin{equation}\label{eq:54}
i\xi_j\al_N -\om_\la \al_j+(-t_1+\om_\la)\beta_j + (-t_2+\om_\la)\ga_j  = -\mu_*^{-1}\wh g_j(0).
\end{equation}
In addition, by \eqref{eq:43}, \eqref{eq:44}, \eqref{eq:8}, and \eqref{eq:50}, we have
\begin{align}
&2\mu_*\la\{-\om_\la\al_N + (-t_1+\om_\la)\beta_N + (-t_2 +\om_\la)\ga_N\} \label{eq:55} \\
&\quad+\{\la(\nu_*-\mu_*)-\ka_*(t_1^2-|\xi'|^2)\}(i\xi'\cdot\beta'-t_1\beta_N)  \notag\\
&\quad+\{\la(\nu_*-\mu_*)-\ka_*(t_2^2-|\xi'|^2)\}(i\xi'\cdot\ga'-t_2\ga_N)  \notag \\
&=-\la\wh{g}_N(0), \notag \\
&t_1\big(i\xi'\cdot\beta'-t_1\beta_N\big)
+t_2\big(i\xi'\cdot\ga' -t_2\ga_N \big)
= -\la\wh h(0). \label{eq:56}
\end{align}

We now derive simultaneous equations with respect to $\beta_N$ and $\ga_N$.
By \eqref{eq:51} and \eqref{eq:54}, we have
\begin{align}
&i\xi_j\al_N -\om_\la \al_j
= -\mu_*^{-1}\wh g_j-t_1^{-1}i\xi_j(t_1-\om_\la)\beta_N-t_2^{-1}i\xi_j(t_2-\om_\la)\ga_N. \label{eq:59}
\end{align}
Since \eqref{eq:55} can be written as
\begin{align*}
&2\mu_*\la\{-\om_\la\al_N + (-t_1+\om_\la)\beta_N + (-t_2 +\om_\la)\ga_N\}  \\
&\quad+\{\la\nu_*-\ka_*(t_1^2-|\xi'|^2)\}(i\xi'\cdot\beta'-t_1\beta_N)  \notag\\
&\quad+\{\la\nu_*-\ka_*(t_2^2-|\xi'|^2)\}(i\xi'\cdot\ga'-t_2\ga_N)  \notag \\
&\quad -\la\mu_*(i\xi'\cdot\beta'-t_1\beta_N+i\xi'\cdot\ga'-t_2\ga_N) \notag \\
&=-\la\wh{g}_N(0), \notag 
\end{align*}
we observe that, by \eqref{eq:13}, \eqref{eq:14}, \eqref{eq:60}, and \eqref{eq:61},
\begin{align}\label{eq:63}
&2t_1t_2\om_\la\al_N  \\
&=\mu_*^{-1}t_1t_2\wh g_N(0) +t_2(2t_1\om_\la-\om_\la^2-|\xi'|^2)\beta_N
+t_1(2t_2\om_\la-\om_\la^2-|\xi'|^2)\ga_N.
\notag
\end{align}
By \eqref{eq:59} and \eqref{eq:63}, we have
\begin{align}\label{eq:64}
\al_j 
&= \frac{1}{\om_\la}\bigg\{\frac{\wh g_j(0)}{\mu_*} + \frac{i\xi_j}{2\mu_*\om_\la}\wh g_N(0)
+\frac{i\xi_j}{2t_1\om_\la}(4t_1\om_\la-3\om_\la^2-|\xi'|^2)\beta_N \\
&+\frac{i\xi_j}{2t_2\om_\la}(4t_2\om_\la-3\om_\la^2-|\xi'|^2)\ga_N\bigg\}. \notag
\end{align}
On the other hand, by \eqref{eq:54}, we have for $i\xi'\cdot\wh \Bg'(0)=\sum_{j=1}^{N-1}i\xi_j\wh g_j(0)$
\begin{equation*}
\om_\la i\xi'\cdot\al'
= \mu_*^{-1}i\xi'\cdot\wh \Bg'(0) -|\xi'|^2\al_N +(-t_1+\om_\la)i\xi'\cdot\beta' + (-t_2+\om_\la)i\xi'\cdot\ga',
\end{equation*}
which, inserted into \eqref{eq:10}, furnishes 
\begin{equation}\label{eq:71}
-t_1i\xi'\cdot\beta' - t_2i\xi'\cdot\ga' -(\om_\la^2+|\xi'|^2)\al_N+\om_\la^2\beta_N+\om_\la^2\ga_N
 = -\mu_*^{-1}i\xi'\cdot\wh \Bg'(0).
\end{equation}
Since it holds by \eqref{eq:51} that
$i\xi'\cdot\beta' = t_1^{-1}|\xi'|^2\beta_N$ and $i\xi'\cdot \ga' = t_2^{-1}|\xi'|^2\ga_N$,
we insert these identities and \eqref{eq:63} into \eqref{eq:71}
in order to obtain
\begin{align}\label{eq:65}
&t_2 \{(\om_\la^2+|\xi'|^2)^2-4t_1\om_\la|\xi'|^2\}\beta_N 
+t_1\{(\om_\la^2+|\xi'|^2)^2-4t_2\om_\la|\xi'|^2\}\ga_N  \\
&=-2\mu_*^{-1}t_1t_2\om_\la i\xi'\cdot\wh \Bg'(0) +\mu_*^{-1}t_1t_2(\om_\la^2+|\xi'|^2)\wh g_N(0).
\notag
\end{align}
Furthermore, by \eqref{eq:60}, \eqref{eq:61}, and \eqref{eq:56}, 
\begin{equation}\label{eq:66}
(t_1^2-|\xi'|^2)\beta_N +(t_2^2-|\xi'|^2)\ga_N = \la \wh h(0).
\end{equation}
Thus, by \eqref{eq:65} and \eqref{eq:66}, we have achieved
\begin{align}
&\BL 
\begin{pmatrix}
\beta_N \\ \ga_N
\end{pmatrix}
=
\begin{pmatrix}
-2\mu_*^{-1}t_1t_2\om_\la i\xi'\cdot\wh\Bg'(0) +\mu_*^{-1}t_1t_2(\om_\la^2+|\xi'|^2)\wh g_N(0) \\
\la \wh h(0)
\end{pmatrix},\label{SEs} \\
&\BL
=\begin{pmatrix}
t_2\{(\om_\la^2+|\xi'|^2)^2-4t_1\om_\la|\xi'|^2\} &
t_1\{(\om_\la^2+|\xi'|^2)^2-4t_2\om_\la|\xi'|^2\} \\
t_1^2-|\xi'|^2 & t_2^2-|\xi'|^2
\end{pmatrix}. \notag 
\end{align}

Finally, we solve \eqref{SEs} and the equations \eqref{eq:1}-\eqref{eq:6}.
By direct calculations, 
\begin{align*}
\det \BL &= 
(t_2-t_1)\{(\om_\la^2+|\xi'|^2)^2(t_2^2+t_2t_1+t_1^2-|\xi'|^2) \\ 
&-4t_1t_2\om_\la|\xi'|^2(t_2+t_1)\}.
\end{align*}
We here prove

\begin{lemm}\label{lemm:detL}
Assume that $\mu_*$, $\nu_*$, and $\ka_*$ satisfy 
the same assumption as in Theorem $\ref{theo:half}$. 
Then $\det\BL\neq 0$
for any $(\xi',\la)\in\BR^{N-1}\times(\overline{\BC_+}\setminus\{0\})$, where $\overline{\BC_+}=\{z\in\BC \mid \Re z\geq0\}$. 
\end{lemm}

\begin{proof}
The lemma is proved by contradiction.
Suppose that $\det\BL=0$ for some 
$(\xi',\la)\in\BR^{N-1}\times(\overline{\BC_+}\setminus\{0\})$.
Then we observe that there is $(\beta_N,\ga_N)\neq(0,0)$
satisfying \eqref{SEs} with $\Bg=0$ and $h=0$.
This implies that the equations \eqref{eq:1}-\eqref{eq:6} with $\Bg=0$ and $h=0$
admits a non-trivial solution sufficiently smooth and decaying exponentially as $x_N\to\infty$,
i.e. there exists $(\rho,u_1,\dots,u_N)\neq(0,0,\dots,0)$ such that,
for $x_N>0$ and $\ph(x_N) = \sum_{j=1}^{N-1}i\xi_j u_j(x_N)+\pa_N u_N(x_N)$,
\begin{align}
&\la\rho(x_N) +\ph(x_N) =0, \label{eq:100} \\ 
&\la u_j(x_N)-\mu_*(\pa_N^2-|\xi'|^2) u_j(x_N) -\nu_*i\xi_j \ph(x_N) \label{eq:101}\\
&\quad -\ka_*i\xi_j(\pa_N^2-|\xi'|^2)\rho(x_N)=0,  \quad j=1,\dots,N-1, \notag \\
&\la  u_N(x_N)-\mu_*(\pa_N^2-|\xi'|^2) u_N(x_N) -\nu_*\pa_N  \ph(x_N)\label{eq:102} \\
&\quad-\ka_*\pa_N(\pa_N^2-|\xi'|^2)\rho(x_N)=0, \notag \\
&\mu_*(i\xi_j u_N(0)+\pa_N u_j(0)) = 0, \quad j=1,\dots,N-1, \label{eq:103} \\
&2\mu_*\pa_N u_N(0)+(\nu_*-\mu_*) \ph(0)+\ka_*(\pa_N^2-|\xi'|^2)\rho(0) =0, \label{eq:104} \\
&\pa_N\rho(0) =0. \label{eq:105}
\end{align}
%
%
%
%

In this proof,  we set 
$(a,b)=\int_0^\infty a(x_N)\,\overline{b(x_N)}\intd x_N$
and $\|a\| =\sqrt{(a,a)}$
for functions $a=a(x_N)$, $b=b(x_N)$ on $\BR_+$,
and also
$\Phi=\|\pa_N \ph\|^2+|\xi'|^2\|\ph\|^2$.
{\bf Step 1.}
We prove in this step
\begin{align}\label{161102_5}
&\la\sum_{J=1}^N\|u_J\|^2+ \mu_*\sum_{j,k=1}^{N-1}\|i\xi_k u_j\|^2
+\mu_*\|\sum_{j=1}^{N-1}i\xi_j u_j\|^2
+2\mu_*\|\pa_N u_N\|^2 \\
&+\mu_*\sum_{j=1}^{N-1}\|\pa_N u_j+i\xi_j u_N\|^2 
+(\nu_*-\mu_*)\|\ph\|^2
+\frac{\ka_*}{|\la|^2}\overline{\la}\Phi
=0. \notag
\end{align}

It now holds that for $j=1,\dots,N-1$
\begin{align*}
&\mu_*(\pa_N^2-|\xi'|^2)u_j(x_N)+\nu_*i\xi_j\ph(x_N) 
=\mu_*\sum_{k=1}^{N-1}i\xi_k\left(i\xi_k u_j(x_N)+i\xi_j u_k(x_N)\right) \\
&\quad +\mu_*\pa_N\left(\pa_N u_j(x_N)+i\xi_j u_N(x_N)\right)
+(\nu_*-\mu_*)i\xi_j \ph(x_N), \\
&\mu_*(\pa_N^2-|\xi'|^2) u_N(x_N)+\nu_*\pa_N \ph(x_N)
=\mu_*\sum_{k=1}^{N-1}i\xi_k\left(i\xi_k u_N(x_N)+\pa_N u_k(x_N)\right) \\
&\quad +2\mu_*\pa_N^2 u_N(x_N)+(\nu_*-\mu_*)\pa_N \ph(x_N).
\end{align*}
Thus \eqref{eq:101} and \eqref{eq:102} can be written as 
\begin{align}
0&=\la u_j(x_N)-\mu_*\sum_{k=1}^{N-1}i\xi_k\left(i\xi_k u_j(x_N)+i\xi_j u_k(x_N)\right) \label{eq:106}\\
& -\mu_*\pa_N\left(\pa_N u_j(x_N)+i\xi_j u_N(x_N)\right) -(\nu_*-\mu_*)i\xi_j \ph(x_N) \notag \\
& -\ka_*i\xi_j(\pa_N^2-|\xi'|^2)\rho(x_N), \notag \\
0&=\la u_N(x_N)-\mu_*\sum_{k=1}^{N-1}i\xi_k\left(i\xi_k u_N(x_N)+\pa_N u_k(x_N)\right)
-2\mu_*\pa_N^2 u_N(x_N)\label{eq:107} \\
&-(\nu_*-\mu_*)\pa_N \ph(x_N)-\ka_*\pa_N(\pa_N^2-|\xi'|^2)\rho(x_N). \notag
\end{align}
Multiplying \eqref{eq:106} by $\overline{u_j(x_N)}$
and integrating the resultant formula with respect to $x_N\in(0,\infty)$,
we observe that, by integration by parts together with \eqref{eq:103},
\begin{align*}
0&=\la\|u_j\|^2 + \mu_*\sum_{k=1}^{N-1}(i\xi_k u_j+i\xi_j u_k,i\xi_k u_j)
+\mu_*(\pa_N u_j+i\xi_j u_N,\pa_N u_j) \\ 
&+(\nu_*-\mu_*)(\ph, i\xi_j u_j)
+\ka_*((\pa_N^2-|\xi'|^2)\rho,i\xi_j u_j).
\end{align*}
Analogously, it follows from \eqref{eq:104}, \eqref{eq:107} that
\begin{align*}
0&=\la\|u_N\|^2+\mu_*\sum_{k=1}^{N-1}(i\xi_k u_N+\pa_N u_k,i\xi_k u_N) +2\mu_*\|\pa_Nu_N\|^2 \\
&+(\nu_*-\mu_*)(\ph,\pa_N u_N)
+\ka_*((\pa_N^2-|\xi'|^2)\rho(x_N),\pa_N u_N).
\end{align*}
Summing these identities,
we see by integration by parts with \eqref{eq:105} that
\begin{align*}
&\la\sum_{J=1}^N\|u_J\|^2+ \mu_*\sum_{j,k=1}^{N-1}(i\xi_k u_j+i\xi_j u_k,i\xi_k u_j)
+\mu_*\sum_{j=1}^{N-1}(\pa_N u_j+i\xi_j u_N,\pa_N u_j) \\
&+\mu_*\sum_{k=1}^{N-1}(i\xi_k u_N+\pa_N u_k,i\xi_k u_N) +2\mu_*\|\pa_Nu_N\|^2 
+(\nu_*-\mu_*)\|\ph\|^2
+\frac{\ka_*}{\la}\Phi=0,
\end{align*}
where we have used \eqref{eq:100} in order to obtain the last term of the left-hand side.
Combining the last identity with
\begin{align*}
&\sum_{j,k=1}^{N-1}(i\xi_k u_j+i\xi_j u_k,i\xi_k u_j) 
=\sum_{j,k=1}^{N-1}\|i\xi_k u_j\|^2 +\|\sum_{j=1}^{N-1}i\xi_j u_j\|^2, \\
&\sum_{j=1}^{N-1}(\pa_N u_j+i\xi_j u_N,\pa_N u_j)+ \sum_{k=1}^{N-1}(i\xi_k u_N+\pa_N u_k,i\xi_k u_N) 
=\sum_{j=1}^{N-1}\|\pa_N u_j+i\xi_j u_N\|^2
\end{align*}
furnishes \eqref{161102_5}.

{\bf Step 2.}
We take the real part of \eqref{161102_5} and the imaginary part of \eqref{161102_5} in order to obtain
\begin{align}
0&=(\Re\la)\sum_{J=1}^N\|u_J\|^2+ \mu_*\sum_{j,k=1}^{N-1}\|i\xi_k u_j\|^2 +\mu_*\|\sum_{j=1}^{N-1}i\xi_j u_j\|^2
+2\mu_*\|\pa_N u_N\|^2 \label{eq:152} \\
&+\mu_*\sum_{j=1}^{N-1}\|\pa_N u_j+i\xi_j u_N\|^2 
+(\nu_*-\mu_*)\|\ph\|^2
+\frac{\ka_*}{|\la|^2}(\Re\la)\Phi, \notag \\
0&=(\Im\la)\left(\sum_{J=1}^N\| u_J\|^2-\frac{\ka_*}{|\la|^2}\Phi\right).
\label{eq:153}
\end{align}
In addition, we have, by $ab\leq (a^2+b^2)/2$ with $a,b\geq0$,
\begin{equation*}
\|\ph\|^2\leq \sum_{j,k=1}^{N-1}\|i\xi_k u_j\|^2+\|\sum_{j=1}^{N-1}i\xi_j u_j\|^2 + 2\|\pa_N u_N\|^2,
\end{equation*}
which, combined with \eqref{eq:152}, furnishes 
\begin{equation}\label{eq:270}
0\geq (\Re\la)\sum_{J=1}^N\| u_J\|^2
+\mu_*\sum_{j=1}^{N-1}\|\pa_N u_j+i\xi_j u_N\|^2
+\nu_*\|\ph\|^2+\frac{\ka_*}{|\la|^2}(\Re\la)\Phi.
\end{equation}

{\bf Step 3.}
Recall that $\la\in\overline{\BC_+}\setminus\{0\}$. 
It then holds by \eqref{eq:270} that $\ph=0$,
which, combined with \eqref{eq:100}, implies $\rho=0$.

Next we show $(u_1,\dots,u_N)=(0,\dots,0)$.
This follows from \eqref{eq:270} immediately when $\Re\la>0$.
When $\Re\la=0$, one notes $\Im\la\neq 0$.
It thus holds for $\Re\la=0$ that $(u_1,\dots,u_N)=(0,\dots,0)$ by \eqref{eq:153} and by $\Phi=0$ following from $\ph=0$.

Summing up the above results, we have $(\rho,u_1,\dots,u_N)=(0,0,\dots,0)$,
which contradicts the fact that $(\rho,u_1,\dots,u_N)$ is a non-trivial solution.
This completes the proof of the lemma.
\end{proof}

Let us write $\BL^{-1} $ as follows:
\begin{equation*}
\BL^{-1} = \frac{1}{\det\BL}
\begin{pmatrix}
L_{11} & L_{12} \\
L_{21} & L_{22}
\end{pmatrix},
\end{equation*}
where
\begin{alignat*}{2}
L_{11} &=
t_2^2-|\xi'|^2,  &&\quad
L_{12} =
-t_1\{(\om_\la^2+|\xi'|^2)^2-4t_2\om_\la|\xi'|^2\}, \\
L_{21} &=
-(t_1^2-|\xi'|^2), && \quad
L_{22}
=t_2\{(\om_\la^2+|\xi'|^2)^2-4t_1\om_\la|\xi'|^2\}.
\end{alignat*}
We thus see that, by solving \eqref{SEs},
\begin{align}\label{bega}
\beta_N &= - \frac{2 t_1t_2\om_\la L_{11}}{\mu_*\det\BL}i\xi'\cdot\wh \Bg'(0)
+\frac{t_1t_2(\om_\la^2+|\xi'|^2)L_{11}}{\mu_*\det\BL}\wh g_N(0)
+\frac{\la L_{12}}{\det\BL}\wh h(0), \\
\ga_N &=
-\frac{2t_1t_2\om_\la L_{21}}{\mu_*\det\BL}i\xi'\cdot\wh \Bg'(0)
+\frac{t_1t_2(\om_\la^2+|\xi'|^2)L_{21}}{\mu_*\det\BL}\wh g_N(0)
+\frac{\la L_{22}}{\det\BL}\wh h(0), \notag
\end{align}
which, combined with \eqref{eq:51}, \eqref{eq:63}, and \eqref{eq:64},
gives the exact formulas of $\al_j$, $\beta_j$, $\ga_j$, and $\al_N$ for $j=1,\dots,N-1$.
Hence we obtain
\begin{align*}
\wh\rho(x_N)
&=
\left(\frac{t_1^2-|\xi'|^2}{\la t_1}\right)e^{-t_1 x_N}\beta_N 
+\left(\frac{t_2^2-|\xi'|^2}{\la t_2}\right) e^{-t_2 x_N}\ga_N, \\
\wh u_j(x_N)
&=
\frac{1}{\mu_*\om_\la} e^{-\om_\la x_N}\wh g_j(0)
+\frac{i\xi_j}{2\mu_*\om_\la^2}e^{-\om_\la x_N} \wh g_N(0) \notag \\
&+\frac{i\xi_j(4t_1\om_\la -3\om_\la^2-|\xi'|^2)}{2t_1 \om_\la^2}e^{-\om_\la x_N}\beta_N  \\
&+\frac{i\xi_j(4t_2\om_\la -3\om_\la^2-|\xi'|^2)}{2t_2 \om_\la^2}e^{-\om_\la x_N}\ga_N \notag \\
& - \frac{i\xi_j}{t_1}(e^{-t_1 x_N}-e^{-\om_\la x_N})\beta_N 
 - \frac{i\xi_j}{t_2}(e^{-t_2 x_N}-e^{-\om_\la x_N})\ga_N, \\ 
\wh u_N(x_N) &=
\frac{1}{2\mu_*\om_\la}e^{-\om_\la x_N} \wh g_N(0)  \\  
& +\left(\frac{2t_1\om_\la -\om_\la^2-|\xi'|^2}{2t_1\om_\la}\right)e^{-\om_\la x_N}\beta_N \\ 
&+\left(\frac{2t_2\om_\la -\om_\la^2-|\xi'|^2}{2t_2\om_\la}\right)e^{-\om_\la x_N}\ga_N \notag \\
& +\left(e^{-t_1 x_N}-e^{-\om_\la x_N}\right)\beta_N  
+\left(e^{-t_2 x_N}-e^{-\om_\la x_N}\right)\ga_N , \notag
\end{align*}
where we have used \eqref{eq:1}, \eqref{eq:60}, \eqref{eq:61}, and \eqref{eq:50}
in order to derive the representation formula of $\rho$.
Setting $\rho=\CF_{\xi'}^{-1}[\wh\rho(x_N)](x')$ and $u_J=\CF_{\xi'}^{-1}[\wh u_J(x_N)](x')$ $(J=1,\dots,N)$,
we see that $\rho$ and $\Bu=(u_1,\dots, u_N)^\SST$ solve the system \eqref{eq:half1}.

\section{Technical lemmas\label{sec:tec}}
In this section, we prove technical lemmas that are used in Section \ref{sec:proof1} below
in order to show the existence of $\CR$-bounded solution operator families associated with
the solution $(\rho,\Bu)$ of \eqref{eq:half1} obtained in Section \ref{sec:half}.
One assumes that $\mu_*$, $\nu_*$, and $\ka_*$ 
satisfy the same assumption as in Theorem \ref{theo:half}
and use the symbols introduced in Section \ref{sec:half} in the following.

Let $\de\geq 0$ and $\si\in(0,\pi/2)$, 
and let $m(\xi',\la)$ be a function defined on $(\BR^{N-1}\setminus\{0\})\times \Si_{\si,\de}$
that is infinitely many times differentiable with respect to $\xi'$
and analytic with respect to $\la$.
If there exists a real number $s$ such that
for any multi-index $\al'=(\al_1,\dots,\al_{N-1})\in\BN_0^{N-1}$ 
and $(\xi',\la)\in(\BR^{N-1}\setminus\{0\})\times\Si_{\si,\de}$
\begin{equation*}
\left|\pa_{\xi'}^{\al'}\left(\left(\la\frac{d}{d\la}\right)^n m(\xi',\la)\right)\right|
\leq C(|\la|^{1/2}+|\xi'|)^{s-|\al'|} \quad (n=0,1)
\end{equation*}
with some positive constant $C=C_{s,\al',\de,\si,\mu_*,\nu_*,\ka_*}$,
then $m(\xi',\la)$ is called a multiplier of order $s$ with type $1$.
If there exists a real number $s$ such that
for any multi-index $\al'=(\al_1,\dots,\al_{N-1})\in\BN_0^{N-1}$ 
and $(\xi',\la)\in(\BR^{N-1}\setminus\{0\})\times\Si_{\si,\de}$
\begin{equation*}
\left|\pa_{\xi'}^{\al'}\left(\left(\la\frac{d}{d\la}\right)^n m(\xi',\la)\right)\right|
\leq C(|\la|^{1/2}+|\xi'|)^{s}|\xi'|^{-|\al'|} \quad (n=0,1)
\end{equation*}
with some positive constant $C=C_{s,\al',\de,\si,\mu_*,\nu_*,\ka_*}$,
then $m(\xi',\la)$ is called a multiplier of order $s$ with type $2$.
In what follows, we denote the set of all multipliers of order $s$ with type $l$ $(l=1,2)$ on 
$(\BR^{N-1}\setminus\{0\})\times\Si_{\si,\de}$ by $\BBM_{\si,\de}^l(s)$.
For instance, 
$\xi_j/|\xi'|\in \BBM_{\si,0}^2(0)$ and
$\xi_j,\la^{1/2}\in\BBM_{\si,0}^1(1)$ for $j=1,\dots,N-1$,
and also $|\xi'|^2,\la \in\BBM_{\si,0}^1(2)$.

\begin{rema}
The sets $\BBM_{\si,\de}^l(s)$ are vector spaces on $\BC$.
In addition, $\BBM_{\si,\de_1}^l(s)\subset \BBM_{\si,\de_2}^l(s)$ for $0\leq \de_1\leq \de_2$,
and especially $\BBM_{\si,0}^l(s)\subset \BBM_{\si,\de}^l(s)$ for any $\de\geq 0$.
\end{rema}


At this point, we introduce two lemmas concerning multipliers.
The first lemma was proved by \cite[Lemma 5.1]{SS12},
and the second one by \cite[Lemma 5.4]{SS12}.

\begin{lemm}\label{lemm:algebra}
Let $r_1,r_2\in\BR$, $\de\geq 0$, and $\si\in(0,\pi/2)$. 
\begin{enumerate}[$(1)$]
\item
Given $l_j\in\BBM_{\si,\de}^1(r_j)$ $(j=1,2)$,
we have $l_1l_2\in \BBM_{\si,\de}^1(r_1+r_2)$.
\item
Given $m_j\in\BBM_{\si,\de}^j(r_j)$ $(j=1,2)$,
we have $m_1m_2\in\BBM_{\si,\de}^2(r_1+r_2)$.
\item
Given $n_j\in\BBM_{\si,\de}^2(r_j)$ $(j=1,2)$,
we have $n_1n_2\in\BBM_{\si,\de}^2(r_1+r_2)$.
\end{enumerate}
\end{lemm}

\begin{lemm}\label{lemm:R-bound0}
Let $q\in(1,\infty)$, $\de\geq 0$, and $\si\in(0,\pi/2)$. 
For
$k(\xi',\la)\in \BBM_{\si,\de}^1(0)$ and 
$l(\xi',\la)\in\BBM_{\si,\de}^2(0)$,
we define 
\begin{align*}
[K(\la)f] (x)
&=\int_0^\infty\CF_{\xi'}^{-1}\left[k(\xi',\la)\la^{1/2}e^{-\om_\la(x_N+y_N)}\wh f(\xi',y_N)\right](x')\intd y_N, \\
[L(\la)f](x)
&=\int_0^\infty\CF_{\xi'}^{-1}\left[l(\xi',\la)|\xi'|e^{-\om_\la(x_N+y_N)}\wh f(\xi',y_N)\right](x')\intd y_N,
\end{align*}
with $\la\in\Si_{\si,\de}$. 
Then the sets 
\begin{align*}
\{\la^n(d/d\la)^n K(\la)\mid \la\in\Si_{\si,\de}\}, \quad
\{\la^n(d/d\la)^n L(\la)\mid\la\in\Si_{\si,\de}\}
\end{align*}
are $\CR$-bounded on $\CL(L_q(\BR_+^N))$ for $n=0,1$,
and also their $\CR$-bounds do not exceed some constant $C_{N,q,\de,\si,\mu_*,\nu_*,\ka_*}$.
\end{lemm}

Let $N_1=N+1$ and $N_2=N^2+N+1$ in what follows,
and note that
\begin{align}\label{170427_9}
\om_\la^2&=|\xi'|^2+\frac{\la}{\mu_*}
=-\sum_{l=1}^{N-1}(i\xi_l)^2+\frac{\la}{\mu_*}, \\
t_k&=|\xi'|^2+s_k\la=-\sum_{l=1}^{N-1}(i\xi_l)^2+s_k\la \quad (k=1,2),
\notag
\end{align}
where $\om_\la$ and $t_k$ are defined as \eqref{omega} and \eqref{t1t2}, respectively.
Then we have 

\begin{coro}\label{coro:R-bound0}
Let $q\in(1,\infty)$ and $\si\in(0,\pi/2)$, and let $j=1,2$.
For $k^j(\xi',\la)\in\BBM_{\si,0}^1(j-2)$, we define operators $K^j(\la)$ by
\begin{equation*}
[K^j(\la)f](x)
=\CF_{\xi'}^{-1}\left[k^j(\xi',\la)e^{-\om_\la x_N}\wh f(\xi',0)\right](x'),
\end{equation*}
with $\la\in\Si_{\si,0}$ and $f\in W_q^j(\BR_+^N)$.
Then there exists operators $\wt K^j(\la)$, with
\begin{equation*}
\wt K^j(\la) \in \Hol(\Si_{\si,0},\CL(L_q(\BR_+^N)^{N_j},W_q^2(\BR_+^N))),
\end{equation*} 
such that for any $g\in W_q^1(\BR_+^N)$ and any $h\in W_q^2(\BR_+^N)$
\begin{equation*}
K^1(\la)g = \wt K^1(\la)(\nabla g,\la^{1/2}g), \quad
K^2(\la)h =\wt K^2(\la)(\nabla^2 h,\la^{1/2}\nabla h,\la h),
\end{equation*}
and also for $n=0,1$
\begin{equation*}
\CR_{\CL(L_q(\BR_+^N)^{N_j},\FB_q(\BR_+^N))}
\big(\big\{\la^n\big(d/d\la\big)^n\big(\CT_\la\wt K^j(\la)\big) \mid 
\la\in\Si_{\si,0}\big\}\big)\leq C
\end{equation*}
with a positive constant $C=C_{N,q,\si,\mu_*,\nu_*,\ka_*}$,
where $\FB_q(\BR_+^N)$, $\CT_\la$ are given in \eqref{ABset}.
\end{coro}

\begin{proof}
We only prove the case $j=1$.

For functions $a(x_N)$, $b(x_N)$ $(x_N\geq 0)$
with $a(x_N+y_N)b(y_N)\to 0$ as $y_N\to\infty$,
we observe that
\begin{align}\label{vole}
&a(x_N)b(0) = -\int_0^\infty\frac{d}{dy_N}\left(a(x_N+y_N)b(y_N)\right)\intd y_N \\
&=-\int_0^\infty(\pa_N a)(x_N+y_N)b(y_N)\intd y_N 
-\int_0^\infty a(x_N+y_N)(\pa_N b)(y_N)\intd y_N.
\notag
\end{align}
By this relation and \eqref{170427_9}, 
\begin{align*}
[K^1(\la)g](x)&=
\int_0^\infty \CF_{\xi'}^{-1}
\left[\frac{k^1(\xi',\la)}{\mu_*\om_\la}\la^{1/2}e^{-\om_\la(x_N+y_N)}\wh{\la^{1/2}g}(\xi',y_N)\right](x')\intd y_N \\
&-\sum_{l=1}^{N-1}\int_0^\infty
\CF_{\xi'}^{-1}\left[
\frac{i\xi_l k^1(\xi',\la)}{|\xi'|\om_\la}|\xi'| e^{-\om_\la(x_N+y_N)}\wh {\pa_l g}(\xi',y_N)\right](x')\intd y_N \\
&-\int_0^\infty \CF_{\xi'}^{-1}
\left[\frac{\la^{1/2} k^1(\xi',\la)}{\mu_*\om_\la^2}\la^{1/2}e^{-\om_\la(x_N+y_N)}\wh{\pa_N g}(\xi',y_N)\right](x')\intd y_N \\
&-\sum_{l=1}^{N-1}
\int_0^\infty\CF_{\xi'}^{-1}\left[\frac{|\xi'|k^1(\xi',\la)}{\om_\la^2}|\xi'| e^{-\om_\la(x_N+y_N)}\wh{\pa_N g}(\xi',y_N)\right](x')\intd y_N \\
&=:\wt K^1(\la)(\nabla g,\la^{1/2}g).
\end{align*}
We combine Lemma \ref{lemm:algebra} and the assumption for $k^1(\xi',\la)$ with
\begin{equation}\label{170426_1}
\om_\la^s \in\BBM_{\si,0}^1(s) \quad \text{for any $s\in\BR$}
\end{equation}
(cf. \cite[Lemma 5.2]{SS12}) in order to obtain
\begin{align*}
\frac{k^1(\xi',\la)}{\mu_*\om_\la}, \frac{\la^{1/2} k^1(\xi',\la)}{\mu_*\om_\la^2}\in\BBM_{\si,0}^1(-2), \quad
\frac{i\xi_l k^1(\xi',\la)}{|\xi'|\om_\la}, \frac{|\xi'|k^1(\xi',\la)}{\om_\la^2}\in\BBM_{\si,0}^2(-2).
\end{align*}
Thus Lemmas \ref{lemm:algebra}, \ref{lemm:R-bound0} and \eqref{170426_1} yield
that $\CT_\la\wt K^1(\la)$ has the required properties of Corollary \ref{coro:R-bound0}. 
This completes the proof of the corollary.
\end{proof}

\subsection{Analysis of symbols\label{subsec:5_2}}
In this subsection, we estimate several symbols appearing
in the representation formulas obtained in Section \ref{sec:half}.

\begin{lemm}\label{est:1}
Let $\si_1\in(0,\pi/2)$ and $\si_2\in(\si_*^w,\pi/2)$.
\begin{enumerate}[$(1)$]
\item\label{est:1_1}
There is a positive constant $C_{\mu_*}$ such that
\begin{equation*}
\Re\om_\la\geq C_{\mu_*}\{\sin(\si_1/2)\}^{3/2}(|\la|^{1/2}+|\xi'|),
\quad (\xi',\la)\in\BR^{N-1}\times\Si_{\si_1,0}.
\end{equation*}
\item\label{est:1_2}
There is a positive constant $C_{\si_2,\mu_*,\nu_*,\ka_*}$ such that
for $j=1,2$
\begin{equation*}
\Re t_j\geq C_{\si_2,\mu_*,\nu_*,\ka_*}(|\la|^{1/2}+|\xi'|),
\quad (\xi',\la)\in\BR^{N-1}\times\Si_{\si_2,0}.
\end{equation*}
\end{enumerate}
\end{lemm}

\begin{proof}
\eqref{est:1_1}. 
Let $\mu_*^{-1}\la+|\xi'|^2=|\mu_*^{-1}\la+|\xi'|^2|e^{i\te}$.
Then $-(\pi-\si_1)\leq \te\leq  \pi-\si_1$, and
\begin{equation*}
\Re\om_\la =|\mu_*^{-1}\la+|\xi'|^2|^{1/2}\cos(\te/2)
\geq \min(1,\mu_*^{-1/2})|\la+|\xi'|^2|^{1/2}\sin(\si_1/2).
\end{equation*}
Combining this inequality with \eqref{160910_1} yields the required inequality,
which completes the proof of Lemma \ref{est:1} \eqref{est:1_1}.

\eqref{est:1_2}.
Note that $s_j\la\in\Si_{\si_2-\si_*^w,0}$ $(j=1,2)$ with $\si_2-\si_*^w\in(0,\pi/2)$.
Since $t_j=\om_z$ for $z=\mu_*s_j\la\in\Si_{\si_2-\si_*^w,0}$,
we have the required inequality by \eqref{est:1_1}.
This completes the proof of Lemma \ref{est:1} \eqref{est:1_2}.
\end{proof}

Let us define the following symbols: for $j=1,2$, 
\begin{align*}
\Fm_j(\xi',\la)
&=
\la^{-1}(t_j+\om_\la)\{(\om_\la^2+|\xi'|^2)^2-4t_j\om_\la|\xi'|^2\}, \\
\Fp_j(\xi',\la)
&=
\la^{-1}(t_j+\om_\la)(4t_j\om_\la-3\om_\la^2-|\xi'|^2), \notag \\
\Fq_j(\xi',\la)
&=\la^{-1}(t_j+\om_\la)(2t_j\om_\la -\om_\la^2-|\xi'|^2). \notag
\end{align*}
On the other hand, we have
\begin{align*}
(\om_\la^2+|\xi'|^2)^2-4t_j\om_\la|\xi'|^2
&=(\om_\la^2-|\xi'|^2)^2-4|\xi'|^2\om_\la(t_j-\om_\la), \\
4t_j\om_\la-3\om_\la^2-|\xi'|^2
&=4\om_\la(t_j-\om_\la)+\om_\la^2-|\xi'|^2, \\
2t_j\om_\la -\om_\la^2-|\xi'|^2
&=2\om_\la(t_j-\om_\la)+\om_\la^2-|\xi'|^2,
\end{align*}
which, combined with $\om_\la^2-|\xi'|^2=\mu_*^{-1}\la$ and
$t_j^2-\om_\la^2=(s_j-\mu_*^{-1})\la$ following from \eqref{170427_9}, furnishes that
\begin{align*}
\Fm_j(\xi',\la)
&=\mu_*^{-2}\la(t_j+\om_\la)-4(s_j-\mu_*^{-1})|\xi'|^2\om_\la, \\
\Fp_j(\xi',\la)
&=(4s_j-3\mu_*^{-1})\om_\la+\mu_*^{-1}t_j, \notag \\
\Fq_j(\xi',\la)
&=(2s_j-\mu_*^{-1})\om_\la+\mu_*^{-1}t_j. \notag
\end{align*}
Analogously, $\det\BL$ can be written as
\begin{equation}\label{170321_3}
\det\BL = \frac{\la(t_2-t_1)\Fl_1(\xi',\la)}{t_1(t_1+\om_\la)} = \frac{\la(t_2-t_1)\Fl_2(\xi',\la)}{t_2(t_2+\om_\la)},
\end{equation}
where we have set
\begin{align*}
\Fl_j(\xi',\la)
&=\mu_*^{-2}\la t_j(t_j+\om_\la)(t_2^2+t_2t_1+t_1^2-|\xi'|^2) \\
&+4\om_\la|\xi'|^2
\{s_jt_j\om_\la(t_j+\om_\la)-(s_j-\mu_*^{-1})t_1t_2(t_2+t_1)\}.
\end{align*}

Now we prove

\begin{lemm}\label{lemm:symbols}
There is a constant $\si_*\in(\si_*^w,\pi/2)$ such that,
for any  $\si\in(\si_*,\pi/2)$,
there exists a positive constant $C_{\si,\mu_*,\nu_*,\ka_*}$ such that,
for any $\la\in\Si_{\si,0}$ and $\xi'\in\BR^{N-1}$, we have for $j=1,2$
\begin{equation}\label{170315_1}
|\Fl_j(\xi',\la)|\geq C_{\si,\mu_*,\nu_*,\ka_*}(|\la|^{1/2}+|\xi'|)^6. 
\end{equation}
\end{lemm}

\begin{proof}
Let $\wt\si$ be any number of $(\si_*^w,\pi/2)$.

First, we consider the case $|\xi'|^2/|\la|\leq R_1$ for $(\xi',\la)\in\BR^{N-1}\times\Si_{\wt\si,0}$
and for some sufficiently small positive number $R_1$ determined below.
Let $z=|\xi'|^2/\la$. Then, 
\begin{equation*}
t_j = \sqrt{s_j\la}\left(1+O(z)\right), \quad 
\om_\la = \sqrt{\la/\mu_*}(1+O(z)) \quad  \text{as $|z|\to 0$,} 
\end{equation*}
which implies that
\begin{equation*}
\Fl_j(\xi',\la)
=\mu_*^{-2}\la^3\sqrt{s_j}(\sqrt{s_j}+\sqrt{1/\mu_*})(s_2+\sqrt{s_2s_1}+s_1)(1+O(z)).
\end{equation*}
%
%
Here we note that, by $s_1+s_2=(\mu_*+\nu_*)/\ka_*$ and $s_1s_2=\ka_*^{-1}$,
\begin{equation*}
s_2+\sqrt{s_2s_1}+s_1
=\frac{\mu_*+\nu_*}{\ka_*}+\frac{1}{\sqrt{\ka_*}}>0.
\end{equation*}
Thus there exists a constant $R_1\in(0,1)$
such that, for any $(\xi',\la)\in\BR^{N-1}\times\Si_{\wt\si,0}$ with $|\xi'|^2/|\la|\leq R_1$,
we have \eqref{170315_1}.

Next, we consider the case $|\la|/|\xi'|^2\leq R_2$ for $(\xi',\la)\in\BR^{N-1}\times\Si_{\wt\si,0}$ and
for some sufficiently small positive number $R_2$ determined below. 
Let $y=\la/|\xi'|^2$. Then,
\begin{equation*}
\om_\la=|\xi'|(1+O(y)), \quad t_j = |\xi'|(1+O(y)) \quad \text{as $|y|\to 0$,} 
\end{equation*}
which yields that
$\Fl_j(\xi',\la)=8\mu_*^{-1}|\xi'|^6(1+O(y))$.
Hence there is a constant $R_2\in(0,1)$
such that, for any $(\xi',\la)\in\BR^{N-1}\times\Si_{\wt\si,0}$ with $|\la|/|\xi'|^2\leq R_2$,
we have \eqref{170315_1}.

Finally, we consider the case $|\xi'|^2/|\la|\geq R_1/2$ and $|\la|/|\xi'|^2\geq R_2/2$, i.e.
\begin{equation}\label{161118_1}
\frac{R_2}{2} |\xi'|^2 \leq |\la| \leq \frac{2}{R_1}|\xi'|^2.
\end{equation}
Let $\wt\xi'$, $\wt\la$, $\wt t_j$, and $\wt\om_\la$ be given by
\begin{align*}
&\wt\xi'=(|\la|^{1/2}+|\xi'|)^{-1}\xi',\quad
\wt\la = (|\la|^{1/2}+|\xi'|)^{-2}\la, \\
&\wt{t_j}=\sqrt{|\wt\xi'|^2+s_j\wt\la}, \quad
\wt\om_\la=\sqrt{|\wt\xi'|^2+\mu_*^{-1}\wt\la}.
\end{align*}
Then we observe that
\begin{equation}\label{170321_1}
\Fl_j(\xi',\la)=(|\la|^{1/2}+|\xi'|)^6\Fl_j(\wt\xi',\wt\la)
\end{equation}
and that \eqref{161118_1} is equivalent to
the condition: $r_1 \leq |\wt \xi'| \leq r_2$, $r_3 \leq |\wt \la| \leq r_4$, where
\begin{alignat*}{2}
r_1 &=\left(\sqrt{\frac{2}{R_1}}+1\right)^{-1}, \quad
&r_2 &=\left(\sqrt{\frac{R_2}{2}}+1\right)^{-1}, \\
r_3 &= \left(\sqrt{\frac{2}{R_2}}+1\right)^{-2}, \quad
&r_4 &=\left(\sqrt{\frac{R_1}{2}}+1\right)^{-2}.
\end{alignat*}
We define a compact set $K_{\wt\si}$ by
\begin{equation*}
K_{\wt\si} = \{(\wt\xi',\wt\la)\in \BR^{N-1}\cap\overline{\Si_{\wt\si,0}}
\mid r_1\leq |\wt\xi'|\leq r_2, r_3\leq |\wt\la|\leq r_4\}.
\end{equation*}
%
%
%
%
%
%
%
%
%
%
%
Since $\Fl_j(\xi',\la)$ are continuous functions on
$\BR^{N-1}\times(\overline{\Si_{\wt\si,0}}\setminus\{0\})$
and since $\Fl_j(\xi',\la)\neq 0$ on $\BR^{N-1}\times(\overline{\BC_+}\setminus\{0\})$
by \eqref{170321_3} and Lemma \ref{lemm:detL}, 
there exist constants $C_{\mu_*,\nu_*,\ka_*}>0$ and $\si_*\in(\si_*^w,\pi/2)$ such that
$|\Fl_j(\wt\xi',\wt\la)|\geq C_{\mu_*,\nu_*,\ka_*}$ for any $(\wt\xi',\wt\la)\in K_{\si_*}$.
Combining this inequality with \eqref{170321_1} yields that
\eqref{170315_1} holds for any $(\xi',\la)\in \BR^{N-1}\times\Si_{\si,0}$ with $\si\in(\si_*,\pi/2)$ and \eqref{161118_1}.
This completes the proof of Lemma \ref{lemm:symbols}. 
%
%
%
%
%
%
\end{proof}

By Lemma \ref{est:1}, we have for $j=1,2$
\begin{align*}
&|t_j|\geq \Re t_j \geq C_{\si,\mu_*,\nu_*,\ka_*}(|\la|^{1/2}+|\xi'|), \\
&|t_j+\om_\la|\geq \Re(t_j+\om_\la)\geq C_{\si,\mu_*,\nu_*,\ka_*}(|\la|^{1/2}+|\xi'|),
\end{align*}
for any $\la\in\Si_{\si,0}$ $(\si\in(\si_*^w,\pi/2))$ and $\xi'\in\BR^{N-1}$.
Together with these inequalities and \eqref{170426_1},
we can prove the following corollary
by Bell's formula \eqref{Bell} and Lemma \ref{lemm:symbols}
in the same manner as we have obtained Corollary \ref{coro:P} from Lemma \ref{lemm:P}.

\begin{coro}\label{coro:symbols}
Let $s\in\BR$, $\si_2\in(\si_*^w,\pi/2)$, and $\si_3\in(\si_*,\pi/2)$
for the same constant $\si_*$ as in Lemma $\ref{lemm:symbols}$. Then, for $j=1,2$,
\begin{equation*}
t_j^s, (t_j+\om_\la)^s\in\BBM_{\si_2,0}^1(s), \quad
\Fl_j(\xi',\la)^s\in \BBM_{\si_3,0}^{1}(6s),
\end{equation*}
and also
\begin{equation*}
\Fm_j(\xi',\la)\in \BBM_{\si_2,0}^1(3), \quad \Fp_j(\xi',\la),\Fq_j(\xi',\la)\in\BBM_{\si_2,0}^1(1).
\end{equation*}
\end{coro}

\subsection{Technical lemmas\label{subsec:5_3}}
This subsection introduces some technical lemmas that
play an important role to show the existence of $\CR$-bounded solution operator families.
Let $\CM_j(x_N)$ $(j=0,1,2)$ be given by
\begin{equation}\label{defi:M}
\CM_0(x_N) = \frac{e^{-t_2 x_N}-e^{-t_1 x_N}}{t_2 - t_1}, \quad
\CM_j(x_N) = \frac{e^{- t_j x_N}-e^{-\om_\la x_N}}{t_2-t_1} \quad (j=1,2).
\end{equation}
Then we have

\begin{lemm}\label{lemm:R-bound1}
Let $q\in(1,\infty)$, $\de\geq 0$, $\si_1\in(0,\pi/2)$, and $\si_2\in(\si_*^w,\pi/2)$.
For
\begin{equation*}
k(\xi',\la)\in\BBM_{\si_1,\de}^1(1), \quad  l(\xi',\la)\in\BBM_{\si_2,\de}^1(1), \quad
m(\xi',\la)\in\BBM_{\si_2,\de}^1(2),
\end{equation*}
we define operators $K(\la)$, $L(\la)$, and $M_j(\la)$ $(j=0,1,2)$ by
\begin{alignat*}{2}
[K(\la)f](x)
&=\int_0^\infty
\CF_{\xi'}^{-1}\left[k(\xi',\la)e^{-\om_\la(x_N+y_N)}\wh f(\xi',y_N)\right](x')\intd y_N 
&& \quad (\la\in\Si_{\si_1,\de}), \\
[L(\la)f](x)
&=\int_0^\infty
\CF_{\xi'}^{-1}\left[l(\xi',\la)e^{-t_1(x_N+y_N)}\wh f(\xi',y_N)\right](x')\intd y_N 
&& \quad (\la\in\Si_{\si_2,\de}), \\
[M_j(\la)f](x)
&=\int_0^\infty
\CF_{\xi'}^{-1}\left[m(\xi',\la)\CM_j(x_N+y_N)\wh f(\xi',y_N)\right](x')\intd y_N
&&\quad (\la\in\Si_{\si_2,\de}).
\end{alignat*}
Then the sets
$\{\la^n(d/d\la)^n K(\la) \mid \la\in\Si_{\si_1,\de}\}$, 
$\{\la^n(d/d\la)^n L(\la) \mid \la\in\Si_{\si_2,\de}\}$, 
and $\{\la^n(d/d\la)^n M_j(\la) \mid \la\in\Si_{\si_2,\de}\}$
are $\CR$-bounded on $\CL(L_q(\BR_+^N))$ for $n=0,1$ and $j=0,1,2$.
In addition, the $\CR$-bound of the first set is bounded above by some constant $C_{N,q,\de,\si_1,\mu_*,\nu_*,\ka_*}$,
while the $\CR$-bounds of the other sets are bounded above by some constant $C_{N,q,\de,\si_2,\mu_*,\nu_*,\ka_*}$.
\end{lemm}

\begin{proof}
{\bf Case $K(\la)$.}
The proof is similar to Corollary \ref{coro:R-bound0},
so that we may omit the detailed proof.

{\bf Case $L(\la)$.}
As was seen in the proof of Lemma \ref{est:1}, 
we know that $t_1=\om_z$ for $z=\mu_* s_1\la\in\Si_{\si_2-\si_2^w,0}$ with $\si_2-\si_*^w\in(0,\pi/2)$.
Thus the required properties follow from the above result of $K(\la)$ and Definition \ref{defi:R} immediately.

{\bf Case $M_0(\la)$.}
Following Shibata and Shimizu \cite[Lemma 5.4]{SS12},
we prove the $\CR$-boundedness 
by means of \cite[Theorem 2.3]{SS01} and \cite[Proposition 3.3]{DHP03}.

Let $\al'$ be any  multi-index of $\BN_0^{N-1}$ and $n=0,1$ in this proof.
We then note that for $r_1,r_2\in\{\om_\la,t_1,t_2\}$
and for non-negative real numbers $a_1$, $a_2$
\begin{align}
&\big|\pa_{\xi'}^{\al'}\{\la^n(d/d\la)^n(e^{-a_1 r_1}e^{-a_2 r_2})\}\big| \label{170429_2} \\
&\leq c(|\la|^{1/2}+|\xi'|)^{-|\al'|}e^{-b(|\la|^{1/2}+|\xi'|)(a_1+a_2)},
\notag
\end{align}
for any $(\xi',\la)\in(\BR^{N-1}\setminus\{0\})\times\Si_{\si_2,0}$ (cf. \cite[Lemma 5.2]{Saito15b}).
Here and subsequently,
$c$ is a positive constant depending on at most $\al'$, $\de$, $\si_2$, $\mu_*$, $\nu_*$, and $\ka_*$,
while $b$ is a positive constant depending on at most $\si_2$, $\mu_*$, $\nu_*$, and $\ka_*$.
They especially are independent of  $\xi'$, $\la$, $a_1$, and $a_2$.


By using \eqref{170427_9}, we rewrite $[M_0(\la)f](x)$ as 
\begin{align*}
[M_0(\la)f](x)
&=\int_0^\infty
\CF_{\xi'}^{-1}\left[\frac{m(\xi',\la)}{\mu_*\om_\la^2}
\la \CM_0(x_N+y_N)\wh f(\xi',y_N)\right](x') \\
&+\int_0^\infty
\CF_{\xi'}^{-1}\left[\frac{m(\xi',\la)}{\om_\la^2}
|\xi'|^2\CM_0(x_N+y_N)\wh f(\xi',y_N)\right](x') \\
&=:[M_{0,1}(\la)f](x)+[M_{0,2}(\la)f](x).
\end{align*}
Let $k_{1\la}(x)$, $k_{2\la}(x)$ be given by
\begin{alignat*}{2}
k_{1\la}(x)&= \CF_{\xi'}^{-1}\left[m_1(\xi',\la)\la\CM_0(x_N)\right](x') \quad
&&\text{for } m_1(\xi',\la)=\frac{m(\xi',\la)}{\mu_*\om_\la^2}, \\
k_{2\la}(x)&= \CF_{\xi'}^{-1}\left[m_2(\xi',\la)|\xi'|^2\CM_0(x_N)\right](x') \quad
&&\text{for } m_2(\xi',\la)=\frac{m(\xi',\la)}{\om_\la^2},
\end{alignat*}
and then
\begin{equation}\label{170324_1}
[M_{0,l}(\la)f](x) =\int_{\uhs}k_{l \la}(x'-y',x_N+y_N)f(y)\intd y \quad (l=1,2)
\end{equation}
and $m_l(\xi',\la)\in\BBM_{\si_2,\de}^1(0)$ $(l=1,2)$ by \eqref{170426_1} and the assumption for $m(\xi',\la)$.
Since $\CM_0(x_N)=-x_N\int_0^1e^{-\{\te t_2+(1-\te)t_1\}x_N}\intd\te$,
we have, by \eqref{170429_2} with $r_1=t_1$, $r_2=t_2$, $a_1=\te x_N$, and $a_2=(1-\te)x_N$,
\begin{equation}\label{170429_5}
|\pa_{\xi'}^{\al'}\{\la^n(d/d\la)^n\CM_0(x_N)\}|
\leq c x_N(|\la|^{1/2}+|\xi'|)^{-|\al'|}e^{-b(|\la|^{1/2}+|\xi'|)x_N}.
\end{equation}

First we estimate $k_{1\la}(x)$.
By \eqref{170429_5}, $m_1(\xi',\la)\in\BBM_{\si_2,\de}^1(0)$, and Leibniz's rule,
\begin{align}\label{160830_1}
&\big|\pa_{\xi'}^{\al'}\big\{\la^n(d/d\la)^n\big(m_1(\xi',\la)\la\CM_0(x_N)\big)\big\}\big|   \\
&\leq c x_N|\la|(|\la|^{1/2}+|\xi'|)^{-|\al'|}e^{-b(|\la|^{1/2}+|\xi'|)x_N} \notag \\ 
&\leq c |\la|^{1/2}(|\la|^{1/2}+|\xi'|)^{-|\al'|}e^{-(b/2)(|\la|^{1/2}+|\xi'|)x_N}. \notag 
\end{align}
On the other hand, using the identity:
\begin{equation*}
e^{ix'\cdot\xi'} 
= \sum_{|\al'|=l}\left(\frac{-ix'}{|x'|^2}\right)^{\al'}\pa_{\xi'}^{\al'}e^{ix'\cdot\xi'} \quad (l\in\BN_0),
\end{equation*}
we have, by integration by parts,
\begin{align*}
&\left(\la\frac{d}{d\la}\right)^n k_{1\la}(x) 
= \int_{\BR^{N-1}}e^{ix'\cdot\xi'}\left(\la\frac{d}{d\la}\right)^n\left( m_1(\xi',\la)\la\CM_0(x_N)\right) \intd \xi' \\
&=\left(\frac{1}{2\pi}\right)^{N-1}\sum_{|\al'|=N}\left(\frac{ix'}{|x'|^2}\right)^{\al'} \\
&\quad \cdot\int_{\BR^{N-1}}e^{ix'\cdot\xi'} \pa_{\xi'}^{\al'}\left\{\left(\la\frac{d}{d\la}\right)^n\left(m_1(\xi',\la)\la\CM_0(x_N)\right)\right\} \intd \xi'.
\end{align*}
This identity, together with \eqref{160830_1}, furnishes that
\begin{equation}\label{160904_10}
|\la^n(d/d\la)^n k_{1\la}(x)| \leq C|x'|^{-N}|\la|^{1/2}\int_{\BR^{N-1}}(|\la|^{1/2}+|\xi'|)^{-N}\intd\xi' \leq C|x'|^{-N}
\end{equation}
for a positive constant $C=C_{N,\de,\si_2,\mu_*,\ka_*,\nu_*}$.
In addition, by \eqref{160830_1} with $|\al'|=0$,
\begin{align*}
&|\la^n(d/d\la)^n k_{1\la}(x)|\leq C|\la|^{1/2}e^{-(b/2)|\la|^{1/2}x_N}\int_{\BR^{N-1}}e^{-(b/2)|\xi'|x_N}\intd\xi' \\
&\leq C (x_N)^{-1}\int_{\BR^{N-1}}e^{-(b/2)|\xi'|x_N}\intd\xi' \leq C (x_N)^{-N},
\end{align*}
which, combined with \eqref{160904_10}, implies that $|\la^n(d/d\la)^n k_{1\la}(x)| \leq C|x|^{-N}$.

Next we estimate $k_{2\la}(x)$.
Since $|\pa_{\xi'}^{\al'}|\xi'|^2|\leq c|\xi'|^{2-|\al'|}$ by \eqref{170418_7},
we have, by \eqref{170429_5} with $n=0$, $m_2(\xi',\la)\in\BBM_{\si_2,\de}^1(0)$, and Leibniz's rule,
\begin{align}\label{170430_5}
&\big|\pa_{\xi'}^{\al'}\big(m_2(\xi',\la)|\xi'|^2\CM_0(x_N)\big)\big| 
\leq c\sum_{\beta'+\ga'+\de'=\al'}\big|\pa_{\xi'}^{\beta'}|\xi'|^2\big|
\big|\pa_{\xi'}^{\ga'}m_2(\xi',\la)\big|\big|\pa_{\xi'}^{\de'}\CM_0(x_N)\big|  \\
&\leq cx_N\sum_{\beta'+\ga'+\de'=\al'}
|\xi'|^{2-|\beta'|}(|\la|^{1/2}+|\xi'|)^{-|\ga'|}(|\la|^{1/2}+|\xi'|)^{-|\de'|}e^{-b(|\la|^{1/2}+|\xi'|)x_N} \notag \\
&\leq c x_N|\xi'|^{2-|\al'|}e^{-b(|\la|^{1/2}+|\xi'|)x_N}
\leq c|\xi'|^{1-|\al'|}e^{-(b/2)(|\la|^{1/2}+|\xi'|)x_N},
\notag
\end{align}
which, combined with \cite[Theorem 2.3]{SS01}, furnishes that
$|k_{2\la}(x)| \leq C|x'|^{-N}$
for some positive constant $C=C_{N,\de,\si_2,\mu_*,\ka_*,\nu_*}$.
On the other hand, by \eqref{170430_5} with $|\al'|=0$,
\begin{equation*}
|k_{2\la}(x)|\leq C\int_{\BR^{N-1}}|\xi'|e^{-(b/2)|\xi'|x_N}\intd\xi'
\leq C(x_N)^{-N}.
\end{equation*}
Thus $|k_{2\la}(x)|\leq C|x|^{-N}$, and similarly $|\la (d/d\la)k_{2\la}(x)|\leq C|x|^{-N}$.

Summing up the above calculations, 
we have obtained
$|\la^n(d/d\la)^n k_{l\la}(x)| \leq C|x|^{-N}$ for $l=1,2$.
Thus, following the proof of Shibata and Shimizu \cite[Lemma 5.4]{SS12},
we can prove by \eqref{170324_1} and \cite[Poposition 3.3]{DHP03} that $\{\la^n(d/d\la)^nM_{0,l}(\la)\mid\la\in\Si_{\si_2,\de}\}$
are $\CR$-bounded on $\CL(L_q(\BR_+^N))$
and their $\CR$-bounds do not exceed a positive constant $C_{N,q,\de,\si_2,\mu_*,\ka_*,\nu_*}$.
This completes the case $M_0(\la)$.

{\bf Case $M_j(\la)$ $(j=1,2)$.}
By \eqref{170427_9}, we have for $j=1,2$
\begin{equation}\label{170429_11}
\CM_j(x_N)=\Fr_j(\xi',\la)
\left(\frac{e^{-t_j x_N}-e^{-\om_\la x_N}}{t_j-\om_\la}\right),
\quad \Fr_j(\xi',\la)= \frac{(s_j-\mu_*^{-1})(t_2+t_1)}{(s_2-s_1)(t_j+\om_\la)}.
\end{equation}
Since $\CM_j(x_N)=-\Fr_j(\xi',\la) x_N\int_0^1 e^{-\{\te t_j+(1-\te)\om_\la\}x_N}\intd\te$
and $\Fr_j(\xi',\la)\in\BBM_{\si_2,0}^1(0)$ by Corollary \ref{coro:symbols},
we can prove the required properties in the same manner as the case $M_0(\la)$
by virtue of \eqref{170429_2}.
This completes the proof of Lemma \ref{lemm:R-bound1}.
\end{proof}

\begin{coro}\label{coro:density}
Let $q\in(1,\infty)$, $\de>0$, and $\si\in(\si_*^w,\pi/2)$. 
\begin{enumerate}[$(1)$]
\item\label{coro:density1}
Let $j=1,2$. For 
$l^j(\xi',\la)\in\BBM_{\si,0}^1(j-3)$ and
$m^j(\xi',\la)\in\BBM_{\si,0}^1(j-2)$,
we define operators $L^j(\la)$, $M_0^j(\la)$ by
\begin{align*}
[L^j(\la)f](x)
&=\CF_{\xi'}^{-1}\left[l^j(\xi',\la)e^{-t_1 x_N}\wh f(\xi',0)\right](x'), \\
[M_0^j(\la)f](x)
&=\CF_{\xi'}^{-1}\left[m^j(\xi',\la)\CM_0(x_N)\wh f(\xi',0)\right](x'),
\end{align*}
for $\la\in\Si_{\si,\de}$ and $f\in W_q^j(\BR_+^N)$.
Then there exist operators $\wt L^j(\la)$, $\wt M_0^j(\la)$, with
\begin{equation*}
\wt L^j(\la), \wt M_0^j(\la)\in
\Hol(\Si_{\si,\de},\CL(L_q(\BR_+)^{N_j},W_q^3(\BR_+^N))),
\end{equation*}
such that for any $g\in W_q^1(\BR_+^N)$ and for any $h\in W_q^2(\BR_+^N)$
\begin{alignat*}{2}
L^1(\la)g&=\wt L^1(\la)(\nabla g,\la^{1/2}g),  \quad
&L^2(\la)h &=\wt L^2(\la)(\nabla^2 h,\la^{1/2}\nabla h,\la h), \\
M_0^1(\la)g &=\wt M_0^1(\la)(\nabla g,\la^{1/2}g), \quad
&M_0^2(\la)h & =\wt M_0^2(\la)(\nabla^2 h,\la^{1/2}\nabla h,\la h),
\end{alignat*}
and also for $n=0,1$
\begin{align*}
\CR_{\CL(L_q(\BR_+^N)^{N_j},\FA_q(\BR_+^N))}
\big(\big\{\la^n(d/d\la)^n\big(\CS_\la\wt L^j(\la)\big) \mid \la\in\Si_{\si,\de}\big\}\big)
&\leq C, \\
\CR_{\CL(L_q(\BR_+^N)^{N_j},\FA_q(\BR_+^N))}
\big(\big\{\la^n(d/d\la)^n\big(\CS_\la\wt M_0^j(\la)\big) \mid \la\in\Si_{\si,\de}\big\}\big)
&\leq C,
\end{align*} 
with a positive constant $C=C_{N,q,\de,\si,\mu_*,\nu_*,\ka_*}$.
\item\label{coro:density2}
Let $j,k=1,2$.
For 
$m^j(\xi',\la)\in\BBM_{\si,0}^1(j-1)$,
we define operators $M_k^j(\la)$ by
\begin{equation*}
[M_k^j(\la)f](x)
=\CF_{\xi'}^{-1}\left[m^j(\xi',\la)\CM_k(x_N)\wh f(\xi',0)\right](x'),
\end{equation*}
for $\la\in\Si_{\si,0}$ and $f\in W_q^j(\BR_+^N)$.
Then there exist operators $\wt M_k^j(\la)$, with
\begin{equation*}
\wt M_k^j(\la)\in\Hol(\Si_{\si,0},\CL(L_q(\BR_+^N)^{N_j},W_q^2(\BR_+^N))),
\end{equation*}
such that for any $g\in W_q^1(\BR)$ and for any $h\in W_q^2(\BR_+^N)$
\begin{equation*}
M_k^1(\la)g =\wt M_k^1(\la)(\nabla g,\la^{1/2}g), \quad
M_k^2(\la)h =\wt M_k^2(\la)(\nabla^2 h,\la^{1/2}\nabla h,\la h),
\end{equation*}
and also for $n=0,1$
\begin{equation*}
\CR_{\CL(L_q(\BR_+^N)^{N_j},\FB_q(\BR_+^N))}
\big(\big\{\la^n(d/d\la)^n\big(\CT_\la\wt M_k^j(\la)\big) \mid \la\in\Si_{\si,0}\big\}\big)
\leq C,
\end{equation*} 
with some positive constant $C=C_{N,q,\si,\mu_*,\nu_*,\ka_*}$.
\end{enumerate}
\end{coro}

\begin{proof}
By \eqref{170427_9} and $\Fr_j(\xi',\la)$ defined in \eqref{170429_11}, we have for $j=1,2$
\begin{align}
\pa_N \CM_0(x_N)
&=-t_2\CM_0(x_N)-e^{-t_1 x_N}, \label{M:deriv_1} \\
\pa_N \CM_j(x_N)
&=-t_j\CM_j(x_N)-\Fr_j(\xi',\la)e^{-\om_\la x_N}. \label{M:deriv_2}
\end{align}

\eqref{coro:density1}.
We only prove the case $M_0^1(\la)$. 
By \eqref{170427_9}, \eqref{vole}, and \eqref{M:deriv_1}, 
\begin{align*}
M_0^1(\la)g
&=\int_0^\infty\CF_{\xi'}^{-1}
\left[\frac{\la^{1/2}t_2m^1(\xi',\la)}{\mu_*\om_\la^2}\CM_0(x_N+y_N)
\wh{\la^{1/2}g}(\xi',y_N)\right](x')\intd y_N \\
&-\sum_{k=1}^{N-1}\int_0^\infty
\CF_{\xi'}^{-1}\left[\frac{i\xi_k t_2 m^1(\xi',\la)}{\om_\la^2}\CM_0(x_N+y_N)\wh{\pa_k g}(\xi',y_N)\right](x')\intd y_N \\
&+\int_0^\infty \CF_{\xi'}^{-1}
\left[\frac{\la^{1/2}m^1(\xi',\la)}{\mu_*\om_\la^2}e^{-t_1(x_N+y_N)}\wh{\la^{1/2}g}(\xi',y_N)\right](x')\intd y_N \\
&-\sum_{k=1}^{N-1}\int_0^\infty
\CF_{\xi'}^{-1}
\left[\frac{i\xi_k m^1(\xi',\la)}{\om_\la^2}e^{-t_1(x_N+y_N)}\wh{\pa_ k g}(\xi',y_N)\right](x')\intd y_N \\
&-\int_0^\infty\CF_{\xi'}^{-1}\left[m^1(\xi',\la)\CM_0(x_N+y_N)\wh{\pa_N g}(\xi',y_N)\right](x')\intd y_N \\
&=:\sum_{l=1}^5\wt M_{0,l}^1(\la)(\nabla g,\la^{1/2}g).
\end{align*}

First we estimate $\wt M_{0,1}^1(\la)$, $\wt M_{0,2}^1(\la)$, and $\wt M_{0,5}^1(\la)$.
Here we consider $\wt M_{0,5}^1(\la)$ only, because the other terms can be treated similarly.
Let $j,k,l=1,\dots,N-1$. By \eqref{M:deriv_1},  $\wt M_{0,5}^1(\la)$ can be written as
\begin{align*}
&\la\wt M_{0,5}^1(\la)(\nabla g,\la^{1/2}g) \\
&\quad = 
-\int_0^\infty\CF_{\xi'}^{-1}\Big[\la m^1(\xi',\la) \CM_0(x_N+y_N)\wh{\pa_N g}(\xi',y_N)\Big](x')\intd y_N,
\\
&\big(\la\pa_j,\la^{1/2}\pa_j\pa_k,\pa_j\pa_k\pa_l\big) 
\wt M_{0,5}^1(\la)(\nabla g,\la^{1/2}g) \\
&\quad = 
-\int_0^\infty\CF_{\xi'}^{-1}\Big[\big(\la i\xi_j,-\la^{1/2}\xi_j\xi_k,-i\xi_j\xi_k\xi_l\big) m^1(\xi',\la) \\
&\quad \cdot\CM_0(x_N+y_N)\wh{\pa_N g}(\xi',y_N)\Big](x')\intd y_N,
\\
&\big(\la,\la^{1/2}\pa_j,\pa_j\pa_k\big)\pa_N\wt M_{0,5}^1(\la)(\nabla g,\la^{1/2}g) \\
&\quad =
\int_0^\infty\CF_{\xi'}^{-1}\Big[\big(\la,\la^{1/2}i\xi_j,-\xi_j\xi_k\big) t_2 m^1(\xi',\la)
\CM_0(x_N+y_N)\wh{\pa_N g}(\xi',y_N)\Big](x')\intd y_N
\\
&\quad+\int_0^\infty\CF_{\xi'}^{-1}\Big[\big(\la,\la^{1/2}i\xi_j,-\xi_j\xi_k\big)m^1(\xi',\la) 
e^{-t_1(x_N+y_N)}\wh{\pa_N g}(\xi',y_N)\Big](x')\intd y_N,
\\
&\big(\la^{1/2},\pa_j\big)\pa_N^2\wt M_{0,5}^1(\la)(\nabla g,\la^{1/2}g) \\
&\quad =
-\int_0^\infty\CF_{\xi'}^{-1}\Big[\big(\la^{1/2},i\xi_j\big)t_2^2 m^1(\xi',\la)
\CM_0(x_N+y_N)\wh{\pa_N g}(\xi',y_N)\Big](x')\intd y_N \\
&\quad -\int_0^\infty\CF_{\xi'}^{-1}
\Big[\big(\la^{1/2},i\xi_j\big) (t_2+t_1) m^1(\xi',\la)
e^{-t_1(x_N+y_N)}\wh{\pa_N g}(\xi',y_N)\Big](x')\intd y_N,
\\
&\pa_N^3\wt M_{0,5}^1(\la)(\nabla g,\la^{1/2}g) \\
&\quad =
\int_0^\infty\CF_{\xi'}^{-1}\Big[t_2^3m^1(\xi',\la)\CM_0(x_N+y_N)\wh{\pa_N g}(\xi',y_N)\Big](x')\intd y_N \\
&\quad+\int_0^\infty\CF_{\xi'}^{-1}\Big[\{t_2^2+t_1(t_2+t_1)\}m^1(\xi',\la)e^{-t_1(x_N+y_N)}\wh{\pa_N g}(\xi',y_N)\Big](x')\intd y_N.
\end{align*}
Then, by Lemma \ref{lemm:algebra}, Corollary \ref{coro:symbols}, and the assumption for $m^1(\xi',\la)$,
\begin{align*}
&\la m^1(\xi',\la)\in \BBM_{\si,0}^1(1)\subset \BBM_{\si,\de}^1(2), 
\quad\big(\la i\xi_j,-\la^{1/2}\xi_j\xi_k,-i\xi_j\xi_k\xi_l\big) m^1(\xi',\la)\in \BBM_{\si,0}^1(2), \\
&\big(\la,\la^{1/2}i\xi_j,-\xi_j\xi_k\big) t_2 m^1(\xi',\la)\in \BBM_{\si,0}^1(2), 
\quad \big(\la,\la^{1/2}i\xi_j-\xi_j\xi_k\big)m^1(\xi',\la)\in \BBM_{\si,0}^1(1), \\
&\big(\la^{1/2},i\xi_j\big)t_2^2 m^1(\xi',\la)\in \BBM_{\si,0}^1(2), \quad
\big(\la^{1/2},i\xi_j\big) (t_2+t_1) m^1(\xi',\la)\in \BBM_{\si,0}^1(1), \\
&t_2^3m^1(\xi',\la)\in \BBM_{\si,0}^1(2), \quad
\{t_2^2+t_1(t_2+t_1)\}m^1(\xi',\la)\in \BBM_{\si,0}^1(1).
\end{align*}
Combining these properties with Lemma \ref{lemm:R-bound1} furnishes  for $n=0,1$ 
\begin{equation}\label{170501_1}
\CR_{\CL(L_q(\BR_+^N)^{N_1},\FA_q(\BR_+^N))}
\big(\big\{\la^n(d/d\la)^n\big(\CS_\la\wt M_{0,5}^1(\la)\big) \mid \la\in\Si_{\si,\de}\big\}\big) \leq C,
\end{equation}
with some positive constant $C=C_{N,q,\de,\si,\mu_*,\nu_*,\ka_*}$. Analogously, for $l=1,2$,
\begin{equation}\label{170501_2}
\CR_{\CL(L_q(\BR_+^N)^{N_1},\FA_q(\BR_+^N))}
\big(\big\{\la^n(d/d\la)^n\big(\CS_\la\wt M_{0,l}^1(\la)\big) \mid \la\in\Si_{\si,\de}\big\}\big) \leq C.
\end{equation}

Next we estimate $\wt M_{0,3}^1(\la)$, $\wt M_{0,4}^1(\la)$.
By Lemma \ref{lemm:algebra}, Corollary \ref{coro:symbols}, and the assumption for $m^1(\xi',\la)$, we have for $k=1,\dots,N-1$
\begin{equation*}
\frac{\la^{1/2}m^1(\xi',\la)}{\mu_*\om_\la^2}, \frac{i\xi_k m^1(\xi',\la)}{\om_\la^2}
\in\BBM_{\si,0}^1(-2).
\end{equation*}
Thus, similarly to the case $\wt M_{0,5}^1(\la)$,
we obtain by Lemma \ref{lemm:R-bound1}
\begin{equation*}
\CR_{\CL(L_q(\BR_+^N)^{N_1},\FA_q(\BR_+^N))}
\big(\big\{\la^n(d/d\la)^n(\CS_\la\wt M_{0,l}^1(\la)) \mid \la\in\Si_{\si,\de}\big\}\big)
\leq C
\end{equation*}
for $n=0,1$ and $l=3,4$ with $C=C_{N,q,\de,\si,\mu_*,\nu_*,\ka_*}$,
which, combined with \eqref{170501_1} and \eqref{170501_2}, completes the proof of the case $M_0^1(\la)$.
This completes the proof of Corollary \ref{coro:density} \eqref{coro:density1}.

\eqref{coro:density2}.
The proof is similar to \eqref{coro:density1} by virtue of \eqref{M:deriv_2} and $\Fr_j(\xi',\la)\in\BBM_{\si,0}^1(0)$,
so that we may omit the detailed proof.
This completes the proof of the corollary.
\end{proof}

\section{Proof of Theorem \ref{theo1}}\label{sec:proof1}
As was seen in Section \ref{sec:half},
it suffices to show Theorem \ref{theo:half} in order to
complete the proof of Theorem \ref{theo1}.
From this viewpoint, we prove Theorem \ref{theo:half} in this section.
Let $(\rho,\Bu)$ be the solution of \eqref{eq:half1}
obtained in Section \ref{sec:half},
and let $u_J$ be the $J$th component of $\Bu$ for $J=1,\dots,N$.

\subsection{Construction of solution operator families}
This subsection constructs solution operator families associated with $(\rho,\Bu)$.

First, we consider the formula of $\rho$.
It can be written as 
\begin{align*}
\rho 
&= 
\CF_{\xi'}^{-1}
\left[\left\{\left(\frac{t_1^2-|\xi'|^2}{\la t_1}\right)\beta_N 
+\left(\frac{t_2^2-|\xi'|^2}{\la t_2}\right)\ga_N\right\}e^{-t_1 x_N}\right](x') \\
&+ \CF_{\xi'}^{-1}\left[\left(\frac{t_2^2-|\xi'|^2}{\la t_2}\right)\ga_N\left(e^{-t_2 x_N}-e^{-t_1 x_N}\right)\right](x').
\end{align*}
Inserting \eqref{bega} into the above formula 
yields that $\rho = \SSS^1(\la)\Bg + \SSS^2(\la)h$ with
\begin{align*}
&\SSS^1(\la)\Bg := 
-\sum_{l=1}^2\sum_{k=1}^{N-1}
\CF_{\xi'}^{-1}\left[\frac{2i\xi_k \om_\la t_1 t_2   L_{l 1} (t_{l}^2-|\xi'|^2)}{\mu_*\la t_l \det\BL}e^{-t_1 x_N}\wh g_k(0)\right](x') \\
&\quad+\sum_{l=1}^2
\CF_{\xi'}^{-1}\left[\frac{t_1t_2 L_{l 1}(t_l^2-|\xi'|^2)(\om_\la^2+|\xi'|^2)}{\mu_* \la t_l \det\BL}e^{-t_1 x_N}\wh g_N(0)\right](x') \\
&\quad -\sum_{k=1}^{N-1}
\CF_{\xi'}^{-1}\left[\frac{2i\xi_k \om_\la t_1 L_{21} (t_2^2-|\xi'|^2)}{\mu_*\la \det\BL}\left(e^{-t_2 x_N}-e^{-t_1 x_N}\right)\wh g_k(0)\right](x') \\
&\quad+ 
\CF_{\xi'}^{-1}\left[\frac{t_1 L_{21}(t_2^2-|\xi'|^2)(\om_\la^2+|\xi'|^2)}{\mu_*\la\det\BL}\left(e^{-t_2 x_N}-e^{-t_1 x_N}\right)\wh g_N(0)\right](x'), \\
&\SSS^2(\la)h :=\sum_{l=1}^2\CF_{\xi'}^{-1}
\left[\frac{L_{l2}(t_l^2-|\xi'|^2)}{t_l \det\BL} e^{-t_1 x_N}\wh h(0)\right](x') \\
&\quad + \CF_{\xi'}^{-1}\left[\frac{L_{22}(t_2^2-|\xi'|^2)}{t_2\det\BL}\left(e^{-t_2 x_N}-e^{-t_1 x_N}\right)\wh h(0)\right](x').
\end{align*}

Secondly, we consider the formula of $u_J$.
Let $j=1,\dots,N-1$, and then inserting \eqref{bega} into $u_J$ yields that 
$u_J=\SST_J^1(\la)\Bg+\SST_J^2(\la)h$ with
\begin{align*}
& \SST_j^1(\la)\Bg :=
\CF_{\xi'}^{-1}\left[\frac{1}{\mu_*\om_\la} e^{-\om_\la x_N}\wh g_j(0)\right](x') 
+\CF_{\xi'}^{-1}\left[\frac{i\xi_j}{2\mu_*\om_\la^2}e^{-\om_\la x_N} \wh g_N(0)\right](x')\\
&\quad +\sum_{l=1}^2\sum_{k=1}^{N-1}
\CF_{\xi'}^{-1}\left[\frac{\xi_j\xi_k t_1 t_2 L_{l 1}(4 t_l \om_\la-3\om_\la^2-|\xi'|^2)}{\mu_* t_l\om_\la \det\BL}e^{-\om_\la x_N}\wh{g}_k(0)\right](x') \\
&\quad +\sum_{l=1}^2\CF_{\xi'}^{-1}
\left[\frac{i\xi_j t_1t_2 L_{l 1}(\om_\la^2+|\xi'|^2)(4 t_l \om_\la-3\om_\la^2-|\xi'|^2)}{2\mu_*t_l \om_\la^2\det\BL}e^{-\om_\la x_N}\wh g_N(0)\right](x') \\
&\quad -\sum_{l=1}^2\sum_{k=1}^{N-1}\CF_{\xi'}^{-1}
\left[\frac{2\xi_j\xi_k \om_\la  t_1t_2 L_{l 1}}{\mu_*t_l \det\BL}\left(e^{-t_l x_N}-e^{-\om_\la x_N}\right)\wh g_k(0)\right](x') \\
&\quad -\sum_{l=1}^2\CF_{\xi'}^{-1}
\left[\frac{i\xi_j t_1t_2L_{l 1} (\om_\la^2+|\xi'|^2)}{\mu_* t_l\det\BL}\left(e^{-t_l x_N}-e^{-\om_\la x_N}\right)\wh g_N(0)\right](x'), \\
&\SST_j^2(\la)h :=
\sum_{l=1}^2\CF_{\xi'}^{-1}
\left[\frac{\la i\xi_j L_{l 2}(4 t_l\om_\la-3\om_\la^2-|\xi'|^2)}{2 t_l \om_\la^2\det\BL}e^{-\om_\la x_N}\wh h(0)\right](x') \\
&\quad -\sum_{l=1}^2 \CF_{\xi'}^{-1}
\left[\frac{\la i\xi_j L_{l2}}{t_l\det\BL}\left(e^{-t_l x_N}-e^{-\om_\la x_N}\right)\wh h(0)\right](x'), \\
&\SST_N^1(\la)\Bg
:= \CF_{\xi'}^{-1}\left[\frac{1}{2\mu_*\om_\la}e^{-\om_\la x_N}\wh g_N(0)\right](x') \\
&\quad -\sum_{l=1}^2\sum_{k=1}^{N-1}
\CF_{\xi'}^{-1}\left[\frac{i\xi_k t_1t_2 L_{l1}(2t_l\om_\la-\om_\la^2-|\xi'|^2)}{\mu_* t_l  \det\BL}
e^{-\om_\la x_N}\wh g_k(0)\right](x') \\
&\quad +\sum_{l=1}^2 \CF_{\xi'}^{-1}
\left[\frac{t_1t_2 L_{l1}(\om_\la^2+|\xi'|^2)(2t_l\om_\la-\om_\la^2-|\xi'|^2)}{2\mu_* t_l \om_\la \det\BL}e^{-\om_\la x_N}\wh g_N(0)\right](x') \\
&\quad -\sum_{l=1}^2\sum_{k=1}^{N-1}
\CF_{\xi'}^{-1}\left[\frac{2 i\xi_k t_1t_2\om_\la L_{l1}}{\mu_* \det\BL}\left(e^{-t_l x_N}-e^{-\om_\la x_N}\right)\wh g_k(0)\right](x') \\
&\quad + \sum_{l=1}^2\CF_{\xi'}^{-1}
\left[\frac{t_1t_2 L_{l1}(\om_\la^2+|\xi'|^2)}{\mu_* \det\BL}\left(e^{-t_l x_N}-e^{-\om_\la x_N}\right) \wh g_N(0)\right](x'), \\
&\SST_N^2(\la) h
:= 
\sum_{l=1}^2 \CF_{\xi'}^{-1}\left[\frac{\la L_{l2}(2t_l\om_\la-\om_\la^2-|\xi'|^2)}{ 2t_l\om_\la\det\BL}e^{-\om_\la x_N}\wh h(0)\right](x') \\
&\quad +\sum_{l=1}^2\CF_{\xi'}^{-1}\left[\frac{\la L_{l2}}{\det\BL}\left(e^{-t_l x_N}-e^{-\om_\la x_N}\right)\wh h(0)\right](x').
\end{align*}

Thirdly, we rewrite $\SSS^1(\la)\Bg$, $\SSS^2(\la)$, $\SST_J^1(\la)\Bg$, and $\SST_J^2(\la)h$
by means of \eqref{170427_9}, \eqref{defi:M}, and 
symbols introduced in Subsection \ref{subsec:5_2}
in order to eliminate $\la$ and $t_2-t_1$ contained in $\det\BL$. 
For simplicity, we set
\begin{equation*}
\Fa(\xi',\la) =
\frac{s_1s_2(t_2+t_1)}{s_2-s_1}, \quad
\Fb(\xi',\la) =
\frac{t_2+t_1}{s_2-s_1}.
\end{equation*}

It now holds that
\begin{align*}
\sum_{l=1}^2\frac{L_{l1}(t_l^2-|\xi'|^2)}{\la t_l \det\BL}
&= \frac{s_1s_2(t_1+\om_\la)}{t_2 \Fl_1(\xi',\la)}, \quad
\frac{L_{21}(t_2^2-|\xi'|^2)}{\la \det\BL}
= -\frac{s_1 s_2 t_1(t_1+\om_\la)}{(t_2-t_1)\Fl_1(\xi',\la)}, \\
\sum_{l=1}^2 \frac{L_{l2}(t_l^2-|\xi'|^2)}{t_l\det\BL} 
&=\Fa(\xi',\la)\sum_{l=1}^2\frac{(-1)^{l+1} t_l\Fm_l(\xi',\la)}{s_l\Fl_l(\xi',\la)}, \\
\frac{L_{22}(t_2^2-|\xi'|^2)}{t_2 \det\BL}
&=  \frac{\la s_2 t_1 \Fm_1(\xi',\la)}{(t_2-t_1)\Fl_1(\xi',\la)},
\end{align*}
%
%
which, inserted into $\SSS^1(\la)\Bg$ and $\SSS^2(\la)h$, furnishes that
\begin{align*}
&\SSS^1(\la)\Bg = 
 -\sum_{k=1}^{N-1}
\CF_{\xi'}^{-1}\left[\frac{2i\xi_k \om_\la t_1 s_1 s_2(t_1+\om_\la)}
{\mu_*\Fl_1(\xi',\la)}e^{-t_1 x_N}\wh g_k(0)\right](x')  \\
&\quad+
\CF_{\xi'}^{-1}\left[\frac{t_1s_1s_2(t_1+\om_\la)(\om_\la^2+|\xi'|^2)}
{\mu_* \Fl_1(\xi',\la)}e^{-t_1 x_N}\wh g_N(0)\right](x')  \\
&\quad +\sum_{k=1}^{N-1}
\CF_{\xi'}^{-1}\left[\frac{2i\xi_k \om_\la s_1 s_2 t_1^2(t_1+\om_\la)}
{\mu_*\Fl_1(\xi',\la)}\CM_0(x_N)\wh g_k(0)\right](x') \\
&\quad-
\CF_{\xi'}^{-1}\left[\frac{s_1 s_2 t_1^2(t_1+\om_\la)(\om_\la^2+|\xi'|^2)}
{\mu_*\Fl_1(\xi',\la)}\CM_0(x_N)\wh g_N(0)\right](x'),  \\
&\SSS^2(\la)h =\sum_{l=1}^2\CF_{\xi'}^{-1}
\left[\frac{(-1)^{l+1} \Fa(\xi',\la) t_l\Fm_l(\xi',\la)}{s_l \Fl_l(\xi',\la)} e^{-t_1 x_N}\wh h(0)\right](x') 
 \\
&\quad + \CF_{\xi'}^{-1}\left[\frac{\la s_2 t_1\Fm_1(\xi',\la)}{\Fl_1(\xi',\la)}\CM_0(x_N)\wh h(0)\right](x').
\end{align*}
%
%
%
%
In addition, we observe that for $A=\{(1,2),(2,1)\}$
\begin{align*}
\sum_{l=1}^2\frac{L_{l1}(4t_l\om_\la-3\om_\la^2-|\xi'|^2)}{t_l\det\BL}
&=\Fa(\xi',\la)\sum_{l=1}^2
\frac{(-1)^{l+1}\Fp_l(\xi',\la)}{s_l\Fl_l(\xi',\la)}, \\
\sum_{l=1}^2 \frac{L_{l1}(e^{-t_l x_N}-e^{-\om_\la x_N})}{t_l \det\BL}
&=s_1 s_2\sum_{l=1}^2 \frac{(-1)^{l+1}(t_l+\om_\la)}{s_l\Fl_l(\xi',\la)}\CM_l(x_N), \\
\sum_{l=1}^2 \frac{L_{l2}(4 t_l\om_\la-3\om_\la^2-|\xi'|^2)}{t_l\det\BL}
&=\Fb(\xi',\la)\sum_{(l,m)\in A} \frac{(-1)^{l+1}t_l\Fm_l(\xi',\la)\Fp_m(\xi',\la)}{\Fl_l(\xi',\la)(t_m+\om_\la)}, \\
\sum_{l=1}^2\frac{L_{l2}(e^{-t_l x_N}-e^{-\om_\la x_N})}{t_l\det\BL}
&=\sum_{(l,m)\in A}
\frac{(-1)^{l+1}t_l\Fm_l(\xi',\la)}{\Fl_l(\xi',\la)}\CM_m(x_N), \\
\sum_{l=1}^2\frac{L_{l1}(2t_l\om_\la-\om_\la^2-|\xi'|^2)}{t_l\det\BL}
&=\Fa(\xi',\la)\sum_{l=1}^2
\frac{(-1)^{l+1}\Fq_l(\xi',\la)}{s_l\Fl_l(\xi',\la)}, \\
\sum_{l=1}^2 \frac{L_{l1}(e^{-t_l x_N}-e^{-\om_\la x_N})}{\det\BL}
&= s_1 s_2\sum_{l=1}^2
\frac{(-1)^{l+1} t_l(t_l+\om_\la)}{s_l\Fl_l(\xi',\la)}\CM_l(x_N), \\
\sum_{l=1}^2 \frac{L_{l2}(2 t_l\om_\la-\om_\la^2-|\xi'|^2)}{t_l\det\BL} 
&=\Fb(\xi',\la)\sum_{(l,m)\in A} \frac{(-1)^{l+1}t_l\Fm_l(\xi',\la)\Fq_m(\xi',\la)}{\Fl_l(\xi',\la)(t_m+\om_\la)}, \\
\sum_{l=1}^2\frac{L_{l2}(e^{-t_l x_N}-e^{-\om_\la x_N})}{\det\BL}
&=t_1t_2\sum_{(l,m)\in A}
\frac{(-1)^{l+1}\Fm_l(\xi',\la)}{\Fl_l(\xi',\la)}\CM_m(x_N).
\end{align*}
Inserting these relations into $\SST_J^1(\la)\Bg$ and $\SST_J^2(\la)h$
furnishes that
\begin{align*}
&\SST_j^1(\la)\Bg =
\CF_{\xi'}^{-1}\bigg[
\frac{1}{\mu_*\om_\la} e^{-\om_\la x_N}\wh g_j(0)
\bigg](x') 
+\CF_{\xi'}^{-1}\bigg[
\frac{i\xi_j}{2\mu_*\om_\la^2}e^{-\om_\la x_N} \wh g_N(0)
\bigg](x') \\
&\quad+\sum_{l=1}^2\sum_{k=1}^{N-1}
\CF_{\xi'}^{-1}\bigg[
\frac{(-1)^{l+1}\xi_j\xi_kt_1t_2\Fa(\xi',\la)\Fp_l(\xi',\la)}
{\mu_*\om_\la s_l\Fl_l(\xi',\la)}e^{-\om_\la x_N}\wh g_k(0)
\bigg](x') \notag \\
&\quad +\sum_{l=1}^2 \CF_{\xi'}^{-1}\bigg[
\frac{(-1)^{l+1}i\xi_j t_1t_2(\om_\la^2+|\xi'|^2)\Fa(\xi',\la)\Fp_l(\xi',\la)}
{2\mu_*\om_\la^2 s_l\Fl_l(\xi',\la)} e^{-\om_\la x_N}\wh g_N(0)
\bigg](x')
\notag \\
&\quad-\sum_{l=1}^2\sum_{k=1}^{N-1}
\CF_{\xi'}^{-1}\bigg[
\frac{(-1)^{l+1}2\xi_j\xi_k\om_\la t_1t_2s_1s_2(t_l+\om_\la)}
{\mu_* s_l\Fl_l(\xi',\la)}\CM_l(x_N)\wh g_k(0)
\bigg](x') \notag \\
&\quad-\sum_{l=1}^2\CF_{\xi'}^{-1}\bigg[
\frac{(-1)^{l+1}i\xi_j t_1t_2(\om_\la^2+|\xi'|^2)s_1s_2(t_l+\om_\la) }{\mu_* s_l\Fl_l(\xi',\la)}
\CM_l(x_N)\wh g_N(0)\bigg](x'), \notag \\
&\SST_j^2(\la) h=
\sum_{(l,m)\in A}
\CF_{\xi'}^{-1}\bigg[
\frac{(-1)^{l+1}\la i\xi_j\Fb(\xi',\la)t_l\Fm_l(\xi',\la)\Fp_m(\xi',\la)}
{2\om_\la^2\Fl_l(\xi',\la)(t_m+\om_\la)}
e^{-\om_\la x_N}\wh h(0)\bigg](x') \\
&\quad -\sum_{(l,m)\in A}
\CF_{\xi'}^{-1}
\bigg[\frac{(-1)^{l+1}\la i\xi_j t_l \Fm_l(\xi',\la)}{\Fl_l(\xi',\la)}
\CM_m(x_N)\wh h(0)\bigg](x'), 
\notag \\
&\SST_N^1(\la)\Bg =
\CF_{\xi'}^{-1}\bigg[\frac{1}{2\mu_*\om_\la}e^{-\om_\la x_N}\wh g_N(0)\bigg](x') \\
&\quad -
\sum_{l=1}^2\sum_{k=1}^{N-1}
\CF_{\xi'}^{-1}
\bigg[
\frac{(-1)^{l+1} i\xi_k t_1 t_2 \Fa(\xi',\la)\Fq_l(\xi',\la)}{\mu_* s_l\Fl_l(\xi',\la)}
e^{-\om_\la x_N}\wh g_k(0)
\bigg](x') 
\notag \\
&\quad +\sum_{l=1}^2
\CF_{\xi'}^{-1}\bigg[
\frac{(-1)^{l+1} t_1 t_2(\om_\la^2+|\xi'|^2)\Fa(\xi',\la)\Fq_l(\xi',\la)}{2\mu_*\om_\la s_l\Fl_l(\xi',\la)}
e^{-\om_\la x_N}\wh g_N(0)
\bigg](x')
\notag \\
& \quad -\sum_{l=1}^2\sum_{k=1}^{N-1}
\CF_{\xi'}^{-1}\bigg[
\frac{(-1)^{l+1} 2 i\xi_k t_1 t_2\om_\la s_1s_2t_l(t_l+\om_\la)}{\mu_* s_l \Fl_l(\xi',\la)} \CM_l(x_N)\wh g_k(0)
\bigg](x') \notag \\
&\quad +\sum_{l=1}^2
\CF_{\xi'}^{-1}\bigg[
\frac{(-1)^{l+1} t_1t_2(\om_\la^2+|\xi'|^2)s_1 s_2 t_l(t_l+\om_\la)}{\mu_* s_l\Fl_l(\xi',\la)}
\CM_l(x_N)\wh g_N(0)\bigg](x') \notag \\
&\SST_N^2(\la)h =
\sum_{(l,m)\in A}\CF_{\xi'}^{-1}
\left[\frac{(-1)^{l+1}\la \Fb(\xi',\la)t_l\Fm_l(\xi',\la)\Fq_m(\xi',\la)}{2\om_\la \Fl_l(\xi',\la)(t_m+\om_\la)}
e^{-\om_\la x_N}\wh h(0)
\right](x') \notag \\
&\quad +\sum_{(l,m)\in A}
\CF_{\xi'}^{-1}
\bigg[\frac{(-1)^{l+1}\la t_1t_2\Fm_l(\xi',\la)}{\Fl_l(\xi',\la)}
\CM_m(x_N)\wh h(0)\bigg](x').
\notag
\end{align*}

\subsection{$\CR$-bounded solution operator families}
This subsection shows the existence of $\CR$-bounded solution operator families
associated with $(\rho,\Bu)$.
Let $\de>0$ and $\si\in(\si_*,\pi/2)$ for the same constant $\si_*$ as in Lemma \ref{lemm:symbols},
and let $j,k=1,\dots,N-1$ and $l,m=1,2$ in what follows.
In addition, we note that
$\Fa(\xi',\la), \Fb(\xi',\la)\in\BBM_{\si,0}^1(1)$ by Corollary \ref{coro:symbols}. 

First, we consider $\SSS^1(\la)$, $\SSS^2(\la)$.
By Lemma \ref{lemm:algebra}, \eqref{170426_1}, and Corollary \ref{coro:symbols}, 
\begin{align*}
\frac{2i\xi_k\om_\la t_1 s_1 s_2(t_1+\om_\la)}{\mu_*\Fl_1(\xi',\la)},
\,\frac{t_1 s_1 s_2(t_1+\om_\la)(\om_\la^2+|\xi'|^2)}{\mu_*\Fl_1(\xi',\la)}
&\in\BBM_{\si,0}^1(-2); \\
\frac{2i\xi_k\om_\la s_1 s_2 t_1^2(t_1+\om_\la)}{\mu_*\Fl_1(\xi',\la)},
\,\frac{s_1 s_2 t_1^2(t_1+\om_\la)(\om_\la^2+|\xi'|^2)}{\mu_*\Fl_1(\xi',\la)},
\,\frac{\Fa(\xi',\la)t_l\Fm_l(\xi',\la)}{s_l\Fl_l(\xi',\la)}
&\in\BBM_{\si,0}^1(-1); \\
\frac{\la s_2 t_1 \Fm_1(\xi',\la)}{\Fl_1(\xi',\la)}
&\in\BBM_{\si,0}^1(0).
\end{align*}
Let $\CN_1=N^2+N$ and $\CN_2=N^2+N+1$. Then
Corollary \ref{coro:density} \eqref{coro:density1}
furnishes that
there exist operators $\wt \SSS^1(\la)$, $\wt\SSS^2(\la)$, with 
\begin{align*}
\wt\SSS^l(\la)\in\Hol(\Si_{\si,\de},\CL(L_q(\BR_+^N)^{\CN_l},W_q^3(\BR_+^N))),
\end{align*} 
such that for any $\Bg\in W_q^1(\BR_+^N)^N$ and for any $h\in W_q^2(\BR_+^N)$
\begin{align*}
\SSS^1(\la)\Bg=\wt\SSS^1(\la)(\nabla \Bg,\la^{1/2}\Bg),\quad
\SSS^2(\la)h = \wt\SSS^2(\la)(\nabla^2 h,\la^{1/2}\nabla h,\la h),
\end{align*}
and also for $n=0,1$
\begin{align*}
\CR_{\CL(L_q(\BR_+^N)^{\CN_l},\FA_q(\BR_+^N))}
\big(\big\{\la^n(d/d\la)^n\big(\CS_\la\wt\SSS^l(\la)\big) \mid 
\la\in\Si_{\si,\de} \big\} \big) \leq C,
\end{align*}
with some positive constant $C=C_{N,q,\de,\si,\mu_*,\nu_*,\ka_*}$.

Next we consider $\SST_j^1(\la)$, $\SST_j^2(\la)$.
By Lemma \ref{lemm:algebra}, \eqref{170426_1}, and Corollary \ref{coro:symbols}, 
\begin{align*}
\frac{1}{\mu_*\om_\la},\,\frac{i\xi_j}{2\mu_*\om_\la^2},\,
\frac{\xi_j\xi_k t_1 t_2\Fa(\xi',\la)\Fp_l(\xi',\la)}{\mu_*\om_\la s_l\Fl_l(\xi',\la)},&\\
\frac{i\xi_jt_1 t_2 (\om_\la^2+|\xi'|^2)\Fa(\xi',\la)\Fp_l(\xi',\la)}{2\mu_*\om_\la^2s_l\Fl_l(\xi',\la)}
&\in\BBM_{\si,0}^1(-1); \\
\frac{2\xi_j\xi_k\om_\la t_1t_2s_1 s_2(t_l+\om_\la) }{\mu_*s_l\Fl_l(\xi',\la)},
\,\frac{i\xi_j t_1 t_2 (\om_\la^2+|\xi'|^2)s_1 s_2(t_l+\om_\la)}{\mu_* s_l\Fl_l(\xi',\la)}, &\\
\frac{\la i\xi_j\Fb(\xi',\la)t_l\Fm_l(\xi',\la)\Fp_m(\xi',\la)}{2\om_\la^2\Fl_l(\xi',\la)(t_m+\om_\la)} 
\in \BBM_{\si,0}^1(0); \quad
\frac{\la i\xi_j t_l\Fm_l(\xi',\la)}{\Fl_l(\xi',\la)}
&\in\BBM_{\si,0}^1(1).
\end{align*}
Corollary \ref{coro:R-bound0} and Corollary \ref{coro:density} \eqref{coro:density2} thus furnish that
there exist operators $\wt \SST_j^1(\la)$, $\wt\SST_j^2(\la)$, with
$\wt\SST_j^l(\la)\in\Hol(\Si_{\si,0},\CL(L_q(\BR_+^N)^{\CN_l},W_q^2(\BR_+^N)))$,
such that for any $\Bg\in W_q^1(\BR_+^N)^N$ and for any $h\in W_q^2(\BR_+^N)$
\begin{align*}
\SST_j^1(\la)\Bg=\wt\SST_j^1(\la)(\nabla \Bg,\la^{1/2}\Bg),\quad
\SST_j^2(\la)h = \wt\SST_j^2(\la)(\nabla^2 h,\la^{1/2}\nabla h,\la h),
\end{align*}
and also for $n=0,1$
\begin{align*}
\CR_{\CL(L_q(\BR_+^N)^{\CN_l},\FB_q(\BR_+^N))}
\big(\big\{\la^n(d/d\la)^n\big(\CT_\la\wt\SST_j^l(\la)\big) \mid 
\la\in\Si_{\si,0} \big\} \big) \leq C,
\end{align*}
with some positive constant $C=C_{N,q,\si,\mu_*,\nu_*,\ka_*}$

Finally, we consider $\SST_N^1(\la)$, $\SST_N^2(\la)$.
By Lemma \ref{lemm:algebra}, \eqref{170426_1}, and Corollary \ref{coro:symbols},  
\begin{align*}
\frac{1}{2\mu_*\om_\la},
\,\frac{i\xi_k t_1 t_2 \Fa(\xi',\la)\Fq_l(\xi',\la)}{\mu_* s_l\Fl_l(\xi',\la)},
\,\frac{t_1 t_2 (\om_\la^2+|\xi'|^2)\Fa(\xi',\la)\Fq_l(\xi',\la)}{2\mu_*\om_\la s_l\Fl_l(\xi',\la)}
&\in\BBM_{\si,0}^1(-1); \\
\frac{2i\xi_k t_1 t_2 \om_\la s_1 s_2 t_l(t_l+\om_\la)}{\mu_* s_l\Fl_l(\xi',\la)}, 
\,\frac{t_1 t_2 (\om_\la^2+|\xi'|^2)s_1 s_2 t_l(t_l+\om_\la)}{\mu_* s_l\Fl_l(\xi',\la)},& \\
\frac{\la \Fb(\xi',\la)t_l\Fm_l(\xi',\la)\Fq_m(\xi',\la)}{2\om_\la \Fl_l(\xi',\la)(t_m+\om_\la)}
\in\BBM_{\si,0}^1(0); \quad
\frac{\la t_1 t_2 \Fm_l(\xi',\la)}{\Fl_l(\xi',\la)}&\in\BBM_{\si,0}^1(1).
\end{align*}
Corollary \ref{coro:R-bound0} and Corollary \ref{coro:density} \eqref{coro:density2} thus furnish that
there exist operators $\wt \SST_N^1(\la)$, $\wt\SST_N^2(\la)$, with
$\wt\SST_N^l(\la)\in\Hol(\Si_{\si,0},\CL(L_q(\BR_+^N)^{\CN_l},W_q^2(\BR_+^N)^N))$,
such that for any $\Bg\in W_q^1(\BR_+^N)^N$ and for any $h\in W_q^2(\BR_+^N)$
\begin{align*}
\SST_N^1(\la)\Bg=\wt\SST_N^1(\la)(\nabla \Bg,\la^{1/2}\Bg),\quad
\SST_N^2(\la)h = \wt\SST_N^2(\la)(\nabla^2 h,\la^{1/2}\nabla h,\la h),
\end{align*}
and also for $n=0,1$
\begin{align*}
\CR_{\CL(L_q(\BR_+^N)^{\CN_l},\FB_q(\BR_+^N))}
\big(\big\{\la^n(d/d\la)^n\big(\CT_\la\wt\SST_j^l(\la)\big) \mid 
\la\in\Si_{\si,0} \big\} \big) \leq C,
\end{align*}
with some positive constant $C=C_{N,q,\si,\mu_*,\nu_*,\ka_*}$

Summing up the above results and setting for $\BU=(U_1,U_2,U_3,U_4,U_5)\in\wt\FX_q(\BR_+^N)$
\begin{align*}
\CA_2(\la)\BU&=\wt\SSS^1(\la)(U_1,U_2)+\wt\SSS^2(\la)(U_3,U_4,U_5), \\
\CB_2(\la)\BU&=(\wt \SST_1^1(\la)(U_1,U_2)+\wt\SST_1^2(\la)(U_3,U_4,U_5), \\
&\qquad\dots,\wt\SST_N^1(\la)(U_1,U_2)+\wt\SST_N^2(\la)(U_3,U_4,U_5))^\SST,
\end{align*}
we have completed the proof of Theorem \ref{theo:half}.

\section{Proof of Theorem \ref{theo2}}\label{sec:proof2}
In this section, we prove Theorem \ref{theo2} by using Theorem \ref{theo1}.
Let $\de=1/2$ and $\si\in(\si_*,\pi/2)$,
and let $c_0$ be a positive constant defined as $c_0=c(\de,\si)$
for the positive constant $c(\de,\si)$ given by Theorem \ref{theo1}.
We then observe by Definition \ref{defi:R} that for any $\la_0\geq 1$ and for $n=0,1$
\begin{align}\label{170419_1}
\CR_{\CL(\FX_q(\BR_+^N),\FA_q(\BR_+^N))}
\left(\left\{\la^n (d/d\la)^n\left(\CS_\la\CA_0(\la)\right)\mid \la\in\Si_{\si,\la_0}\right\}\right)
&\leq c_0, \\
\CR_{\CL(\FX_q(\BR_+^N),\FB_q(\BR_+^N))}
\left(\left\{\la^n(d/d\la)^n\left(\CT_\la\CB_0(\la)\right)\mid \la\in\Si_{\si,\la_0}\right\}\right)
&\leq c_0.
\notag
\end{align}

Let $\BF=(d,\Bf,\Bg,h)\in X_q(\BR_+^N)$ and $(\rho,\Bu)=(\CA_0(\la)\CF_\la\BF,\CB_0(\la)\CF_\la\BF)$.
Then $(\rho,\Bu)$ satisfies the following system:
\begin{equation}\label{170421_1}
\left\{\begin{aligned}
\la\rho+\di\Bu &= d  && \text{in $\BR_+^N$,} \\
\la\Bu-\mu_*\De\Bu-\nu_*\nabla \di\Bu+(\ga_*-\ka_*\De)\nabla \rho 
&=\Bf+\ga_*\nabla\CA_0(\la)\CF_\la\BF && \text{in $\BR_+^N$,} \\
\{\mu_*\BD(\Bu)+(\nu_*-\mu_*)\di\Bu\BI -(\ga_*-\ka_*\De)\rho\BI\}\Bn 
&=\Bg-\ga_*(\CA_0(\la)\CF_\la\BF)\Bn && \text{on $\BR_0^N$,} \\
\Bn\cdot\nabla\rho &= h && \text{on $\BR_0^N$.}
\end{aligned}\right.
\end{equation}

Now we set $\CG(\la)\BF$ and $X_{q,\la}(\BR_+^N)$ as
\begin{align*}
&\CG(\la)\BF
= (0, -\ga_*\nabla\CA_0(\la)\CF_\la\BF,\ga_*(\CA_0(\la)\CF_\la\BF)\Bn,0), \\
&X_{q,\la}(\BR_+^N) 
=\{\CF_\la\BF\in\FX_q(\BR_+^N) \mid \BF\in X_q(\BR_+^N)\} \quad (\la\in\Si_{\si,\la_0}).
\end{align*}
Note that $\CF_\la$ is a bijection from $X_q(\BR_+^N)$ onto $X_{q,\la}(\BR_+^N)$.
In view of \eqref{170421_1},
if there is the inverse operator $(I-\CG(\la))^{-1}\in\CL(X_q(\BR_+^N))$ of $I-\CG(\la)$,
then we see that
\begin{equation}\label{170419_4}
(\rho,\Bu)
=(\CA_0(\la)\CF_\la(I-\CG(\la))^{-1}\BF,\CB_0(\la)\CF_\la(I-\CG(\la))^{-1}\BF)
\end{equation}
solves the system \eqref{eq:rescale}.
It thus suffices to show the invertibility of $I-\CG(\la)$
and the $\CR$-boundedness of the inverse operator in what follows.

Here we define a further operator $\CH(\la)$ by
\begin{equation*}
\CH(\la)\BU 
= (0, \ga_*\nabla\CA_0(\la)\BU,-(\ga_*\CA_0(\la)\BU)\Bn,0)
\quad \text{for $\BU\in\FX_q(\BR_+^N)$,}
\end{equation*}
and introduce the following lemma proved essentially in \cite[Remark 3.2 (4)]{DHP03}
(cf. also \cite[page 84]{KW04} for Khinchine's inequality).

\begin{lemm}\label{lemm:multi}
Let $1\leq q <\infty$, and let $m(\la)$ be a bounded function defined on a subset $\La$ of $\BC$.
Assume that $M_m(\la)$ is a multiplication operator with $m(\la)$
defined by $M_m(\la)f=m(\la)f$ for any $f\in L_q(G)$ with an open set $G$ of $\BR^N$.
Then $\CR_{\CL(L_q(G))}(\{M_m(\la) \mid \la\in\La\})\leq K_q\|m\|_{L_\infty(\La)}$
for some positive constant $K_q\geq 1$ depending only on $q$.
\end{lemm}

By \eqref{170419_1} and Lemma \ref{lemm:multi}, 
we have for $n=0,1$
\begin{align}\label{170305_3}
&\CR_{\CL(\FX_q(\BR_+^N))}
\left(\left\{\la^n(d/d\la)^n\left(\CF_\la\CH(\la)\right) \mid
\la\in \Si_{\si,\la_0}\right\}\right) \\
&\leq c_0 |\ga_*| K_q(2\la_0^{-1}+\la_0^{-1/2}). 
\notag
\end{align}
We here choose $\la_0$ large enough so that $c_0|\ga_*|K_q(2\la_0^{-1}+\la_0^{-1/2})\leq 1/2$.
Since $\CH(\la)\CF_\la\BF = \CG(\la)\BF$, the estimate \eqref{170305_3} with $n=0$ implies that
\begin{align*}
\|\CF_\la\CG(\la)\BF\|_{\FX_q(\BR_+^N)}
=\|\CF_\la\CH(\la)\CF_\la\BF\|_{\FX_q(\BR_+^N)}\leq \frac{1}{2}\|\CF_\la\BF\|_{\FX_q(\BR_+^N)}. 
\end{align*}
Thus $\CF_\la\CG(\la)\CF_\la^{-1} \in \CL(X_{q,\la}(\BR_+^N))$ with
$\|\CF_\la\CG(\la)\CF_\la^{-1}\|_{\CL(X_{q,\la}(\BR_+^N))}\leq 1/2$,
which, combined with the Neumann series expansion,
furnishes that $(I-\CF_\la\CG(\la)\CF_\la^{-1})^{-1}$ exists in $\CL(X_{q,\la}(\BR_+^N))$.
Hence, $\CF_\la^{-1}(I-\CF_\la\CG(\la)\CF_\la^{-1})^{-1}\CF_\la=(I-\CG(\la))^{-1}$
exists in $\CL(X_q(\BR_+^N))$,
and also $\CF_\la(I-\CG(\la))^{-1}\BF=(I-\CF_\la\CH(\la))^{-1}\CF_\la\BF$
by $\CG(\la)\CF_\la^{-1}=\CH(\la)$.
We set for $\BU\in\FX_q(\BR_+^N)$
\begin{equation*}
\CA(\la)\BU= \CA_0(\la)(I-\CF_\la\CH(\la))^{-1}\BU, \quad
\CB(\la)\BU=\CB_0(\la)(I-\CF_\la\CH(\la))^{-1}\BU,
\end{equation*}
and then $(\rho,\Bu)=(\CA(\la)\CF_\la\BF,\CB(\la)\CF_\la\BF)$ solves \eqref{eq:rescale}
as was discussed in \eqref{170419_4}.

Next we prove the $\CR$-boundedness of $\CA(\la)$, $\CB(\la)$ constructed above.
Noting that $c_0 |\ga_*| K_q(2\la_0^{-1}+\la_0^{-1/2})\leq 1/2$,
we have, by \eqref{170305_3}, the Neumann series expansion, and Definition \ref{defi:R},
\begin{align*}
&(I-\CF_\la\CH(\la))^{-1}\in \Hol(\Si_{\si,\la_0},\CL(\FX_q(\BR_+^N))), \\
&\CR_{\CL(\FX_q(\BR_+^N))}
\left(\left\{\la^n(d/d\la)^n(I-\CF_\la\CH(\la))^{-1}
\mid \la\in\Si_{\si,\la_0}\right\}\right)\leq 2.
\notag
\end{align*}
Combining these properties with \eqref{170419_1} and Remark \ref{rema:half} \eqref{rema:half2}
furnishes that
\begin{align*}
\CR_{\CL(\FX_q(\BR_+^N),\FA_q(\BR_+^N))}
\left(\left\{\la^n (d/d\la)^n\left(\CS_\la\CA(\la)\right)\mid \la\in\Si_{\si,\la_0}\right\}\right)
&\leq 4c_0, \\
\CR_{\CL(\FX_q(\BR_+^N),\FB_q(\BR_+^N))}
\left(\left\{\la^n (d/d\la)^n\left(\CT_\la\CB(\la)\right)\mid \la\in\Si_{\si,\la_0}\right\}\right)
&\leq 4c_0.
\end{align*}

Finally, we prove the uniqueness of the system \eqref{eq:rescale}.
Let $(\rho,\Bu)\in W_q^3(\BR_+^N)\times W_q^2(\BR_+^N)^N$ satisfy 
\eqref{eq:rescale} whose right-hand side is zero.
Then $(\rho,\Bu)$ satisfies \eqref{eq:full}
with $d=0$, $\Bf=-\ga_*\nabla\rho$, $\Bg=\ga_*\rho\Bn$, and $h=0$.
We thus have by \eqref{170419_1} with $n=0$
\begin{align*}
&\|\CS_\la\rho\|_{\FA_q(\BR_+^N)} 
\leq c_0\|(\Bf,\nabla\Bg,\la^{1/2}\Bg)\|_{L_q(\BR_+^N)} \\
&\leq c_0|\ga_*|(2\la_0^{-1}+\la_0^{-1/2})\|\CS_\la\rho\|_{\FA_q(\BR_+^N)}
\leq \frac{1}{2}\|\CS_\la\rho\|_{\FA_q(\BR_+^N)}
\end{align*}
for any $\la\in\Si_{\si,\la_0}$, 
which implies that $\rho=0$.
In addition, by \eqref{170419_1} with $n=0$ again,
\begin{equation*}
\|\CT_\la\Bu\|_{\FB_q(\BR_+^N)}
\leq c_0\|(\Bf,\nabla\Bg,\la^{1/2}\Bg)\|_{L_q(\BR_+^N)}=0.
\end{equation*}
Hence $\Bu=0$.
This completes the proof of Theorem \ref{theo2}.

\bigskip
\noindent{\bf Acknowledgments.}
The author is greatly indebted to Professor Yoshihiro Shibata
for suggesting the problem and for many stimulating conversations.
This research was partly supported by JSPS Grant-in-aid for Young Scientists (B) \#17K14224,
JSPS Japanese-German Graduate Externship at Waseda University,
and Waseda University Grant for Special Research Projects (Project number: 2017K-176).


\begin{thebibliography}{99}

\bibitem{DD01}
R.~Danchin and B.~Desjardins.
\newblock Existence of solutions for compressible fluid models of {K}orteweg
  type.
\newblock {\em Ann. Inst. H. Poincar\'e Anal. Non Lin\'eaire}, 18(1):97--133,
  2001.

\bibitem{DHP03}
R.~Denk, M.~Hieber, and J.~Pr\"uss.
\newblock $\mathcal{R}$-boundedness, {F}ourier multipliers and problems of
  elliptic and parabolic type.
\newblock {\em Mem. Amer. Math. Soc.}, 166(788):viii+114 pp., 2003.

\bibitem{DS85}
J.~E. Dunn and J.~Serrin.
\newblock On the thermomechanics of interstitial working.
\newblock {\em Arch. Rational Mech. Anal.}, 88(2):95--133, 1985.

\bibitem{ES13}
Y.~Enomoto and Y.~Shibata.
\newblock On the $\mathcal{R}$-sectoriality and the initial boundary value
  problem for the viscous compressible fluid flow.
\newblock {\em Funkcial. Ekvac.}, 56(3):441--505, 2013.

\bibitem{Haspot11}
B.~Haspot.
\newblock Existence of global weak solution for compressible fluid models of
  {K}orteweg type.
\newblock {\em J. Math. Fluid Mech.}, 13(2):223--249, 2011.

\bibitem{HL96}
H.~Hattori and D.~Li.
\newblock Global solutions of a high-dimensional system for {K}orteweg
  materials.
\newblock {\em J. Math. Anal. Appl.}, 198(1):84--97, 1996.

\bibitem{Kotschote08}
M.~Kotschote.
\newblock Strong solutions for a compressible fluid model of {K}orteweg type.
\newblock {\em Ann. Inst. H. Poincar\'e Anal. Non Lin\'eaire}, 25(4):679--696,
  2008.

\bibitem{Kotschote10}
M.~Kotschote.
\newblock Strong well-posedness for a {K}orteweg-type model for the dynamics of
  a compressible non-isothermal fluid.
\newblock {\em J. Math. Fluid Mech.}, 12(4):473--484, 2010.

\bibitem{Kotschote12}
M.~Kotschote.
\newblock Dynamics of compressible non-isothermal fluids of non-{N}ewtonian
  {K}orteweg type.
\newblock {\em SIAM J. Math. Anal.}, 44(1):74--101, 2012.

\bibitem{Kotschote14}
M.~Kotschote.
\newblock Existence and time-asymptotics of global strong solutions to dynamic
  {K}orteweg models.
\newblock {\em Indiana Univ. Math. J.}, 63(1):21--51, 2014.

\bibitem{KW04}
P.C. Kunstmann and L.~Weis.
\newblock Maximal ${L}_{p}$-regularity for parabolic equations, {F}ourier
  multiplier theorems and ${H}^\infty$-functional calculus.
\newblock In {\em Functional Analytic Methods for Evolution Equations}, volume
  1855 of {\em Lect. Notes in Math.}, pages 65--311. Springer, Berlin, 2004.

\bibitem{MS17}
S.~Maryani and H.~Saito.
\newblock On the $\mathcal{R}$-boundedness of solution operator families for two-phase Stokes resolvent equations.
\newblock {\em Differential Integral Equations}, 30(1-2):1--52, 2017.

\bibitem{Saito15b}
H.~Saito.
\newblock On the $\mathcal{R}$-boundedness of solution operator families of the
  generalized {S}tokes resolvent problem in an infinite layer.
\newblock {\em Math. Methods Appl. Sci.}, 38(9):1888--1925, 2015.

\bibitem{SS01}
Y.~Shibata and S.~Shimizu.
\newblock A decay property of the {F}ourier transform and its application to
  the {S}tokes problem.
\newblock {\em J. Math. Fluid Mech.}, 3(3):213--230, 2001.

\bibitem{SS12}
Y.~Shibata and S.~Shimizu.
\newblock On the maximal ${L}_p\text{-}{L}_q$ regularity of the {S}tokes
  problem with first order boundary condition; model problems.
\newblock {\em J. Math. Soc. Japan}, 64(2):561--626, 2012.

\bibitem{Weis01}
L.~Weis.
\newblock Operator-valued {F}ourier multiplier theorems and maximal
  ${L}_p$-regularity.
\newblock {\em Math. Ann.}, 319(4):735--758, 2001.

\end{thebibliography}


\end{document}